\documentclass[12pt]{article}
\usepackage{mathrsfs,amsthm,graphicx,color,verbatim,bbm,amsmath,amsfonts,amssymb,lmodern,newclude,nicefrac,amsfonts,graphicx,color,enumerate,hyperref,bm}
\usepackage[capitalize]{cleveref}
\usepackage{todonotes}
\usepackage[T1]{fontenc}
\usepackage{lineno}
\usepackage{geometry}
\theoremstyle{plain}
\newtheorem{theorem}{Theorem}[section]
\newtheorem{lemma}[theorem]{Lemma}
\newtheorem{prop}[theorem]{Proposition}
\newtheorem{cor}[theorem]{Corollary}

\newtheorem{setting}[theorem]{Setting}

\newcommand{\E}{\mathbb{E}}
\newcommand{\cF}{\mathcal{F}}
\renewcommand{\P}{\mathbb{P}}

\newcommand{\R}{\mathbb{R}}
\newcommand{\Rd}{\mathbb{R}^d}
\newcommand{\N}{\mathbb{N}}
\newcommand{\Z}{\mathbb{Z}}
\newcommand{\smallsum}{\textstyle\sum}

\newcommand{\Exp}[1]{ \E \! \left[ #1 \right]}

\newcommand{\EXP}[1]{ \E  [ #1 ]}

\newcommand{\norm}[1]{ \left\| #1 \right\| }
\newcommand{\Norm}[1]{ \| #1 \| }

\makeatletter
\newcommand{\vast}{\bBigg@{3.5}}
\newcommand{\Vast}{\bBigg@{4}}
\makeatother

\newcommand{\fr}{{\bf f}_r}
\newcommand{\fdr}{{\bf f}_{d,r}}

\newcommand{\partialX}[1]{\tfrac{\partial}{\partial{x_{#1}}}}
\newcommand{\partialXX}[1]{\tfrac{\partial^2}{\partial{x_{#1}}^2}}
\newcommand{\vertiii}[1]{{\left\vert\kern-0.25ex\left\vert\kern-0.25ex\left\vert #1 
		\right\vert\kern-0.25ex\right\vert\kern-0.25ex\right\vert}}
	\newcommand{\tripleNorm}[1]{\vertiii#1} 

\begin{document}

\title{Overcoming the curse of dimensionality \\
in the numerical approximation of \\
Allen--Cahn partial differential equations \\
via truncated full-history recursive \\
multilevel Picard approximations}

\author{
	Christian Beck$^1$, 
	Fabian Hornung$^{2,3}$,
	Martin Hutzenthaler$^4$, \\
	Arnulf Jentzen$^5$, 
	and Thomas Kruse$^6$ 
	\bigskip
	\\
	\small{$^1$ Department of Mathematics, ETH Zurich, Z\"urich,}\\
	\small{Switzerland, e-mail: christian.beck@math.ethz.ch} 
	\\
	\small{$^2$ Department of Mathematics, ETH Zurich, Z\"urich,}\\
	\small{Switzerland, e-mail: fabianhornung89@gmail.com}
	\\
	\small{$^3$ Faculty of Mathematics, Karlsruhe Institute of Technology, Karlsruhe,}\\
	\small{Germany, e-mail: fabianhornung89@gmail.com} 
	\\
	\small{$^4$ Faculty of Mathematics, University of Duisburg-Essen, Essen, }\\
	\small{Germany, e-mail: martin.hutzenthaler@uni-due.de}
	\\
	\small{$^5$ Department of Mathematics, ETH Zurich, Z\"urich,}\\
	\small{Switzerland, e-mail: arnulf.jentzen@sam.math.ethz.ch} 
	\\
	\small{$^6$ Institute of Mathematics, University of Gie{\ss}en, Gie{\ss}en,} \\
	\small{Germany, e-mail: thomas.kruse@math.uni-giessen.de}
}

\maketitle

\begin{abstract}
One of the most challenging problems in applied mathematics is the approximate solution of nonlinear partial differential equations (PDEs) in high dimensions. Standard deterministic approximation methods like finite differences or finite elements suffer from the curse of dimensionality in the sense that the computational effort grows exponentially in the dimension. 
In this work we overcome this difficulty in the case of reaction-diffusion type PDEs with a locally Lipschitz continuous coervice nonlinearity (such as Allen--Cahn PDEs) by introducing and analyzing truncated variants of the recently introduced full-history recursive multilevel Picard approximation schemes.  
\end{abstract}
  
 \pagebreak
  
\tableofcontents

\section{Introduction}
\label{sect:intro}
One of the most challenging problems in applied mathematics is the approximate solution of nonlinear partial differential equations (PDEs) in high dimensions. Standard deterministic approximation methods like finite differences or finite elements suffer from the curse of dimensionality in the sense that the computational effort grows exponentially in the dimension. Linear parabolic PDEs of second order can be solved approximately without the curse of dimensionality by means of Monte Carlo averages. In the last few years, several probabilistic approximation methods, which seem in certain situations to be capable of efficiently approximating high-dimensional nonlinear PDEs, have been proposed. For instance, the articles \cite{BouchardEtAl2017NumericalApproximation,Henry-Labordere2012counterparty,Henry-LabordereOudjaneTanTouziWarin2019BranchingDiffusion,Henry-LabordereTanTouzi2014NumericalAlgorithmBSDEBranching} propose and study approximation methods based on stochastic representations of solutions of PDEs by means of branching diffusion processes (cf., for example, \cite{McKean1975ApplicationBrownianMotion,Skorokhod1964BranchingDiffusion,Watanabe1965BranchingProcess} for theoretical relations and cf., for example, \cite{Warin2018NestingMonteCarlo} for a related method), 
the articles \cite{DeepSplitting,DeepKolmogorov,BeckEJentzen2017MachineLearning,becker2018deep,berg2018unified,chan2018machine,EHanJentzen2017DeepLearning, EYu17,FarahmandNabiNikovski17, FujiiTakahashiTakahashi17,goudenege2019machine,HanEJentzen2017SolvingHighdimensionalPDEs, han2018convergence,	HenryLabordere17,hure2019some,jacquier2019deep,LongLuMaDong17,lye2019deep,magill2018neural,Raissi18,SirignanoDGM2017} propose and study approximation methods based on the reformulation of PDEs as stochastic learning problems involving deep artificial neural networks, and the articles  \cite{hutzenthaler2016multilevel,EHutzenthalerJentzenKruse2019MLP,hutzenthaler2019proof,Overcoming,hutzenthaler2019overcoming,hutzenthaler2017multi} propose and study full-history recursive multilevel Picard (MLP) approximation methods. In particular, the articles \cite{Overcoming,hutzenthaler2019overcoming} prove that MLP approximation schemes do indeed overcome the curse of dimensionality in the numerical approximation of semilinear parabolic PDEs. More formally, Theorem 3.8 in \cite{Overcoming} shows that MLP approximation schemes are able to approximate the solutions of semilinear parabolic PDEs with a root mean square error of size $\varepsilon \in (0,\infty)$ and a computational effort which grows at most polynomially both in the dimension as well as in the reciprocal $\nicefrac{1}{\varepsilon}$ of the desired approximation accuracy. However, the articles \cite{Overcoming,hutzenthaler2019overcoming} are only applicable in the case where the nonlinearity is globally Lipschitz continuous and, to the best of our knowledge, there exists no result in the scientific literature which shows for every $T\in (0,\infty)$ that the solution of a semilinear parabolic PDE with a non-globally Lipschitz continuous nonlinearity can be efficiently approximated at time $T$ without the curse of dimensionality. 

In this work we overcome this difficulty by introducing a truncated variant of the MLP approximation schemes introduced in \cite{hutzenthaler2016multilevel,Overcoming} and by proving that this truncated MLP approximation scheme succeeds in approximately solving reaction-diffusion type PDEs with a locally Lipschitz continuous coercive nonlinearity (such as Allen--Cahn type PDEs) without the curse of dimensionality. 
More specifically, \cref{thm:main_theorem_forward_formulation} in \cref{sec:algorithm} below, which is the main result of this article, proves under suitable assumptions that for every $\delta\in (0,\infty)$, $\varepsilon\in (0,1]$ it holds that the proposed truncated MLP approximations can achieve a root mean square error of size at most $\varepsilon$ with a computational effort of order $d\varepsilon^{-(2+\delta)}$. To illustrate the findings of this article in more detail, we now present in \cref{intro:theorem} below a special case of  \cref{thm:main_theorem_forward_formulation}.

\begin{theorem}\label{intro:theorem}
Let
	$\delta,\kappa,T \in (0,\infty)$,
	$\Theta = \cup_{n\in\N} \Z^n$, 
	$f \in C^1(\R,\R)$, 
	$(\mathbf{f}_d)_{d\in\N} \subseteq C(\R,\R)$, 
	$(u_d)_{d\in\N} \subseteq C([0,T]\times\R^d,\R)$, 
assume that 
	$f^{\prime}$ 
is at most polynomially growing, 
assume for every 
	$d\in\N$,
	$t\in (0,T]$,
	$x\in\R^d$, 
	$v\in\R$
that 
	$vf(v)\leq \kappa(1+v^2)$,
	$|u_d(0,x)|\leq \kappa$,
	$u_d|_{(0,T]\times\R^d}\in C^{1,2}((0,T]\times\R^d,\R)$, 
	$\inf_{c\in\R}(\sup_{s\in [0,T]} \sup_{y=(y_1,\ldots,y_d)\in\R^d} (e^{c(|y_1|^2+\ldots+|y_d|^2)}|u_d(s,y)|)) < \infty$,
	$\mathbf{f}_d(v) = f(\min\{\ln(1+\ln(d)),\max\{-\ln(1+\ln(d)),v\}\})$,
and 
	\begin{equation}\label{intro:allen_cahn_pde}
	(\tfrac{\partial}{\partial t}u_d)(t,x) 
	= 
	(\Delta_x u_d)(t,x) 
	+ 
	f(u_d(t,x)),
	\end{equation} 
let 
	$(\Omega,\mathcal{F},\P)$ 
be a probability space, 
let
	$\mathcal{R}^{\theta}\colon \Omega \to [0,1]$, $\theta\in\Theta$, 
be independent $\mathcal{U}_{[0,1]}$-distributed random variables, 
let 
	$W^{d,\theta}\colon [0,T]\times\Omega\to\R^d$, 
	$d\in\N$,
	$\theta\in\Theta$,
be independent standard Brownian motions, 
assume that 
	$(\mathcal{R}^{\theta})_{\theta\in\Theta}$ 
and
	$(W^{d,\theta})_{(d,\theta)\in\N\times\Theta}$ 
are independent, 
let 
	$R^{\theta}\colon [0,T]\times\Omega \to [0,T]$, $\theta\in\Theta$,  
satisfy for every 
	$\theta\in\Theta$,
	$t\in [0,T]$ 
that 
	$R^{\theta}_t = t \mathcal{R}^{\theta}$,
for every 
	$d\in \N$,
	$s\in [0,T]$, 
	$t\in [s,T]$,
	$x\in \R^d$, 
	$\theta\in \Theta$ 
let 
	$X^{d,\theta}_{s,t,x}\colon\Omega\to\R^d$
satisfy 
	$
	X^{d,\theta}_{s,t,x}
	= 
	x + \sqrt{2}(W^{d,\theta}_t - W^{d,\theta}_s)
	$, 
let 
	$U^{d,\theta}_{n,M}\colon [0,T]\times\R^d\times\Omega\to\R$, 
	$d,M\in\N$,
	$\theta\in\Theta$, $n\in\N_0$, 
satisfy for every 
	$d,n,M\in\N$,
	$\theta\in\Theta$,
	$t\in [0,T]$, 
	$x\in \R^d$ 
that
	$U^{d,\theta}_{0,M}(t,x) = 0$ 
and 
	\begin{equation}\label{intro:mlp}
	\begin{split}
	& 
	U^{d,\theta}_{n,M}(t,x) 
	= 
	\sum_{k=1}^{n-1} \frac{t}{M^{n-k}} 
	\Bigg[ 
	\sum_{m=1}^{M^{n-k}} 
	\bigg(
	\mathbf{f}_{M}\Big( 
	U^{d,(\theta,k,m)}_{k,M}\big(R^{(\theta,k,m)}_{t}, 
	X^{d,(\theta,k,m)}_{R^{(\theta,k,m)}_{t},t,x}
	\big)
	\Big)
	\\ 
	& 
	-
	\mathbf{f}_{M}\Big( 
	U^{d,(\theta,-k,m)}_{k-1,M}\big( R^{(\theta,k,m)}_{t}, 
	X^{d,(\theta,k,m)}_{R^{(\theta,k,m)}_{t},t,x}\big)
	\Big)
	\bigg)
	\Bigg]
	+ 
	\frac{1}{M^n}\!\left[ 
	\sum_{m=1}^{M^n} 
	\left( 
	u_d(0,X^{d,(\theta,0,-m)}_{0,t,x}) 
	+ 
	t \, f(0)
	\right)
	\right], 
	\end{split}
	\end{equation} 
and for every 
	$d,M\in\N$, 
	$n\in\N_0$
let 
	$\mathfrak{C}_{d,n,M}\in\N_0$
be the number of realizations of scalar standard normal random variables which are used to compute one realization of 
	$U^{d,0}_{n,M}(T,0) \colon \Omega \to \R$ 
(cf.~\cref{cor:polynomially_growing} for a precise definition). 
Then there exist
	$\mathfrak{N} \colon (0,1] \to \N$ 
and 
	$c\in\R$ 
such that for every 
	$d\in\N$, 
	$\varepsilon\in (0,1]$ 
it holds that 
	$
	%\sup_{n\in [1,\mathfrak{N}_{\varepsilon}]\cap\N} 
	\mathfrak{C}_{d,\mathfrak{N}_{\varepsilon},\mathfrak{N}_{\varepsilon}} 
	\leq c d \varepsilon^{-(2+\delta)}
	$
and 
	\begin{equation}
	%\sup_{n\in [\mathfrak{N}_{\varepsilon},\infty)\cap\N} 
	%\left[ 
	\sup_{x\in\R^d}
	\left(  
	\Exp{
		\big|
		U^{d,0}_{\mathfrak{N}_{\varepsilon},\mathfrak{N}_{\varepsilon}}(T,x)
		-
		u_d(T,x)
		\big|^2}  
	\right)^{\!\nicefrac12} 
	%\right]
	\leq \varepsilon. 
	\end{equation}
\end{theorem} 
\cref{intro:theorem} above is an immediate consequence of \cref{cor:polynomially_growing} in \cref{sec:applications} below. \cref{cor:polynomially_growing} follows from \cref{cor_lip}   which, in turn, is deduced from \cref{thm:main_theorem_forward_formulation}, the main result of this article. 
\cref{intro:theorem} establishes under suitable assumptions that for every 
	$\delta\in (0,\infty)$ 
there exists 
	$c \in (0,\infty)$ 
such that for every $d\in\N$ the solution $u_d\colon 
[0,T]\times\R^d\to\R$ of the reaction-diffusion type partial differential equation in \eqref{intro:allen_cahn_pde} can be approximated by the MLP approximation scheme in \eqref{intro:mlp} with a root mean square error of size $\varepsilon \in (0,\infty)$ while the computational effort is bounded by $c d 
\varepsilon^{-(2+\delta)}$. The numbers $\mathfrak{C}_{d,n,M}$, $d,M\in\N$, $n\in\N_0$, in \cref{intro:theorem} model the computational effort. The nonlinearity $f\colon\R\to\R$ in \cref{intro:theorem} is required to be locally Lipschitz continuous (which follows from the hypothesis in \cref{intro:theorem} that $f^{\prime}$ is continuous) and to satisfy a coercivity type condition in the sense that there exists $\kappa\in\R$ such that for all $v\in\R$ it holds that $vf(v)\leq \kappa(1+v^2)$. 
This coercivity type condition together with the growth assumption on the solutions $u_d\colon 
[0,T]\times\R^d\to\R$, $d\in\N$, allows us to deduce in 
\cref{sec:maximum_principle} that the solutions $u_d\colon [0,T]\times\R^d\to\R$, $d\in\N$, are uniformly bounded. 
In particular, \cref{cor:backwardReactionDiffusionEquationWithCoerciveNonlinearity} in \cref{sec:maximum_principle} yields that 
there exists $\mathfrak{M}\in\N$ such that for every 	
	$M\in [\mathfrak{M},\infty)\cap\N$,  
	$d\in \N$,
	$t\in [0,T]$, 
	$x\in \R^d$ 
it holds that 
	$(\tfrac{\partial}{\partial t}u_d)(t,x) 
	= 
	(\Delta_x u_d)(t,x) + \mathbf{f}_M(u_d(t,x))$. 
The fact that for every $d,M\in\N$ it holds that $(\tfrac{\partial}{\partial t}u_d)(t,x) 
= 
(\Delta_x u_d)(t,x) + \mathbf{f}_M(u_d(t,x))$, $(t,x)\in [0,T]\times\R^d$, is a parabolic PDE with a globally Lipschitz continuous 
nonlinearity then permits us to bring the machinery from \cite{Overcoming} into play. This will finally allow us to prove \cref{intro:theorem} (see \cref{sec:maximum_principle,sec:algorithm} for details). We note that although \cref{intro:theorem} uses the assumption that the nonlinearity $f\colon\R\to\R$ satisfies the coercivity type condition that there exists $\kappa\in\R$ such that for all 
	$v\in\R$ 
it holds that 
	$vf(v) \leq \kappa (1 + v^2)$, 
explicit knowledge of the coercivity constant $\kappa$ is not required for the implementation of the MLP approximation scheme.  

The remainder of this article is organized as follows. In \cref{sec:maximum_principle} we present elementary a priori bounds for classical solutions of reaction-diffusion type PDEs with coercive nonlinearities. 
In \cref{sec:algorithm} we introduce truncated MLP  approximation schemes and we provide upper bounds for the root mean square distance between the truncated MLP approximations and the exact solution of the PDE under consideration. In \cref{sec:complexity} we combine the error estimates from \cref{sec:algorithm} with estimates for the computational effort for truncated MLP approximations to show under suitable assumptions that for every $\delta\in (0,\infty)$ a root mean square error of size $\varepsilon\in (0,1]$ can be achieved by truncated MLP approximations with a computational effort of order $d\varepsilon^{-(2+\delta)}$. In \cref{sec:applications} we specialize our findings to Allen--Cahn type PDEs. 

\section{A priori bounds for reaction-diffusion equations with coercive nonlinearity}
\label{sec:maximum_principle}

For convenience of the reader, we recall the following well-known maximum principle for subsolutions of the heat equation (cf., e.g., John \cite[Pages 216--217 in Section 1 in Chapter 7]{John}).

%\textbf{TODO: Look up Hairer, Hutzenthaler, Jentzen, Section 4 for another reference!!}

\begin{lemma}\label{lem:maxPrincipleForHeatEquation}
	Let $d\in\N$, $T\in (0,\infty)$, $v\in C([0,T]\times\R^d,\R),$ assume that $v|_{(0,T]\times\R^d} \in C^{1,2}((0,T]\times\R^d,\R),$
	assume for every $t\in (0,T],$ $x\in\R^d$ that
	\begin{equation}\label{eq:subheat}
	(\tfrac{\partial }{\partial t}v)(t,x) \leq (\Delta_x v)(t,x), 
	\end{equation}
	let $\norm{\cdot}\colon \R^d\to [0,\infty)$ be the d-dimensional Euclidean norm,
	and assume that
	\begin{equation}\label{eq:assumptionExponentialGrowth}
	\inf_{a\in \R} \sup_{(t,x)\in [0,T]\times\R^d} (e^{a\|x\|^2}v(t,x)) < \infty.
	\end{equation}
	%  which satisfies 
	Then it holds that 
	\begin{equation}\label{eq:maxPrincipleForHeatEquationInequalityToProve}
	\sup_{(t,x)\in [0,T]\times\R^d} v(t,x) = \sup_{x\in\R^d} v(0,x) . 
	\end{equation}
\end{lemma}

\begin{proof}[Proof of Lemma~\ref{lem:maxPrincipleForHeatEquation}]
	%	Observe that \eqref{eq:maxPrincipleForHeatEquationInequalityToProve} 
	%	follows immediately in the case $\sup_{y\in\R^d} v(0,y)=\infty$. 
	%	W.l.o.g. we may therefore assume that $\sup_{y\in\R^d} v(0,y) < \infty$. 
	% In the case $\sup_{x\in\R^d} v(0,x)=\infty,$ \eqref{eq:maxPrincipleForHeatEquationInequalityToProve} follows immediately.
	%W.l.o.g.\ we may therefore assume that $\sup_{x\in\R^d} v(0,x) < \infty,$
Throughout this proof assume w.l.o.g. that 
	\begin{equation}
	\sup_{x\in\R^d} v(0,x) < \infty,
	\end{equation}
let $\Phi_{\varepsilon}\colon [0,T]\times\R^d\to\R,$  $\varepsilon \in (0,\infty),$	be the functions which satisfy for every $\varepsilon \in (0,\infty),$ $t\in [0,T]$, $x\in\R^d$ that 
	\begin{equation}\label{eq:definition_of_Phi}
	\Phi_{\varepsilon}(t,x) 
	%	= 
	%	\frac{1}{[4\pi(T+\varepsilon-t)]^{d/2}} 
	%	\exp\!\left( \frac{\|x\|^2}{4(T+\varepsilon-t)} \right), 
	= 
	[4\pi(T+\varepsilon-t)]^{-d/2} 
	\exp\!\left( \frac{\|x\|^2}{4(T+\varepsilon-t)} \right), 
	\end{equation}
and let $w_{\varepsilon,M}\colon [0,T]\times\R^d\to\R,$ $\varepsilon, M \in (0,\infty),$ 
be the functions which satisfy for every $\varepsilon, M \in (0,\infty),$ $t\in [0,T]$, $x\in\R^d$ that
	\begin{equation}\label{eq:definition_of_w_epsilon_M}
	w_{\varepsilon,M}(t,x) 
	= 
	v(t,x) 
	- 
	M\Phi_{\varepsilon}(t,x) 
	- 
	%   \frac{1}{[4\pi(T+\varepsilon-t)]^{d/2}} 
	% 	  \exp\!\left( \frac{\|x\|^2}{4(T+\varepsilon-t)} \right) 
	%   - 
	\varepsilon t. 
	\end{equation}
	%	Note that for every $u\in C^1(\R^d,\R),$ $v\in C^1(\R^d,\R^d),$ $x=(x_1,\dots,x_d)\in\R^d$ it holds that
	%	\begin{equation}\label{productRuleDivergence}
	%	 \begin{split}
	%	 		(\divergenceX(u v))(x)=&\sum_{i=1}^d (\partialX{i}(u v_i))(x)=\sum_{i=1}^d \left[(\partialX{i}u)(x)v_i(x)+u(x)(\partialX{i}v_i)(x)\right] \\
	%	 =&\left\langle (\nabla_x u)(x),v(x)\right\rangle_{\Rd}+u(x)(\divergenceX(v))(x).
	%	 \end{split}
	%	\end{equation}
Observe that for every 
	$\varepsilon\in (0,\infty)$, 
	$t\in [0,T]$, 
	$x\in\R^d$ 
it holds that 
	\begin{equation}\label{eq:gradient_of_Phi}
	(\nabla_x \Phi_{\varepsilon})(t,x) 
	= 
	\Phi_{\varepsilon}(t,x)\left[\frac{x}{2(T+\varepsilon-t)}\right]. 
	\end{equation}
This implies that for every $\varepsilon\in (0,\infty)$, $t\in [0,T]$, $x=(x_1,\dots,x_d)\in\R^d$ it holds that 
	\begin{equation}
	\begin{split}
	(\partialXX{k}\Phi_\varepsilon)(t,x)
	&=(\partialX{k}\Phi_\varepsilon)(t,x) \left[\frac{x_k}{2(T+\varepsilon-t)}\right]
	+\Phi_\varepsilon(t,x) \left[\frac{1}{2(T+\varepsilon-t)}\right]\\
	&=\Phi_\varepsilon(t,x) \left[\frac{x_k}{2(T+\varepsilon-t)}\right]^2 +\Phi_\varepsilon(t,x) \left[\frac{1}{2(T+\varepsilon-t)}\right]\\
	&=\Phi_\varepsilon(t,x) \left(\frac{\vert x_k\vert^2}{4(T+\varepsilon-t)^2}+\frac{1}{2(T+\varepsilon-t)}\right).
	\end{split}
	\end{equation}
	Therefore, we obtain that for every $\varepsilon\in (0,\infty)$, $t\in [0,T]$, $x=(x_1,\dots,x_d)\in\R^d$ it holds that
	\begin{equation}\label{eq:calculationLaplacian}
	\begin{split}
	(\Delta_x \Phi_{\varepsilon})(t,x)&=\sum_{k=1}^d (\partialXX{k}\Phi_\varepsilon)(t,x)\\
	&=\sum_{k=1}^d \left[ \Phi_\varepsilon(t,x) \left(\frac{\vert x_k\vert^2}{4(T+\varepsilon-t)^2}+\frac{1}{2(T+\varepsilon-t)}\right)\right]\\
	%		&=\Phi_\varepsilon(t,x)\left(\frac{1}{4(T+\varepsilon-t)^2}\left[\sum_{k=1}^dx_k^2\right]+\frac{d}{2(T+\varepsilon-t)}\right)\\
	&=\Phi_\varepsilon(t,x)\left(\frac{\norm{x}^2}{4(T+\varepsilon-t)^2}+\frac{d}{2(T+\varepsilon-t)}\right).
	\end{split}
	\end{equation}
	%	
	%	
	%	
	%	
	%	
	%	\begin{equation}\label{eq:calculationLaplacian}
	%	\begin{split}
	%	(\Delta_x \Phi_{\varepsilon})(t,x)&=(\divergenceX(\nabla_x \Phi_{\varepsilon}))(t,x) 
	%	= \left(\divergenceX\!\left(\Phi_{\varepsilon}(t,x) \left[\frac{x}{2(T+\varepsilon-t)}\right]\right)\right)(t,x)\\
	%	&=
	%	\left\langle (\nabla_x \Phi_{\varepsilon})(t,x), \frac{x}{2(T+\varepsilon-t)} \right\rangle_{\R^d} \\
	%	&\qquad+ \Phi_{\varepsilon}(t,x) \operatorname{div}_x\left(\frac{x}{2(T+\varepsilon-t)}\right)(t,x) \\
	%	& = 
	%	\left( \frac{\|x\|^2}{4(T+\varepsilon-t)^2} + 
	%	\frac{d}{2(T+\varepsilon-t)} \right) \Phi_{\varepsilon}(t,x) . 
	%	\end{split}
	%	\end{equation}
	Moreover, observe that for every $\varepsilon\in (0,\infty)$, $t\in [0,T]$, $x\in\R^d$ it holds
	that 
	\begin{equation}
	\begin{split}
	&(\tfrac{\partial}{\partial t}\Phi_{\varepsilon})(t,x) 
	=
	-\frac{d}{2} [4\pi(T+\varepsilon-t)]^{-d/2-1} [-4\pi] \exp\!\left( \frac{\|x\|^2}{4(T+\varepsilon-t)} \right)
	\\&+[4\pi(T+\varepsilon-t)]^{-d/2} \exp\!\left( \frac{\|x\|^2}{4(T+\varepsilon-t)} \right)\left[\frac{\|x\|^2}{4(T+\varepsilon-t)^2}\right]
	\\&= [4\pi(T+\varepsilon-t)]^{-d/2} \exp\!\left( \frac{\|x\|^2}{4(T+\varepsilon-t)} \right) \left[-\frac{d}{2}\left(\frac{-4\pi}{4\pi(T+\varepsilon-t)}\right)+\frac{\|x\|^2}{4(T+\varepsilon-t)^2}\right]
	\\&=		
	\left(\frac{d}{2(T+\varepsilon-t)} + \frac{\|x\|^2}{4(T+\varepsilon-t)^2}
	\right) \Phi_{\varepsilon}(t,x) . 
	\end{split}
	\end{equation}
	Combining this with \eqref{eq:calculationLaplacian} ensures that for every $\varepsilon\in (0,\infty)$, $t\in [0,T]$, $x\in\R^d$ it holds that 
	\begin{equation}\label{eq:solution_of_heat_eqn}
	(\tfrac{\partial}{\partial t}\Phi_{\varepsilon})(t,x) = (\Delta_x \Phi_{\varepsilon})(t,x). 
	\end{equation}
	This, \eqref{eq:subheat}, and \eqref{eq:definition_of_w_epsilon_M}
	imply that for every $\varepsilon, M\in (0,\infty)$, $t\in (0,T]$, $x\in\R^d$ it holds that 
	\begin{equation}\label{eq:inequality_for_w}
	\begin{split}
	(\tfrac{\partial}{\partial t}w_{\varepsilon,M})(t,x) 
	&= 
	(\tfrac{\partial }{\partial t}v)(t,x) - M(\tfrac{\partial}{\partial t}\Phi_{\varepsilon})(t,x) - \varepsilon 
	\\&= (\tfrac{\partial }{\partial t}v)(t,x) - M (\Delta_x \Phi_{\varepsilon})(t,x) - \varepsilon \\
	&\leq (\Delta_x v)(t,x) - M (\Delta_x \Phi_{\varepsilon})(t,x) - \varepsilon 
	= 
	(\Delta_x w_{\varepsilon,M})(t,x) - \varepsilon . 
	\end{split}	
	\end{equation}
	In addition, observe that \eqref{eq:assumptionExponentialGrowth} ensures that there exist $C\in [0,\infty)$ and $a\in (0,\infty)$ such that for every $t\in [0,T],$ $x\in\R^d$ it holds that
	\begin{equation}\label{eq:exponentialBoundForV}
	v(t,x)\leq C e^{a\|x\|^2}.
	\end{equation}
	To prove \eqref{eq:maxPrincipleForHeatEquationInequalityToProve} we distinguish 
	between the case $T<\frac{1}{4a}$ and the case $T\geq\frac{1}{4a}$. 
	We first prove \eqref{eq:maxPrincipleForHeatEquationInequalityToProve} in the 
	case $T<\frac{1}{4a}$. 
	%	Observe that for every 
	%	$\varepsilon\in (0,\frac{1}{4a}-T)$  it holds that $4a(T+\varepsilon) < 1$. 
	Observe that \eqref{eq:definition_of_Phi}, \eqref{eq:definition_of_w_epsilon_M} and \eqref{eq:exponentialBoundForV} imply that for every $\varepsilon\in (0,\frac{1}{4a}-T)$, $M\in(0,\infty),$  $t\in [0,T],$ $x\in\R^d$ it holds that  
	\begin{equation}\label{eq:estimating_w_epsilon_M}
	\begin{split}
	w_{\varepsilon,M}(t,x) 
	&\leq v(t,x)-M\Phi_{\varepsilon}(t,x)\\
	&=v(t,x)-\frac{M}{[4\pi(T+\varepsilon-t)]^{d/2}}\exp\!\left(\frac{\|x\|^2}{4(T+\varepsilon-t)}\right) \\
	& \leq 
	v(t,x) - \frac{M}{[4\pi(T+\varepsilon)]^{d/2}}\exp\!\left(\frac{\|x\|^2}{4(T+\varepsilon)}\right) 
	\\
	& \leq 
	C e^{a\|x\|^2} - \frac{M}{[4\pi(T+\varepsilon)]^{d/2}} \exp\!\left(\frac{\|x\|^2}{4(T+\varepsilon)}\right) 
	\\
	& = 
	e^{a\|x\|^2} \left[ C - \frac{M}{[4\pi(T+\varepsilon)]^{d/2}}  
	\exp\!\left(\|x\|^2\left[\frac{1}{4(T+\varepsilon)}-a\right]\right)\right]. 
	\end{split}
	\end{equation}
	Furthermore, observe that the hypothesis that $v\in C([0,T]\times \Rd,\R)$ and the fact that the interval $[0,T]$ is compact ensure that $\inf_{s\in[0,T]} v(s,0)\in\R.$ Hence, we obtain that for every $\varepsilon,M\in (0,\infty)$ it holds that 
	\begin{equation}
	\min\!\left\{0,\left[\inf_{s\in[0,T]} v(s,0)\right]-\varepsilon T-\frac{M}{(4\pi\varepsilon)^{d/2}}\right\}\in\R.
	\end{equation}
	This and the fact that for every $\varepsilon \in (0,\frac{1}{4a}-T)$ it holds that $a<\frac{1}{4(T+\varepsilon)}$ imply that 
	there exists a function $R=(R_{\varepsilon,M})_{(\varepsilon,M)\in (0,\infty)^2}\colon (0,\infty)^2\to (0,\infty)$ such that	
	for every $\varepsilon\in (0,\frac{1}{4a}-T)$, $M\in(0,\infty)$ it holds that
	\begin{equation}\label{eq:choiceOfRZero}
	\begin{split}
	&e^{a\vert R_{\varepsilon,M}\vert^2}\left[C-\frac{M}{[4\pi(T+\varepsilon)]^{d/2}}\exp\!\left(\vert R_{\varepsilon,M}\vert^2\left[\frac{1}{4(T+\varepsilon)}-a\right]\right)\right]\\
	&<\min\!\left\{0,\left[\inf_{s\in[0,T]} v(s,0)\right]-\varepsilon T-\frac{M}{(4\pi\varepsilon)^{d/2}}\right\}.
	\end{split}
	\end{equation}
	Combining this with \eqref{eq:definition_of_Phi} and \eqref{eq:definition_of_w_epsilon_M}
	proves that for every $\varepsilon\in (0,\frac{1}{4a}-T)$, $M\in(0,\infty),$ $t\in [0,T]$ it holds that
	\begin{equation}\label{eq:estimateBall}
	\begin{split}
	&e^{a\vert R_{\varepsilon,M}\vert^2}\left[C-\frac{M}{[4\pi(T+\varepsilon)]^{d/2}}\exp\!\left(\vert R_{\varepsilon,M}\vert^2\left[\frac{1}{4(T+\varepsilon)}-a\right]\right)\right]\\
	&<\left[\inf_{s\in[0,T]} v(s,0)\right]-\varepsilon T-\frac{M}{(4\pi\varepsilon)^{d/2}}
	\leq\left[\inf_{s\in[0,T]} v(s,0)\right]-\varepsilon t-\frac{M}{(4\pi\varepsilon)^{d/2}}\\
	&\leq\left[\inf_{s\in[0,T]} v(s,0)\right]-\varepsilon t-\frac{M}{[4\pi(T+\varepsilon-t)]^{d/2}}
	\leq v(t,0)-\varepsilon t-\frac{M}{[4\pi(T+\varepsilon-t)]^{d/2}}\\&=v(t,0)-\varepsilon t-M\Phi_{\varepsilon}(t,0)
	= w_{\varepsilon,M}(t,0)
	\leq\sup_{\substack{(s,x)\in[0,T]\times\Rd,\\ \norm{x}\leq R_{\varepsilon,M}}} w_{\varepsilon,M}(s,x).
	\end{split}
	\end{equation}	
	This, \eqref{eq:estimating_w_epsilon_M}, and \eqref{eq:choiceOfRZero} ensure that for every $\varepsilon\in (0,\frac{1}{4a}-T)$, $M\in(0,\infty),$ $t\in[0,T],$ $x\in \Rd$ with $\norm{x}>R_{\varepsilon,M}$ it holds that
	\begin{equation}
	\begin{split}
	w_{\varepsilon,M}(t,x) 
	& \leq 
	e^{a\|x\|^2} \left[ C - \frac{M}{[4\pi(T+\varepsilon)]^{d/2}}  
	\exp\!\left(\|x\|^2\left[\frac{1}{4(T+\varepsilon)}-a\right]\right)\right] 
	\\& \leq e^{a\|x\|^2} \left[ C - \frac{M}{[4\pi(T+\varepsilon)]^{d/2}}  
	\exp\!\left(\vert R_{\varepsilon,M}\vert^2\left[\frac{1}{4(T+\varepsilon)}-a\right]\right)\right]
	\\& \leq e^{a\vert R_{\varepsilon,M}\vert^2} \left[ C - \frac{M}{[4\pi(T+\varepsilon)]^{d/2}}  
	\exp\!\left(\vert R_{\varepsilon,M}\vert^2\left[\frac{1}{4(T+\varepsilon)}-a\right]\right)\right]
	\\&\leq\sup_{\substack{(s,y)\in[0,T]\times\Rd,\\ \norm{y}\leq R_{\varepsilon,M}}} w_{\varepsilon,M}(s,y).
	\end{split}
	\end{equation}	
	Therefore, we obtain that for every $\varepsilon\in (0,\frac{1}{4a}-T)$, 
	$M\in (0,\infty)$ it holds that
	\begin{equation}
	\sup_{(t,x)\in[0,T]\times\Rd} w_{\varepsilon,M}(t,x)=\sup_{\substack{(t,x)\in[0,T]\times\Rd,\\ \norm{x}\leq R_{\varepsilon,M}}} w_{\varepsilon,M}(t,x).
	\end{equation}
	%	Furthermore, note that for every $\varepsilon \in (0,\frac{1}{4a}-T)$, $R\in (0,\infty)$ it holds that
	%	\begin{equation}
	%	\inf_{x\in \R^d, \|x\|\ge R} \exp\!\left(\|x\|^2 \left( \frac{1}{4(T+\varepsilon)}-a\right) \right)
	%	=
	%	\exp\!\left(R^2 \left( \frac{1}{4(T+\varepsilon)}-a \right) \right). 
	%	\end{equation}
	%	This and \eqref{eq:estimating_w_epsilon_M} ensure that for every $\varepsilon\in (0,\frac{1}{4a}-T)$, 
	%	$M\in (0,\infty)$ it holds that 
	%	%  \begin{equation}
	%	%  \begin{split}
	%	\begin{multline}
	%	\limsup_{R\to\infty} 
	%	\left[ \sup_{ t\in [0,T], \|x\| \in [R,\infty)} \left( w_{\varepsilon,M}(t,x) \right) \right]
	%	\\
	%	\leq 
	%	\limsup_{R\to\infty}
	%	e^{a R^2} \left[ C - \frac{M}{[4\pi(T+\varepsilon)]^{d/2}} 
	%	\exp\!\left( R^2 \left( \frac{1}{4(T+\varepsilon)} - a \right) \right) \right] 
	%	= -\infty . 
	%	\end{multline}
	%	%  \end{split}
	%	%  \end{equation}
	The fact that for every $\varepsilon\in (0,\frac{1}{4a}-T)$, 
	$M\in (0,\infty)$ it holds that the 
	function $w_{\varepsilon,M}\colon [0,T]\times\R^d\to\R$ is continuous hence
	demonstrates that for every $\varepsilon\in (0,\frac{1}{4a}-T)$, $M\in (0,\infty)$ 
	there exists $(t_{\varepsilon,M},x_{\varepsilon,M})\in [0,T]\times\R^d$ such that it holds that
	\begin{equation}\label{eq:defMaximum}
	\sup_{(t,x)\in [0,T]\times\R^d} w_{\varepsilon,M}(t,x) = w_{\varepsilon,M}(t_{\varepsilon,M},x_{\varepsilon,M}) . 
	\end{equation}
	%Formulation 1
	%	Assume that $t_{\varepsilon,M}\in (0,T].$ Then 
	%	 \eqref{eq:defMaximum} together with $w_{\varepsilon,M}|_{(0,T]\times\R^d} \in C^{1,2}((0,T]\times\R^d,\R)$ ensures 
	%	for every $\varepsilon\in (0,\frac{1}{4a}-T)$, $M\in (0,\infty)$ that
	%Formulation 2
	The fact that for every $\varepsilon\in (0,\frac{1}{4a}-T)$, $M\in (0,\infty)$ it holds that $w_{\varepsilon,M}|_{(0,T]\times\R^d} \in C^{1,2}((0,T]\times\R^d,\R)$ therefore ensures that for every $\varepsilon\in (0,\frac{1}{4a}-T)$, $M\in (0,\infty),$ $v\in \Rd$ with 
	$t_{\varepsilon,M}>0$ it holds that
	%	
	%	with 
	%	$t_{\varepsilon,M}>0$ the fact that 
	%	$w_{\varepsilon,M}|_{(0,T]\times\R^d} \in C^{1,2}((0,T]\times\R^d,\R)$ together 
	%	with the fact that $w_{\varepsilon,M}$ has a maximum at 
	%	$(t_{\varepsilon,M},x_{\varepsilon,M})$ ensures that 
	\begin{equation}\label{eq:inequalities_at_maximum1}
	(\tfrac{\partial }{\partial t}w_{\varepsilon,M})(t_{\varepsilon,M},x_{\varepsilon,M})=0
	\qquad\text{and}\qquad 
	\big((\tfrac{\partial^2}{\partial{x}^2} w_{\varepsilon,M})(t_{\varepsilon,M},x_{\varepsilon,M})\big)(v,v) \leq 0. 
	\end{equation}
	Hence, we obtain that for every $\varepsilon\in (0,\frac{1}{4a}-T)$, $M\in (0,\infty)$ with $t_{\varepsilon,M}>0$ it holds that
	\begin{equation}\label{eq:inequalities_at_maximum}
	(\tfrac{\partial }{\partial t}w_{\varepsilon,M})(t_{\varepsilon,M},x_{\varepsilon,M})\ge 0
	\end{equation}
	and 
	\begin{equation}
	(\Delta_x w_{\varepsilon,M})(t_{\varepsilon,M},x_{\varepsilon,M})=\sum_{k=1}^d (\partialXX{k} w_{\varepsilon,M})(t_{\varepsilon,M},x_{\varepsilon,M}) \leq 0. 
	\end{equation}
	%  Moreover, \eqref{eq:definition_of_w_epsilon_M} implies that 
	%  for every $\varepsilon\in (0,\frac{1}{4a}-T)$, $M\in (0,\infty)$ it holds that
	%  \begin{equation}
	%   \begin{split}
	%   (\tfrac{\partial w_{\varepsilon,M}}{\partial t})(t_{\varepsilon,M},x_{\varepsilon,M}) 
	%   &
	%   = 
	%   (\tfrac{\partial }{\partial t}v)(t_{\varepsilon,M},x_{\varepsilon,M}) 
	%   - 
	%   (\tfrac{\partial }{\partial t})\left[
	%   \frac{1}{[4\pi(T+\varepsilon-t)]^{d/2}} 
	% 	  \exp\!\left( \frac{\|x\|^2}{4(T+\varepsilon-t)} \right) 
	%   \right]
	%   - \varepsilon 
	%   \\
	%   & \leq (\Delta_x ) v(t_{\varepsilon,M},x_{\varepsilon,M}) 
	%   + 
	%   (\Delta_x ) \left[ 
	%   \frac{1}{[4\pi(T+\varepsilon-t)]^{d/2}} 
	% 	  \exp\!\left( \frac{\|x\|^2}{4(T+\varepsilon-t)} \right) 
	%   \right]
	%   - \varepsilon \\
	%   & = (\Delta_x w_{\varepsilon,M})(t_{\varepsilon,M},x_{\varepsilon,M}) - \varepsilon . 
	%   \end{split}  
	%  \end{equation}
	This and \eqref{eq:inequality_for_w} imply that for every 
	$\varepsilon\in (0,\frac{1}{4a}-T)$, $M\in (0,\infty)$ with 
	$t_{\varepsilon,M}>0$ it holds that 
	\begin{equation}
	0 \leq (\tfrac{\partial}{\partial t}w_{\varepsilon,M})(t_{\varepsilon,M},x_{\varepsilon,M}) 
	\leq (\Delta_x w_{\varepsilon,M})(t_{\varepsilon,M},x_{\varepsilon,M}) - \varepsilon 
	\leq -\varepsilon < 0. 
	\end{equation}
	Hence, we obtain for every $\varepsilon\in (0,\frac{1}{4a}-T)$, $M\in (0,\infty)$ that $t_{\varepsilon,M}=0$. 
	Combining this with \eqref{eq:defMaximum} proves that
	for every $\varepsilon\in (0,\frac{1}{4a}-T)$, $M\in (0,\infty)$  it holds
	that 
	\begin{equation}
	\sup_{(t,x)\in [0,T]\times\R^d} w_{\varepsilon,M}(t,x) 
	%	= 
	%	w_{\varepsilon,M}(t_{\varepsilon,M},x_{\varepsilon,M}) 
	= w_{\varepsilon,M}(t_{\varepsilon,M},x_{\varepsilon,M})
	=
	w_{\varepsilon,M}(0,x_{\varepsilon,M}) \leq 
	\sup_{x\in\R^d} w_{\varepsilon,M}(0,x). 
	\end{equation} 
	%  Therefore we obtain for every $M\in (0,\infty)$ that 
	%  \begin{equation}
	%   \sup_{(t,x)\in [0,T]\times\R^d} w_{\varepsilon,M}(t,x) 
	%   = 
	%   w(0,x_{\varepsilon,M}) \leq \sup_{x\in\R^d} w_{\varepsilon,M}(0,x). 
	%  \end{equation}
	This and \eqref{eq:definition_of_w_epsilon_M} imply that for every $t\in [0,T]$, $x\in\R^d$, 
	$\varepsilon \in (0,\frac{1}{4a}-T)$, $M\in (0,\infty)$ it holds that
	%	\begin{equation}
	%	\begin{split}
	%	v(t,x) & = w_{\varepsilon,M}(t,x) + \frac{M}{[4\pi(T+\varepsilon-t)]^{d/2}} 
	%	\exp\!\left( \frac{\|x\|^2}{4(T+\varepsilon-t)} \right)  + \varepsilon t 
	%	\\
	%	& \leq \sup_{x\in\R^d} w_{\varepsilon,M}(0,x) 
	%	+ \frac{M}{[4\pi(T+\varepsilon-t)]^{d/2}} 
	%	\exp\!\left( \frac{\|x\|^2}{4(T+\varepsilon-t)} \right)  + \varepsilon t 
	%	\\
	%	& \leq \sup_{x\in\R^d} v(0,x) 
	%	+
	%	\frac{M}{[4\pi(T+\varepsilon-t)]^{d/2}} 
	%	\exp\!\left( \frac{\|x\|^2}{4(T+\varepsilon-t)} \right)  + \varepsilon t . 
	%	\end{split}
	%	\end{equation}
	\begin{equation}
	\begin{split}
	v(t,x) & = w_{\varepsilon,M}(t,x) + M\Phi_\varepsilon(t,x)  + \varepsilon t 
	\leq \left[\sup_{y\in\R^d} w_{\varepsilon,M}(0,y)\right]
	+ M\Phi_\varepsilon(t,x)  + \varepsilon t 
	\\
	& \leq \left[\sup_{y\in\R^d} v(0,y)\right] 
	+
	M\Phi_\varepsilon(t,x)  + \varepsilon t . 
	\end{split}
	\end{equation}
	Therefore, we obtain that for every $t\in [0,T]$, $x\in\R^d$, $\varepsilon\in (0,\frac{1}{4a}-T)$ it holds 
	that 
	\begin{equation}
	\begin{split}
	v(t,x) 
	& \leq \liminf_{M \searrow 0} \left(\left[ 
	\sup_{y\in\R^d} v(0,y) \right]
	+
	M\Phi_\varepsilon(t,x)  + \varepsilon t 
	\right) 
	= 
	\left[\sup_{y\in\R^d} v(0,y)\right] + \varepsilon t. 
	\end{split}
	\end{equation}
	Hence, we obtain that for every $t\in [0,T]$, $x\in\R^d$ it holds that 
	\begin{equation}
	v(t,x) \leq \liminf_{\varepsilon \searrow 0} \left(\left[\sup_{y\in\R^d} v(0,y)\right] + \varepsilon t\right)
	= 
	\sup_{y\in\R^d} v(0,y). 
	\end{equation}
	This establishes \eqref{eq:maxPrincipleForHeatEquationInequalityToProve} 
	in the case $T<\frac{1}{4a}$.
	%%	We now prove 
	%%	\eqref{eq:maxPrincipleForHeatEquationInequalityToProve} in the case that there 
	%	Next, we assume that there exists $L\in\N$  such that it holds that $T=\frac{L}{8a}$ and prove \eqref{eq:maxPrincipleForHeatEquationInequalityToProve} by in			
	We now prove 
	\eqref{eq:maxPrincipleForHeatEquationInequalityToProve} in the case $T\geq \frac{1}{4a}$. 
	For this let $k\in \N$  and $\mathcal{T} \in (0,\frac{1}{8a}]$ be the real numbers which satisfy that 
	\begin{equation}
	T = \frac{k}{8a} + \mathcal{T},
	\end{equation}
	let  $\tau_l\in\R,$ $l\in \left\{0,1,\dots,k+1\right\}$, be the real numbers which satisfy  for all $l\in \{0,1,\dots, k\}$ that
	\begin{equation}
	\tau_l=\frac{l}{8a}\qquad \text{and}\qquad \tau_{k+1}=T,
	\end{equation}	
	and let $\mathfrak{v}_l\colon [0,\tau_{l+1}-\tau_l]\times \Rd\to\R,$ $l\in \left\{0,1,\dots,k\right\}$, be the functions which satisfy for all $l\in \left\{0,1,\dots,k\right\},$ $t\in [0,\tau_{l+1}-\tau_l],$ $x\in\Rd$ that 
	\begin{equation}\label{eq:defVl}
	\mathfrak{v}_l(t,x)=v(t+\tau_l,x).
	\end{equation}
	Next we claim that for every $l\in \{0,1,\dots, k+1\}$ it holds that
	\begin{equation}\label{eq:maxPrincipleForHeatEquationInequalityToProveInduction}
	\sup_{(t,x)\in [0,\tau_l]\times\R^d} v(t,x)= \sup_{x\in\R^d} v(0,x). 
	\end{equation}
	We now prove \eqref{eq:maxPrincipleForHeatEquationInequalityToProveInduction} by induction on $l\in \{0,1,\dots, k+1\}.$ Observe that the fact that
	\begin{equation}
	\sup_{(t,x)\in [0,\tau_0]\times \Rd} v(t,x)=\sup_{(t,x)\in \{0\}\times \Rd} v(t,x)=\sup_{x\in\Rd} v(0,x)
	\end{equation}
	establishes \eqref{eq:maxPrincipleForHeatEquationInequalityToProveInduction} in the base case $l=0.$
	For the induction step $\{0,1,\dots,k\}\ni l\to l+1 \in \{1,2,\dots,k+1\}$ assume that there exists $l\in \left\{0,1,\dots, k\right\}$ such that 
	\begin{equation}\label{eq:maxPrincipleForHeatEquationInequalityToProveInduction2}
	\sup_{(t,x)\in [0,\tau_l]\times\R^d} v(t,x)= \sup_{x\in\R^d} v(0,x). 
	\end{equation}
	%	
	%	it holds that 
	%	 \begin{equation}\label{eq:maxPrincipleForHeatEquationInequalityToProveInduction}
	%	 \sup_{(t,x)\in [0,\tau_l]\times\R^d} v(t,x) \leq \sup_{x\in\R^d} v(0,x). 
	%	 \end{equation}
	In addition, note that \eqref{eq:subheat},
	\eqref{eq:exponentialBoundForV}, 
	and
	\eqref{eq:defVl}
	ensure that for every $t\in (0,\tau_{l+1}-\tau_l],$ $x\in \Rd$ it holds that 
	\begin{equation}
	(\tfrac{\partial }{\partial t}\mathfrak{v}_l)(t,x)=(\tfrac{\partial }{\partial t}v)(t+\tau_l,x) \leq (\Delta_x v)(t+\tau_l,x)=(\Delta_x \mathfrak{v}_l)(t,x)
	\end{equation}
	and 
	\begin{equation}
%	\inf_{a\in \R} \left(
	\sup_{(t,x)\in [0,\tau_{l+1}-\tau_l]\times\R^d} (e^{-a\|x\|^2}\mathfrak{v}_l(t,x)) %\right)
	=
	%\inf_{a\in \R} \left(
	\sup_{(t,x)\in [\tau_l,\tau_{l+1}]\times\R^d} (e^{-a\|x\|^2}v(t,x)) %\right)
	\leq C < \infty.  
	\end{equation}
	This, \eqref{eq:maxPrincipleForHeatEquationInequalityToProveInduction2}, and \eqref{eq:maxPrincipleForHeatEquationInequalityToProve} 
	in the case $ T<\frac{1}{4a}$ show that 
	\begin{equation}
	\begin{split}
	\sup_{(t,x)\in [\tau_l,\tau_{l+1}]\times\R^d} v(t,x)
	&=\sup_{(t,x)\in [0,\tau_{l+1}-\tau_l]\times\R^d} \mathfrak{v}_l(t,x) 
	=\sup_{x\in\R^d} \mathfrak{v}_l(0,x)=\sup_{x\in\R^d} v(\tau_l,x)\\
	&\leq  \sup_{(t,x)\in [0,\tau_{l}]\times\R^d} v(t,x)
	= \sup_{x\in\R^d} v(0,x).
	\end{split}
	\end{equation}Therefore, we obtain that
	\begin{equation}
	\begin{split}
	\sup_{(t,x)\in [0,\tau_{l+1}]\times\R^d} v(t,x)
	&=\max\!\left\{	\sup_{(t,x)\in [0,\tau_{l}]\times\R^d} v(t,x),	\sup_{(t,x)\in [\tau_l,\tau_{l+1}]\times\R^d} v(t,x)\right\}\\
	&=\max\!\left\{	\sup_{x\in\R^d} v(0,x),	\sup_{(t,x)\in [\tau_l,\tau_{l+1}]\times\R^d} v(t,x)\right\}\\
	&\leq \sup_{x\in\R^d} v(0,x). 
	\end{split}
	\end{equation}
	Induction hence proves \eqref{eq:maxPrincipleForHeatEquationInequalityToProveInduction}. Furthermore, note that \eqref{eq:maxPrincipleForHeatEquationInequalityToProveInduction} and the fact that $T=\tau_{k+1}$ imply that 
	\begin{equation}
	\sup_{(t,x)\in [0,T]\times\R^d} v(t,x)=\sup_{(t,x)\in [0,\tau_{k+1}]\times\R^d} v(t,x)= \sup_{x\in\R^d} v(0,x).
	\end{equation}
	This establishes \eqref{eq:maxPrincipleForHeatEquationInequalityToProve} in the 
	case $T\geq\frac{1}{4a}$. The proof of Lemma~\ref{lem:maxPrincipleForHeatEquation} 
	is thus completed.
\end{proof}

\begin{cor}\label{cor:maxPrincipleForHeatEquation}
	Let $d\in\N$, $T\in (0,\infty)$, $v\in C([0,T]\times\R^d,\R),$ assume that $v|_{(0,T]\times\R^d} \in C^{1,2}((0,T]\times\R^d,\R),$
	%	assume for every $t\in [0,T],$ $x\in\R^d$ that $v(t,x)\ge 0,$ 
	assume for every $t\in (0,T],$ $x\in\R^d$ that
	\begin{equation}\label{eq:subheatCor}
	(\tfrac{\partial }{\partial t}v)(t,x) \leq (\Delta_x v)(t,x), 
	\end{equation}
	let $\norm{\cdot}\colon \R^d\to [0,\infty)$ be a norm,
	and assume that
	\begin{equation}\label{eq:assumptionExponentialGrowthCor}
	\inf_{a\in \R} \sup_{(t,x)\in [0,T]\times\R^d} (e^{a\|x\|^2}v(t,x)) < \infty.
	\end{equation}
	%  which satisfies 
	Then it holds that 
	\begin{equation}\label{eq:maxPrincipleForHeatEquationInequalityToProveCor}
	\sup_{(t,x)\in [0,T]\times\R^d} v(t,x)= \sup_{x\in\R^d} v(0,x) . 
	\end{equation}
\end{cor}

\begin{proof}[Proof of Corollary~\ref{cor:maxPrincipleForHeatEquation}]
	Throughout this proof let $\tripleNorm{\cdot}\colon\Rd\to [0,\infty)$ be the 
	$d$-dimensional Euclidean norm 
	and let $c\in(0,\infty)$ be the real number which satisfies that 
	\begin{equation}\label{eq:normEquivalence}
	c=\sup_{x\in\Rd\setminus\{0\}} \left(\frac{\norm{x}}{\tripleNorm{x}}\right).
	\end{equation}
	Note that \eqref{eq:assumptionExponentialGrowthCor} ensures that there exists $a\in (-\infty,0]$ such that 
	% for every 
	% $t\in [0,T]$, $x\in\R^d$ it holds that 
	\begin{equation}\label{eq:expGrowthUCor}
	\sup_{(t,x)\in [0,T]\times\R^d} ( e^{a\|x\|^2} v(t,x) ) < \infty.
	\end{equation}
	In addition, observe that \eqref{eq:normEquivalence} implies that for all $x\in\Rd\setminus\{0\}$ it holds that 
	\begin{equation}
	a\Norm{x}^2=a \left[\frac{\norm{x}}{\tripleNorm{x}}\right]^2 \tripleNorm{x}^2\ge a c^2 \tripleNorm{x}^2.
	\end{equation}
	%	\begin{equation}
	%\tripleNorm{x}^2=\left[\frac{\tripleNorm{x}}{\norm{x}}\right]^2 \norm{x}^2\le c^2 \norm{x}^2.
	%\end{equation}
	Combining this with \eqref{eq:expGrowthUCor} demonstrates that for every $t\in[0,T],$ $x\in\Rd$ it holds that
	\begin{equation}
	\begin{split}
	e^{ a c^2 \tripleNorm{x}^2}v(t,x)&
	\le e^{ a c^2 \tripleNorm{x}^2}\max\!\left\{0,v(t,x)\right\}
	\\&\le e^{ a\Norm{x}^2}\max\!\left\{0,v(t,x)\right\}\\
	&=\max\!\left\{0,e^{a\Norm{x}^2}v(t,x)\right\}
	\\&\le \max\!\left\{0,\sup_{(s,y)\in [0,T]\times\R^d} ( e^{a\|y\|^2} v(s,y) )\right\}<\infty.
	\end{split}
	\end{equation} 
	This ensures that 
	\begin{equation}
	\sup_{(t,x)\in [0,T]\times\R^d} ( e^{ac^2\tripleNorm{x}^2} v(t,x) )\le \max\!\left\{0,\sup_{(t,x)\in [0,T]\times\R^d} ( e^{a\|x\|^2} v(t,x) )\right\}<\infty.
	\end{equation}
	Hence, we obtain that
	\begin{equation}
	\inf_{\alpha\in\R}\sup_{(t,x)\in [0,T]\times\R^d} ( e^{\alpha\tripleNorm{x}^2} v(t,x) )<\infty.
	\end{equation}
	%
	%
	%	This and \eqref{eq:assumptionExponentialGrowthCor} demonstrate that 
	%	\begin{equation}
	%		\begin{split}
	%		\inf_{a\in \R} \sup_{(t,x)\in [0,T]\times\R^d} (e^{a\tripleNorm{x}^2}v(t,x)) 
	%		&\le\inf_{a\in \R} \sup_{(t,x)\in [0,T]\times\R^d} (e^{ac^2\norm{x}^2}v(t,x))
	%		\\&=\inf_{a\in \R} \sup_{(t,x)\in [0,T]\times\R^d} (e^{a\norm{x}^2}v(t,x))  
	%		< \infty.
	%		\end{split}
	%	\end{equation}
	Combining this with \eqref{eq:subheatCor} enables us to apply Lemma \ref{lem:maxPrincipleForHeatEquation} to obtain that 
	\begin{equation}
	\sup_{(t,x)\in [0,T]\times\R^d} v(t,x)= \sup_{x\in\R^d} v(0,x). 
	\end{equation}
	The proof of Corollary~\ref{cor:maxPrincipleForHeatEquation} 
	is thus completed.
\end{proof}

\begin{theorem}\label{thm:reactionDiffusionEquationWithCoerciveNonlinearity}
	Let $d\in\N$, $T\in (0,\infty)$, $c\in \R$,  
	let $f\colon [0,T]\times\R^d\times\R\to\R$ be a function, 
	assume for every $t\in [0,T]$, $x\in\R^d$, $y\in\R$ that 
	\begin{equation}\label{eq:growthOfNonlinearityMaxPrinciple}
	y f(t,x,y) \leq c ( 1 + y^2 ),
	\end{equation}
	let $u\in C([0,T]\times\R^d,\R),$ 
	assume that $u|_{(0,T]\times\R^d} \in C^{1,2}((0,T]\times\R^d,\R)$, 
	let $\norm{\cdot}\colon \R^d\to [0,\infty)$ be a norm,
	assume that 
	\begin{equation}\label{eq:assumptionExponentialGrowthThm}
	\inf_{a\in \R} \sup_{(t,x)\in [0,T]\times\R^d} (e^{a\|x\|^2}|u(t,x)|) < \infty,
	\end{equation}
	and assume for every $t\in (0,T]$, $x\in\R^d$ that 
	\begin{equation}\label{eq:pde_for_uMaxPrinciple}
	(\tfrac{\partial u}{\partial t})(t,x) = (\Delta_x u)(t,x) + f(t,x,u(t,x)). 
	\end{equation}
	Then it holds for every 
	$t\in [0,T]$ that 
	\begin{equation}
	\sup_{x\in\R^d} |u(t,x)| 
	\leq 
	\left[\sup_{x\in\R^d}(1+  |u(t,x)|^2)\right]^{\nicefrac12}
	\leq e^{c t}\left[1+\sup_{x\in\R^d} |u(0,x)|^2\right]^{\nicefrac12}. 
	\end{equation}
\end{theorem}

\begin{proof}[Proof of Theorem~\ref{thm:reactionDiffusionEquationWithCoerciveNonlinearity}]
	Throughout this proof let $v\colon [0,T]\times\R^d\to\R$ be the function which satisfies for every 
	$t\in [0,T]$, $x\in\R^d$ that 
	\begin{equation}\label{eq:defV}
	v(t,x) = e^{-2ct}(1+|u(t,x)|^2). 
	\end{equation}
	Note that \eqref{eq:assumptionExponentialGrowthThm} ensures that there exists $a\in (-\infty,0]$ such that 
	% for every 
	% $t\in [0,T]$, $x\in\R^d$ it holds that 
	\begin{equation}\label{eq:expGrowthU}
	\sup_{(t,x)\in [0,T]\times\R^d} \left( e^{a\|x\|^2} |u(t,x)| \right) < \infty.
	\end{equation}
	%  $(t,x,w)\in [0,T]\times\R^d\times\R$ it holds that $wf(t,x,w)\leq c(1+|w|^2)$ 
	Moreover, observe that \eqref{eq:growthOfNonlinearityMaxPrinciple} demonstrates that $c\ge 0.$ Combining this with \eqref{eq:defV} and \eqref{eq:expGrowthU} implies that
	\begin{equation}
	\begin{split}
	&\sup_{(t,x)\in [0,T]\times\R^d}  \left({e^{2a\|x\|^2}}{v(t,x)}\right)\\
	&= 
	\sup_{(t,x)\in [0,T]\times\R^d} \left({e^{2a\|x\|^2}} e^{-2ct} \left(1+|u(t,x)|^2\right)\right) 
	\\
	& 
	\leq
	\sup_{(t,x)\in [0,T]\times\R^d} \left({e^{2a\|x\|^2}}  \left(1+|u(t,x)|^2\right)\right) 
	\\
	& 
	\leq
	\left[\sup_{(t,x)\in [0,T]\times\R^d}\left( {e^{2a\|x\|^2}}\right)\right]+\left[\sup_{(t,x)\in [0,T]\times\R^d} \left({e^{2a\|x\|^2}}|u(t,x)|^2\right)\right] 
	%  \leq 
	%  1 + \sup_{(t,x)\in [0,T]\times\R^d} \left(\frac{|u(t,x)|^2}{e^{2a\|x\|^2}}\right)
	\\
	& 
	=
	1 + \left[ \sup_{(t,x)\in [0,T]\times \R^d} \left( {e^{a\|x\|^2}}{|u(t,x)|} \right) \right]^2
	< \infty . 
	\end{split}
	\end{equation} 
	Hence, we obtain that
	\begin{equation}\label{eq:inequality_for_v}
	\inf_{\alpha\in \R} \sup_{(t,x)\in [0,T]\times\R^d} (e^{\alpha\|x\|^2}v(t,x)) < \infty.
	\end{equation}
	Next observe that \eqref{eq:defV}, the hypothesis that $u\in C([0,T]\times\R^d,\R),$ and the hypothesis that $u|_{(0,T]\times\R^d} \in C^{1,2}((0,T]\times\R^d,\R)$ ensure that
	\begin{equation}\label{eq:smoothnessV}
	v\in C([0,T]\times\R^d,\R)\qquad \text{and}\qquad v|_{(0,T]\times\R^d} \in C^{1,2}((0,T]\times\R^d,\R).
	\end{equation}
	Furthermore, note that \eqref{eq:defV} demonstrates that for every $t\in (0,T],$ $x\in\R^d$ it holds 
	that $v(0,x) = 1 + |u(0,x)|^2$ and
	\begin{equation}
	%   \begin{split}
	(\tfrac{\partial}{\partial t} v)(t,x) 
	%   & 
	= 
	%(\tfrac{\partial }{\partial t})(e^{-2ct}) (1 + |u(t,x)|^2) 
	%+ 
	%e^{-2ct} (\tfrac{\partial }{\partial t})( 1 + |u(t,x)|^2 ) 
	%\\
	%& = 
	-2c e^{-2ct}(1+|u(t,x)|^2) + 2 e^{-2ct} u(t,x)(\tfrac{\partial u}{\partial t})(t,x) . 
	%   \end{split}
	\end{equation}
	This, \eqref{eq:pde_for_uMaxPrinciple}, and \eqref{eq:growthOfNonlinearityMaxPrinciple} imply that for every $t\in (0,T]$, $x\in\R^d$ it holds that 
	% \begin{equation}
	%  (\tfrac{\partial }{\partial t}v)(t,x) 
	%  = 
	%  - 2c v(t,x) + 2e^{-2ct} u(t,x) \left( (\Delta_x u)(t,x) + f(t,x,u(t,x)) \right). 
	% \end{equation}
	% Hence, \eqref{eq:growthOfNonlinearityMaxPrinciple} yields 
	% for every $t\in (0,T]$, $x\in\R^d$ that 
	\begin{equation}
	\begin{split}
	&(\tfrac{\partial }{\partial t}v)(t,x) 
	= 
	-2c e^{-2ct}(1+|u(t,x)|^2) + 2e^{-2ct} u(t,x) \left( (\Delta_x u)(t,x) + f(t,x,u(t,x)) \right)\\
	& \leq 
	-2c e^{-2ct}(1+|u(t,x)|^2) + 2e^{-2ct} u(t,x) (\Delta_x u)(t,x) + 2c e^{-2ct} (1+|u(t,x)|^2) \\
	& = 
	2e^{-2ct} u(t,x) (\Delta_x u)(t,x).  
	\end{split}
	\end{equation}
	% In addition, observe that for every $w\in C^2(\Rd,\R)$  it holds that
	The fact that for every twice differentiable function 
	$w\colon\R^d\to\R$ and every $x=(x_1,\dots,x_d)\in\Rd$ it holds that
	\begin{equation}
	\begin{split}
	(\Delta(|w|^2))(x)=&\sum_{k=1}^d \left[\partialXX{k}\left(\vert w(x)\vert^2\right)\right]
	=\sum_{k=1}^d \left[\partialX{k}\left(2 w(x) (\partialX{k}w)(x)\right)\right]
	\\=& \sum_{k=1}^d \left[2\vert (\partialX{k}w)(x)\vert^2+2 w(x) (\partialXX{k}w)(x)\right]\\
	=& 2\left[\sum_{k=1}^d \vert (\partialX{k}w)(x)\vert^2\right]+2 w(x) \left[\sum_{k=1}^d (\partialXX{k}w)(x)\right]\\
	=& 2 w(x) (\Delta w)(x)+2\left[\sum_{k=1}^d \vert (\partialX{k}w)(x)\vert^2\right]
	\end{split}
	\end{equation}
	% 
	% 
	% $\Delta(|w|^2) = 2w\Delta w + 2\|\nabla w\|_{\R^d}^2$ 
	therefore implies that for every $t\in (0,T]$, $x=(x_1,\dots,x_d)\in\Rd$ it holds that 
	\begin{equation}
	\begin{split}
	(\tfrac{\partial }{\partial t}v)(t,x)
	&\leq e^{-2ct} \big(2 u(t,x) (\Delta_x u)(t,x)\big)\\
	&= e^{-2ct}\left(\left(\Delta_x(|u|^2)\right)\!(t,x)-2\left[\sum_{k=1}^d \vert (\partialX{k}u)(t,x)\vert^2\right]\right)\\
	&=
	(\Delta_x v)(t,x) - 2 e^{-2ct}\left[\sum_{k=1}^d \vert (\partialX{k}u)(t,x)\vert^2\right]
	\leq 
	(\Delta_x v)(t,x). 
	\end{split}
	\end{equation}
	%  Moreover, note that the growth conditions on $u$ guarantee that 
	%  \begin{equation} 
	%   \sup_{(t,x)\in [0,T]\times\R^d} \frac{v(t,x)}{e^{2a\|x\|^2}} < \infty. 
	%  \end{equation}
	Combining this with \eqref{eq:inequality_for_v} and \eqref{eq:smoothnessV} enables us to apply Corollary~\ref{cor:maxPrincipleForHeatEquation} 
	to obtain that 
	\begin{equation}
	0 \leq \sup_{(t,x)\in [0,T]\times\R^d} v(t,x) \leq \sup_{x\in\R^d} v(0,x) = 1 + \sup_{x\in\R^d} |u(0,x)|^2 . 
	\end{equation}
	Therefore, we obtain that for every $t\in [0,T]$ it holds that  
	\begin{equation}
	\begin{split}
	\sup_{x\in\R^d} |u(t,x)| 
	&=  \left[\sup_{x\in\R^d}|u(t,x)|^2\right]^{\nicefrac12}
	\leq\left[\sup_{x\in\R^d} \left(1+|u(t,x)|^2\right)\right]^{\nicefrac12}\\
	&=e^{ct}\left[\sup_{x\in\R^d} \left(e^{-2ct}\left(1+|u(t,x)|^2\right)\right)\right]^{\nicefrac12}
	%  \leq e^{ct} \sqrt{v(t,x)} 
	%  \leq e^{ct} \sqrt{1+\|u\|_{\infty}^2} . 
	= e^{c t}\left[\sup_{x\in\R^d} v(t,x)\right]^{\nicefrac12}\\
	&\leq e^{c t}\left[\sup_{(s,x)\in [0,T]\times\R^d} v(s,x)\right]^{\nicefrac12}
	\leq e^{c t}\left[1+\sup_{x\in\R^d} |u(0,x)|^2\right]^{\nicefrac12} .
	\end{split}
	\end{equation}
	The proof of Theorem~\ref{thm:reactionDiffusionEquationWithCoerciveNonlinearity} is thus 
	completed. 
\end{proof}

\begin{cor}\label{cor:backwardReactionDiffusionEquationWithCoerciveNonlinearity} 
Let 
	$d\in\N$, 
	$T\in (0,\infty)$, 
	$c\in \R$, 
let 
	$\norm{\cdot}\colon\R^d\to [0,\infty)$ 
be a norm, 
let 
	$f\colon [0,T]\times\R^d\times\R \to \R$ 
be a function which satisfies for every 
	$t\in [0,T]$, 
	$x\in \R^d$,  
	$y\in \R$
that 
	$
	yf(t,x,y) \leq c(1+y^2)
	$, 
and let 
	$u\in C([0,T]\times\R^d,\R)$ 
satisfy for every 
	$t\in [0,T)$, 
	$x\in \R^d$ 	
that 
	$u|_{[0,T)\times\R^d}\in C^{1,2}([0,T)\times\R^d,\R)$, 
	$
	\inf_{a\in\R} 
	\sup_{(s,y)\in [0,T]\times\R^d} 
	( 
	e^{a\norm{y}^2} |u(s,y)| 
	)
	< \infty
	$, 
and 
	\begin{equation} 
	\label{cor:backwardReactionDiffusionEquationWithCoerciveNonlinearity_pde_for_u} 
	(\tfrac{\partial }{\partial t}u)(t,x) 
	+ 
	\tfrac12 (\Delta_x u)(t,x) 
	+ 
	f(t,x,u(t,x)) 
	= 
	0. 
	\end{equation} 
Then it holds for every 
	$t\in [0,T]$ 
that 
	\begin{equation} 
	\sup_{x\in\R^d} |u(t,x)| 
	\leq 
	\left[ \sup_{x\in\R^d} (1+|u(t,x)|^2)\right]^{\nicefrac12}
	\leq 
	e^{c(T-t)} \left[ 1 + \sup_{x\in\R^d} |u(T,x)|^2\right]^{\nicefrac12}. 
	\end{equation}
\end{cor}

\begin{proof}[Proof of \cref{cor:backwardReactionDiffusionEquationWithCoerciveNonlinearity}]
	Throughout this proof let 
	$U\colon [0,T]\times\R^d\to\R$ 
	and 
	$F\colon [0,T]\times\R^d\times\R \to\R$ 
	be the functions which satisfy for every 
	$t\in [0,T]$, 
	$x\in \R^d$, 
	$y\in\R$ 
	that 
	$U(t,x)=u(T-t,\frac{x}{\sqrt{2}})$ and $F(t,x,y)=f(T-t,\frac{x}{\sqrt{2}},y)$.
	Observe that the assumption that for every 
	$t\in [0,T]$, 
	$x\in \R^d$, 
	$y\in \R$ 
	it holds that
	$vf(t,x,y) \leq c(1+y^2)$ 	
	implies for every 
	$t\in [0,T]$, 
	$x\in \R^d$, 
	$y\in \R$ 
	that 
	\begin{equation} 
	\label{eq:backwardReactionDiffusionEquationWithCoerciveNonlinearity_continuityDifferentiabilityAndCoercivity}
	yF(t,x,y) = y f(T-t,\tfrac{x}{\sqrt{2}},y) \leq c ( 1 + y^2 ) .  
	\end{equation}
	Moreover, observe that the hypothesis that 
	$
	\inf_{a\in\R} 
	\sup_{(s,y)\in [0,T]\times\R^d} 
	( 
	e^{a\norm{y}^2} |u(s,y)| 
	)
	< \infty
	$ 
	ensures that there exists 
	$\alpha\in\R$
	which satisfies that 
	\begin{equation} 
	\sup_{(t,x)\in [0,T]\times\R^d} 
	\big( 
	e^{\alpha\norm{x}^2}|u(t,x)| 
	\big)
	< \infty. 
	\end{equation}
	This implies that 
	\begin{equation} 
	\sup_{(t,x)\in [0,T]\times\R^d} 
	\big(
	e^{\tfrac{\alpha}{2}\norm{x}^2}
	|U(t,x)|
	\big) 
	= 
	\sup_{(t,x)\in [0,T]\times\R^d} 
	\big(
	e^{\alpha\norm{\nicefrac{x}{\sqrt{2}}}^2} 
	|u(T-t,\tfrac{x}{\sqrt{2}})|
	\big) 
	< 
	\infty. 
	\end{equation}  	 
	Hence, we obtain that 
	\begin{equation}
	\label{cor:backwardReactionDiffusionEquationWithCoerciveNonlinearity_growth_of_U}
	\inf_{a\in\R}
	\sup_{(t,x)\in [0,T]\times\R^d} 
	\big( 
	e^{a\norm{x}^2} |U(t,x)| 
	\big)  
	< 
	\infty.
	\end{equation}   
	In addition, note that the hypothesis that 
	$u\in C([0,T]\times\R^d,\R)$, 
	the hypothesis that 
	$u|_{[0,T)\times\R^d}\in C^{1,2}([0,T)\times\R^d,\R)$, 
	the chain rule, and \eqref{cor:backwardReactionDiffusionEquationWithCoerciveNonlinearity_pde_for_u}  ensure that for every 
	$t\in (0,T]$, 
	$x\in \R^d$ 
	it holds that
	$U\in C([0,T]\times\R^d,\R)$, 
	that 
	$U|_{(0,T]\times\R^d}\in C^{1,2}((0,T]\times\R^d,\R)$, 
	and that  
	\begin{equation}\label{cor:backwardReactionDiffusionEquationWithCoerciveNonlinearity_pde_for_U}	 
	(\tfrac{\partial }{\partial t}U)(t,x) 
	= (\Delta_x U)(t,x) + F(t,x,U(t,x)). 
	\end{equation} 
	Combining this,
	\eqref{eq:backwardReactionDiffusionEquationWithCoerciveNonlinearity_continuityDifferentiabilityAndCoercivity}, and 
	\eqref{cor:backwardReactionDiffusionEquationWithCoerciveNonlinearity_growth_of_U} with \cref{thm:reactionDiffusionEquationWithCoerciveNonlinearity} 
	(with $f=F$, $u=U$ in the notation of \cref{thm:reactionDiffusionEquationWithCoerciveNonlinearity})
	demonstrates for every 
	$t\in [0,T]$ 
	that 
	\begin{equation}
	\begin{split}
	\sup_{x\in \R^d} |u(t,x)| 
	& = 
	\left[ 
	1 + \sup_{x\in \R^d} |u(t,x)|^2
	\right]^{\nicefrac12} 
	= 
	\left[ 
	1 + \sup_{x\in\R^d} |U(T-t,x)|^2 
	\right]^{\nicefrac12} 
	\\
	& \leq e^{c (T-t)}
	\left[1+\sup_{x\in\R^d} |U(0,x)|^2\right]^{\nicefrac{1}{2}}
	= 
	e^{c(T-t)} \left[ 1 + \sup_{x\in\R^d} |u(T,x)|^2\right]^{\nicefrac12}. 
	\end{split}
	\end{equation}
	This completes the proof of \cref{cor:backwardReactionDiffusionEquationWithCoerciveNonlinearity}. 
\end{proof}

\section{Truncated full-history recursive multilevel Picard (MLP) approximations}
\label{sec:algorithm}

In this section we present and analyze a (truncated) MLP approximation scheme for reaction-diffusion type PDEs with coercive nonlinearity (see \cref{setting} 
below for details). %and demonstrate under suitable assumptions convergence of the multilevel Picard approximations to exact solutions of reaction-diffusion type PDEs with coercive nonlinearity (see \cref{thm:convergenceLocalLipschitz}).
The error analysis relies on results in~\cite[Section 3]{Overcoming} (cf.~also \cref{thm:convergenceLocalLipschitz} below) in combination with a Feynman--Kac representation (cf.~\cref{feynman_kac}) and the a priori estimates in \cref{sec:maximum_principle} 
above. %In \cref{thm:convergenceLocalLipschitz_forward_formulation} we present truncated multilevel Picard schemes for reaction-diffusion type PDEs in initial value formulation and carry over the error estimates. 

\begin{setting}[Setting and algorithm]
	\label{setting}
Let 
	$d \in \N$, 
	$T \in (0,\infty)$, 
	$\Theta = \cup_{n \in \N} \Z^n$, 
	$f\in C([0,T]\times\R^d\times\R,\R)$,
	$g\in C(\R^d,\R)$, 
let 
	$\fr \colon [0,T]\times\R^d\times\R\to\R$, $r\in (0,\infty)$, 
satisfy for every 
	$r \in (0,\infty)$, 
	$t \in [0,T]$, 
	$x \in \R^d$, 
	$u \in \R$ 
that 
	\begin{equation}\label{eq:truncation}
	\fr (t,x,u) = f(t, x, \min\{r,\max\{-r,u\}\}), 
	\end{equation} 
let 
	$(\Omega,\cF,\P)$ be a probability space, 
let 
	$\mathcal{R}^{\theta}\colon \Omega \to [0,1]$, $\theta\in\Theta$, 
be independent $\mathcal{U}_{[0,1]}$-distributed random variables, 
let 
	$W^{\theta}\colon [0,T]\times\Omega\to\R^d$, $\theta\in\Theta$, 
be independent standard Brownian motions, 
assume that 
	$(\mathcal{R}^{\theta})_{\theta\in\Theta}$  and $(W^{\theta})_{\theta\in\Theta}$ are independent, 
let 
	$R^{\theta}\colon [0,T]\times\Omega \to [0,T]$, $\theta\in\Theta$, 
satisfy for every 
	$\theta\in\Theta$,  
	$t\in [0,T]$
that 
	$R^{\theta}_t = t + (T-t) \mathcal{R}^{\theta}$, 
for every 
	$\theta\in\Theta$,
	$t\in [0,T]$, 
	$s\in [t,T]$, 
	$x\in \R^d$
let 
	$X^{\theta}_{t,s,x}\colon\Omega\to\R^d$
satisfy  
	$X^{\theta}_{t,s,x} = x + W^{\theta}_s - W^{\theta}_t$, 
and let 
	$U^{\theta}_{n,M,r}\colon [0,T]\times\R^d\times\Omega\to\R$, 
	$\theta\in\Theta$, 
	$n\in\N_0$,
	$M\in\N$,  
	$r\in (0,\infty)$, 
satisfy for every 
	$\theta\in\Theta$, 
	$n,M\in\N$, 
	$r\in (0,\infty)$, 
	$t\in [0,T]$, 
	$x\in \R^d$ 
that 
	$
	U^{\theta}_{0,M,r}(t,x) = 0
	$
and 
	\begin{equation}
	\begin{split}
	& U^{\theta}_{n,M,r}(t,x) 
	= 
	\frac{1}{M^n}\left[ 
	\sum_{m=1}^{M^n} 
	\left( 
	g(X^{(\theta,0,-m)}_{t,T,x}) 
	+ 
	(T-t) \, f \big(R^{(\theta,0,m)}_t, X^{(\theta,0,m)}_{t,R^{(\theta,0,m)}_t,x},0\big)
	\right)
	\right] 
	\\
	& \qquad
	+ 
	\sum_{k=1}^{n-1} \frac{(T-t)}{M^{n-k}} 
	\Bigg[ 
	\sum_{m=1}^{M^{n-k}} 
	\bigg(
	\fr\!\left( R^{(\theta,k,m)}_t, 
	X^{(\theta,k,m)}_{t,R^{(\theta,k,m)}_t,x},
	U^{(\theta,k,m)}_{k,M,r}\big( 		
	R^{(\theta,k,m)}_t, 
	X^{(\theta,k,m)}_{t,R^{(\theta,k,m)}_t,x}
	\big)\!
	\right)
	\\
	& \qquad\qquad\qquad\qquad\qquad
	-
	\fr\!\left( 
	R^{(\theta,k,m)}_t, 
	X^{(\theta,k,m)}_{t,R^{(\theta,k,m)}_t,x},
	U^{(\theta,-k,m)}_{k-1,M,r}\big( 		
	R^{(\theta,k,m)}_t, 
	X^{(\theta,k,m)}_{t,R^{(\theta,k,m)}_t,x}\big)\!
	\right)\!\!
	\bigg)\!
	\Bigg]\!.
	\end{split} 
	\end{equation}
\end{setting}

The next result, \cref{L2_estimate} below, is an adaptation of \cite[Theorem 3.5]{Overcoming} to \cref{setting}. 

\begin{lemma}[Convergence rate for stochastic fixed point equations]
	\label{L2_estimate}
Assume 
	\cref{setting}, 
let 
	$\rho\in (0,\infty)$,
let 
	$L\colon (0,\infty)\to [0,\infty)$ 
satisfy for every 
	$ r \in (0,\infty)$,
	$ t \in [0, T] $, 
	$ x \in \R^d $, 
	$ v,w \in [-r,r]$ 
that
	\begin{equation}
	\label{L2_estimate:eq0}
	|f(t, x, v) - f(t, x, w)|
	\leq
	L(r) | v - w |,
	\end{equation}
let 
	$u\in C([0,T]\times\R^d,\R)$ 
satisfy for every 
	$r\in [\rho,\infty)$, 
	$t\in [0,T]$, 
	$x\in \R^d$ 
that 
	\begin{equation}
	\begin{split} 
	& 
	\Exp{|g(X^{0}_{0,T,x})|} 
	+ 
	\int_0^T 
	\left( 
	\Exp{|u(s,X^0_{0,s,x})|^2}
	\right)^{\!\nicefrac12}\,ds
	\\
	& \quad
	+ \int_t^T 
	\Exp{\left|\fr\big(s,X^{0}_{t,s,x},u(s,X^{0}_{t,s,x})\big)\right|
		+ 
		\left| 
		f(s,X^{0}_{t,s,x},0)
		\right| 
	}\!\,ds < \infty 
	\end{split} 
	\end{equation}
	\begin{equation} 
	\text{and} \qquad 
	u(t,x) 
	= 
	\Exp{g(X^{0}_{t,T,x}) 
		+ 
		\int_t^T \fr\big( s, X^0_{t,s,x}, u(s,X^0_{t,s,x}) \big)\,ds}\!. 
	\end{equation} 
Then it holds for every 
	$n \in \N_0$, 
	$M \in \N$, 
	$r\in [\rho,\infty)$,
	$x \in \R^d$ 
that
	\begin{equation}
	\label{L2_estimate:claim}
	\begin{split}
	&
	\left(
	\Exp{|  U^0_{n,M,r} (0, x) - u(0, x)  |^2 }
	\right)^{\!\nicefrac{1}{2}}
	\\
	& \leq
	e^{L(r)T}
	\left[ 
	\left( \Exp{|g(X^{0}_{0,T,x})|^2} \right)^{\!\nicefrac12}
	+ 
	\sqrt{T}
	\left|
	\int_0^T \Exp{|f(s,X^{0}_{0,s,x},0)|^2}\!\,ds\right|^{\nicefrac12}
	\right] 
	\left[ 
	\frac{ e^{\nicefrac{M}{2}} ( 1 + 2 L(r) T )^n }{M^{\nicefrac{n}{2}}}\right]\!.
	\end{split} 
	\end{equation}
\end{lemma}

\begin{proof}[Proof of \cref{L2_estimate}]
Throughout this proof let 
	$P_r\colon \R\to \R$, $r\in (0,\infty)$, 
be the functions which satisfy for every 
	$v \in \R$ 
that 
	$P_r(v)=\min \{r,\max\{-r,v\}\}$ 
and assume w.l.o.g.~that there exists a standard Brownian motion 
	$\mathbf{W}\colon [0,T]\times\Omega\to\R^d$
which satisfies that 
	$(\mathcal{R}^{\theta})_{\theta\in\Theta}$, 
	$(W^{\theta})_{\theta\in\Theta}$, 
and 
	$\mathbf{W}$ 
are independent. Observe that for every 
	$r\in (0,\infty)$
it holds that $P_r\colon\R\to\R$ is the projection onto the 
closed convex interval $[-r,r]$. Therefore, we obtain for every 
	$r \in (0,\infty)$,
	$v,w \in \R$ 
that 
	\begin{equation}
	\vert P_r(v)-P_r(w)\vert\le \vert v-w\vert
	\end{equation}
(cf., e.g., Br\'ezis~\cite[Proposition 5.3]{Brezis}).
This, \eqref{eq:truncation}, and \eqref{L2_estimate:eq0} imply for every 
	$r\in (0,\infty)$, 
	$t \in [0, T]$, 
	$x \in \R^d$, 
	$v,w \in \R$ 
that
	\begin{equation} \label{L2_estimate:globally_lipschitz_truncations}
	\begin{split}
	|\fr(t, x, v) - \fr(t, x, w)|
	&
	=|f(t, x, P_r(v)) - f(t, x, P_r(w))|
	\\&
	\le L(r) | P_r(v) - P_r(w) |
	\le L(r) | v - w |.
	\end{split}
	\end{equation} 
This and \cite[Theorem 3.5]{Overcoming} (with 
	$d=d$, 
	$T=T$,
	$L=L(r)$,
	$\xi=x$,
	$F=(C([0,T]\times\R^d,\R) \ni v\mapsto ([0,T]\times\R^d \ni (t,x) \mapsto \fr(t,x,v(t,x))\in\R) \in C([0,T]\times\R^d,\R) )$, 
	$(\Omega,\cF,\P)=(\Omega,\cF,\P)$,
	$g=g$, 
	$u=u$,
	$\Theta = \Theta$, 
	$W^{\theta}=W^{\theta}$,
	$\mathfrak{r}^{\theta}=\mathcal{R}^{\theta}$,
	$\mathcal{R}^{\theta}=R^{\theta}$,
	$U^{\theta}_{n,M} = U^{\theta}_{n,M,r}$
for
	$\theta\in\Theta$, 
	$n\in\N_0$, 
	$M\in\N$ 	
in the notation of \cite[Theorem 3.5]{Overcoming}) ensure for every 
	$n,M\in\N$, 
	$r\in [\rho, \infty)$, 
	$x\in \R^d$
that 
	\begin{equation} \label{L2_estimate:case_N_bigger_than_zero}
	\begin{split}
	&
	\left( 
	\Exp{|U^0_{n,M,r}(0,x) - u(0,x)|^2} 
	\right)^{\!\nicefrac12}
	\\
	& \leq
	e^{L(r)T}
	\left[ 
	\left( \Exp{|g(X^{0}_{0,T,x})|^2} \right)^{\!\nicefrac12}
	+ 
	\sqrt{T}
	\left|
	\int_0^T \Exp{|f(s,X^{0}_{0,s,x},0)|^2}\!\,ds\right|^{\nicefrac12}
	\right] 
	\left[ \frac{ e^{\nicefrac{M}{2}} ( 1 + 2 L(r) T )^n }{M^{\nicefrac{n}{2}}}\right]\!.
	\end{split}
	\end{equation} 
Moreover, note that \eqref{L2_estimate:globally_lipschitz_truncations} and \cite[Lemma 3.4]{Overcoming}
	(with 
	$d=d$, 
	$T=T$,
	$L=L(r)$,
	$\xi=x$,
	$F=(C([0,T]\times\R^d,\R) \ni v\mapsto ([0,T]\times\R^d \ni (t,x) \mapsto \fr(t,x,v(t,x))\in\R) \in C([0,T]\times\R^d,\R) )$, 
	$(\Omega,\cF,\P)=(\Omega,\cF,\P)$,
	$g=g$, 
	$u=u$,
	$\Theta = \Theta$, 
	$W^{\theta}=W^{\theta}$,
	$\mathfrak{r}^{\theta}=\mathcal{R}^{\theta}$,
	$\mathcal{R}^{\theta}=R^{\theta}$,
	$U^{\theta}_{n,M} = U^{\theta}_{n,M,r}$ 
for 
	$\theta\in\Theta$, 
	$n\in\N_0$, 
	$M\in\N$ 	
in the notation of \cite[Lemma 3.4]{Overcoming}) yield that for every 
	$M\in\N$, 
	$r\in [\rho,\infty)$, 
	$x\in\R^d$  
it holds that 
	\begin{equation} 
	\begin{split}
	& 
	\left( \Exp{\left| U^{0}_{0,M,r}(0,x) - u(0,x) \right|^2} \right)^{\!\nicefrac12} 
	\\
	& \leq 
	e^{L(r)T} \left[ 
	\left( \Exp{|g(X^{0}_{0,T,x})|^2} \right)^{\!\nicefrac12} 
	+ 
	\sqrt{T} \left| \int_0^T \Exp{|f(s,X^{0}_{0,s,x},0)|^2}\!\,ds \right|^{\!\nicefrac12}
	\right]\!.
	\end{split}
	\end{equation} 
Combining this with \eqref{L2_estimate:case_N_bigger_than_zero} establishes 
\eqref{L2_estimate:claim}. 
The proof of \cref{L2_estimate} is thus completed. 
\end{proof}

%\begin{remark} 
%Note that the standard Brownian motion $\mathbf{W}\colon [0,T]\times\Omega\to\R^d$ in the proof of \cref{L2_estimate} above is introduced only pro forma. Formally, we are required to have $\mathbf{W}$ as it is mentioned in the settings of \cite[Lemma 3.4 and Theorem 3.5]{Overcoming} which we use in the proof of \cref{L2_estimate}. But checking \cite[Lemma 3.4 and Theorem 3.5]{Overcoming} one sees that $\mathbf{W}$ is not necessary at all. 
%\end{remark} 

%Next we recall for the reader's convenience a special instance of the well-known Feyman--Kac formula (cf., for example, Karatzas \& Shreve~\cite{KaSh1991_BrownianMotionAndStochasticCalculus}) in \cref{feynman_kac}. \cref{feynman_kac} will justify the application of \cref{L2_estimate} to the numerical approximation of backward Allen--Cahn PDEs (cf., for example, \eqref{eq:pde_for_u} in  \cref{thm:convergenceLocalLipschitz} below). 

\begin{lemma}[Feyman--Kac formula]\label{feynman_kac}
Let 
	$d\in\N$, 
	$T\in (0,\infty)$,  
	$u,h\in C([0,T]\times\R^d,\R)$, 
let 
	$(\Omega,\cF,\P)$ be a probability space, 
let 
	$W\colon [0,T]\times\Omega\to\R^d$ 
be a standard Brownian motion, 
for every 
	$t\in [0,T]$, 
	$s\in [t,T]$, 
	$x\in \R^d$ 
let 
	$X_{t,s,x} \colon \Omega \to \R^d$ 
satisfy
	$X_{t,s,x}=x+W_s-W_t$, 
and assume for every 
	$t\in [0,T)$, 
	$x\in \R^d$
that
	$
	\sup_{s\in [0,T],y\in\R^d} |u(s,y)| < \infty
	$,  
	$
	\EXP{\int_t^T |h(s,X_{t,s,x})|\,ds} < \infty
	$, 
	$
	u|_{[0,T)\times\R^d}\in C^{1,2}( [0,T)\times\R^d,\R)
	$, 
and 
	\begin{equation}
	(\tfrac{\partial }{\partial t}u)(t,x) 
	+ 
	\tfrac12 (\Delta_x u)(t,x) 
	+ 
	h(t,x) 
	= 
	0. 
	\end{equation}
Then it holds for every 
	$t\in [0,T]$, 
	$x\in \R^d$ 
that 
	\begin{equation}\label{feynman_kac:claim}
	u(t,x) 
	= 
	\Exp{u(T,X_{t,T,x}) 
		+ 
		\int_t^T 
		h(s,X_{t,s,x})\,ds}\!.
	\end{equation}
\end{lemma}

\begin{proof}[Proof of \cref{feynman_kac}]
Throughout this proof let 
	$\langle\cdot,\cdot\rangle\colon\R^d\times\R^d\to\R$ 
be the Euclidean scalar product on $\R^d$, 
let 
	$\Norm{\cdot}\colon \R^d \to [0,\infty)$ 
be the Euclidean norm on $\R^d$,  
and for every
	$r\in (0,\infty)$, 
	$t\in [0,T]$, 
%	$\mathfrak{t}\in [t,T]$,
	$x\in\R^d$
with 
	$t<T-\nicefrac{1}{r}$ 
let the function 
	$\tau^{t,x}_r \colon \Omega \to [t,T-\nicefrac{1}{r}]$ 
satisfy that  
	$\tau^{t,x}_r = \inf(\{s\in [t,T]\colon \Norm{X_{t,s,x}-x}>r\}\cup\{T-\nicefrac{1}{r}\})$. 
Observe that It\^o's formula and the hypothesis that $u|_{[0,T)\times\R^d}\in C^{1,2}([0,T)\times\R^d,\R)$ ensure that for every 
	$r\in (0,\infty)$,
	$t\in [0,T]$,
	$x\in\R^d$
with 
	$t<T-\nicefrac{1}{r}$ 
it holds $\P$-a.s.~that 
	\begin{equation} 
	\begin{split}
	u(\tau^{t,x}_r,X_{t,\tau^{t,x}_r,x}) 
	= 
	u(t,x) 
	+ 
	\int_t^{\tau^{t,x}_r} 
	\langle (\nabla_x u)(s,X_{t,s,x}),\,dW_s\rangle
	- 
	\int_t^{\tau^{t,x}_r} 
	h(s,X_{t,s,x})\,ds. 
	\end{split}
	\end{equation} 
This implies for every 
	$r\in (0,\infty)$,
	$t\in [0,T]$,
	$x\in\R^d$
with 
	$t<T-\nicefrac{1}{r}$
that 
	\begin{equation} \label{feynman_kac:till_stopping_time}
	u(t,x)
	= 
	\Exp{u(\tau^{t,x}_r,
		X_{t,\tau^{t,x}_r,x}) 
		+ 
		\int_t^{\tau^{t,x}_r} h( s, X_{t,s,x} )\,ds}\!.
	\end{equation} 
Combining the fact that for every 
	$t\in [0,T]$, 
	$x\in\R^d$ 
it holds $\P$-a.s.~that 
	$\limsup_{r\to\infty} |\tau^{t,x}_r-T| = 0$ 
and the hypothesis that 
	$u\colon [0,T]\times\R^d\to\R$ 
is a bounded continuous function with Lebesgue's dominated convergence theorem hence implies that for every 
	$t\in [0,T]$, 
	$x\in\R^d$
it holds that
	\begin{equation} \label{feynman_kac:terminal_data_convergence}
	\limsup_{r\to\infty}  
	\Exp{\left|u(\tau^{t,x}_r,X_{t,\tau^{t,x}_r,x}) 
		- 
		u(T,X_{t,T,x})\right|}
	= 0. 
	\end{equation} 
In addition, note that the fact that for every 
	$t\in [0,T]$, 
	$x\in \R^d$ 
it holds $\P$-a.s.~that 
	$\limsup_{r\to\infty} |\tau^{t,x,\mathfrak{t}}_r-\mathfrak{t}| = 0$,  
the hypothesis that 
	$h\colon [0,T]\times\R^d\to\R$ 
is a continuous function,
the hypothesis that for every 
	$t\in [0,T]$, 
	$x\in \R^d$ 
it holds that 
	$\int_t^T \EXP{|h(s,X_{t,s,x})|}\,ds < \infty$, 
and Lebesgue's dominated convergence theorem ensure for every 
	$t\in [0,T]$, 
	$x\in \R^d$ 
that 
	\begin{equation}
	\limsup_{r\to\infty} 
	\left| 
	\Exp{\int_t^{\tau^{t,x}_r} 
		h( s, X_{t,s,x})\,ds} 
	-\Exp{\int_t^{T} h( s, X_{t,s,x})\,ds}
	\right| 
	= 0.
	\end{equation}
This, \eqref{feynman_kac:till_stopping_time}, and \eqref{feynman_kac:terminal_data_convergence} imply for every 
	$t\in [0,T)$, 
	$x\in \R^d$ 
that 
	\begin{equation}
	\begin{split}
	u(t,x) 
	& = 
	\lim_{r \to \infty} 
	\left(
	\Exp{u(\tau^{t,x}_r,X_{t,\tau^{t,x}_r,x}) 
		+ 
		\int_t^{\tau^{t,x}_r} h(s,X_{t,s,x})\,ds}
	\right) 
	\\
	& =   
	\Exp{u(T,X_{t,T,x}) 
		+ 
		\int_t^T h(s,X_{t,s,x})\,ds}\!. 
	\end{split}
	\end{equation} 
This establishes \eqref{feynman_kac:claim}. The proof of \cref{feynman_kac} is thus completed.
\end{proof}

%The next result, \cref{thm:convergenceLocalLipschitz} below, asserts the convergence of multilevel Picard approximations to classical solutions of certain backward Allen--Cahn equations. Its proof relies on combining the a priori estimate in \cref{cor:backwardReactionDiffusionEquationWithCoerciveNonlinearity} with the error estimate in \cref{L2_estimate} and the Feyman--Kac formula in \cref{feynman_kac}. 

\begin{prop}[Convergence rate for Allen--Cahn PDEs]
	\label{thm:convergenceLocalLipschitz} 
Assume 
	Setting~\ref{setting}, 
let
	$\rho \in (0,\infty)$,
	$c\in [0,\infty)$, 
let 
	$\norm{\cdot}\colon\R^d \to [0,\infty)$ 
be a norm, 
let
	$L\colon (0,\infty)\to[0,\infty)$ 
satisfy for every 
	$r \in (0,\infty)$,
	$t \in [0, T]$, 
	$x \in \R^d$, 
	$v,w \in [-r,r]$ 
that
	\begin{equation}
	\label{convergenceLocalLipschitz:Lipschitz_ass}
	|f(t, x, v) - f(t, x, w)|
	\leq
	L(r) | v - w |,
	\end{equation}	
let 
	$u\in C([0,T]\times\R^d,\R)$ 
satisfy that 
	$	
	\inf_{a\in \R} 
	[ 
	\sup_{t\in [0,T]}\sup_{x\in\R^d} (e^{a\norm{x}^2}|u(t,x)|) 
	] 
	< \infty
	$ 
and 
	$u|_{[0,T)\times\R^d} \in C^{1,2}([0,T)\times\R^d,\R)$, 
and assume for every 
	$t\in [0,T)$, 
	$x\in\R^d$, 
	$v\in\R$
that
	$\rho\geq e^{cT} (1+\vert g(x)\vert^2)^{\nicefrac{1}{2}}$, 
	$v f(t,x,v) \leq c ( 1 + v^2 )$, 	
	$\int_t^T \EXP{|f(s,X^0_{t,s,x},0)|}\,ds<\infty$,
	$u(T,x)=g(x)$,  
and 
	\begin{equation}\label{eq:pde_for_u}
	(\tfrac{\partial }{\partial t}u)(t,x) 
	+ 
	\tfrac12 (\Delta_x u)(t,x) 
	+
	f(t,x,u(t,x))
	= 
	0. 
	\end{equation}
Then it holds for every 
	$n \in \N_0$, 
	$M \in \N$, 
	$r\in [\rho,\infty)$, 
	$x\in\R^d$
that 
	\begin{equation} \label{convergenceLocalLipschitz:claim1} 
	\begin{split}
	&
	\left(
	\Exp{|  U^0_{n,M,r} (0, x) - u(0, x)  |^2 }
	\right)^{\!\nicefrac{1}{2}}
	\\
	&
	\leq
	e^{L(r)T}
	\left[ 
	\left( \Exp{|g(X^{0}_{0,T,x})|^2} \right)^{\!\nicefrac12}
	+ 
	\sqrt{T}
	\left|
	\int_0^T \Exp{|f(s,X^{0}_{0,s,x},0)|^2}\!\,ds\right|^{\nicefrac12}
	\right]
	\left[\frac{ e^{\nicefrac{M}{2}} ( 1 + 2 L(r) T )^n }{M^{\nicefrac{n}{2}}}\right]
	\!. 
	\end{split} 
	\end{equation}
\end{prop}

\begin{proof}[Proof of \cref{thm:convergenceLocalLipschitz}] 
First, observe that the hypothesis that $\sup_{x\in\R^d} |g(x)| < \infty$ 
implies that for every 
	$x \in \R^d$ 
it holds that 
	\begin{equation}\label{cor:backwardReactionDiffusionEquationWithCoerciveNonlinearity:g_integrability}
	\Exp{|g(X^0_{0,T,x})|}<\infty. 
	\end{equation}
Next note that \cref{cor:backwardReactionDiffusionEquationWithCoerciveNonlinearity} (with $d=d$, $T=T$, $c=c$, $\norm{\cdot}=\norm{\cdot}$, $f=f$, $u=u$ in the notation of \cref{cor:backwardReactionDiffusionEquationWithCoerciveNonlinearity}) ensures for every 
	$t\in [0,T]$
that 
	\begin{equation}\label{cor:backwardReactionDiffusionEquationWithCoerciveNonlinearity:u_bound}
	\sup_{x\in\R^d} 
	|u(t,x)| 
	\leq e^{c (T-t)} 
	\left[
	1+\sup_{x\in\R^d} |u(T,x)|^2\right]^{\nicefrac{1}{2}}
	\leq 
	e^{cT} \left[ 1 + \sup_{x\in\R^d} |g(x)|^2\right]^{\nicefrac12} 
	\leq  \rho.
	\end{equation}
Combining this with \eqref{eq:truncation} yields for every 
	$r\in [\rho,\infty)$, 
	$t\in [0,T]$, 
	$x\in \R^d$ 
that 
	\begin{equation}
	\fr(t,x,u(t,x))
	= 
	f(t,x,\min\{r,\max\{-r,u(t,x)\}\}) 
	=
	f(t,x,u(t,x)).
	\end{equation}
This and \eqref{eq:pde_for_u} demonstrate that for every 
	$r\in [\rho,\infty)$, 
	$t\in [0,T)$, 
	$x\in\R^d$ 
it holds that 
	\begin{equation}\label{eq:cutoffpde_for_u}
	(\tfrac{\partial }{\partial t}u)(t,x) 
	+ 
	\tfrac12(\Delta_x u)(t,x) 
	+ 
	\fr(t,x,u(t,x))
	= 
	0. 
	\end{equation}
Next observe that the fact that $\sup_{t\in [0,T], x\in\R^d} |u(t,x)|\leq \rho$ and \eqref{convergenceLocalLipschitz:Lipschitz_ass} ensure that for every 
	$r\in [\rho,\infty)$, 
	$t\in [0,T]$, 
	$x\in \R^d$ 
it holds that 
	$
	(\EXP{|u(s,X^0_{0,s,x})|^2})^{\nicefrac12}
	\leq \rho < \infty	
	$
and
	\begin{equation}\label{FeynmanKacFormulaFRintegrability}
	\begin{split}
	&		
	\Exp{
		\int_{t}^T
		\big| \fr  \big(s, X^0_{t,s,x},  u(s, X^0_{t,s,x}) \big) \big|
		\, ds
	}
	\\
	& 
	\leq 
	\Exp{
		\int_{t}^T
		\big| \fr  \big(s, X^0_{t,s,x}, 0 \big) \big|
		\, ds
	}
	+ 
	\int_t^{T} 
	L(r) \, \Exp {|u(s, X^0_{t,s,x})|}\!\,ds 
	\\
	& \leq 
	\Exp{\int_t^T |f(s,X^0_{t,s,x},0)|\,ds} 
	+ L(r) T \rho
	<\infty. 
	\end{split}
	\end{equation}
Hence, we obtain that \eqref{eq:cutoffpde_for_u} and \cref{feynman_kac} (with 
	$d=d$, 
	$T=T$, 
	$u=u$, 
	$h=([0,T]\times\R^d\ni (t,x)\mapsto \fr(t,x,u(t,x))\in\R)$, 
	$(\Omega,\cF,\P)=(\Omega,\cF,\P)$, 
	$W=W^{0}$, 
	$X_{t,s,x}=X^{0}_{t,s,x}$ 
for 
	$t\in[0,T]$, 
	$s\in[t,T]$, 
	$x\in\R^d$ 
in the notation of \cref{feynman_kac}) demonstrate that for every 
	$ r \in [ \rho , \infty ) $, 
	$ t \in [ 0 , T ]$, 
	$ x \in \R^d $ 
it holds that 
	\begin{equation} 
	u(t,x) 
	= 
	\Exp{g(X^0_{t,T,x}) 
		+ 
		\int_t^{T} 
		\fr\big(s,X^0_{t,s,x},u(s,X^0_{t,s,x})\big)\,ds}\!.
	\end{equation} 
	\cref{L2_estimate} (with 
	$\rho=\rho$, 
	$L=L$, 
	$u=u$ 
	in the notation of \cref{L2_estimate}), \eqref{cor:backwardReactionDiffusionEquationWithCoerciveNonlinearity:g_integrability}, and \eqref{FeynmanKacFormulaFRintegrability}
	%(with 
	%	$d=d$, 
	%	$T=T$, 
	%	$\Theta=\Theta$, 
	%	$f=f$, 
	%	$g=g$, 
	%	$\fr=\fr$, 
	%	$\mathcal{R}^{\theta}=\mathcal{R}^{\theta}$, 
	%	$R^{\theta} = R^{\theta}$, 
	%	$W^{\theta} = W^{\theta}$, 
	%	$V^{\theta}_{n,M,r}=V^{\theta}_{n,M,r}$
	%) 
	hence establish~\eqref{convergenceLocalLipschitz:claim1}. 
	%Next note that Item~\eqref{convergenceLocalLipschitz:item2} is an immediate consequence of Item~\eqref{convergenceLocalLipschitz:item1}. 
	%Finally, observe that \cref{comp_cost_vs_accuracy} establishes Item~\eqref{convergenceLocalLipschitz:item2}. 
	The proof of \cref{thm:convergenceLocalLipschitz} is thus completed. 
\end{proof}

%In \cref{thm:convergenceLocalLipschitz_forward_formulation} below we translate \cref{thm:convergenceLocalLipschitz}  into the situation of forward semilinear heat equations of Allen--Cahn type. In particular, a truncated full-history recursive multilevel Picard scheme for semilinear Allen--Cahn equations in forward formulation is given in \eqref{convergenceLocalLipschitz_forward_formulation:transformed_mlp} in \cref{thm:convergenceLocalLipschitz_forward_formulation}. 

\begin{prop} \label{thm:convergenceLocalLipschitz_forward_formulation}
Let
	$d\in\N$,
	$\rho,T\in (0,\infty)$,
	$c\in [0,\infty)$,
	$\Theta = \cup_{n\in\N} \Z^n$,
	$f\in C([0,T]\times\R^d\times\R,\R)$, 
	$(\fr)_{r\in (0,\infty)} \subseteq C([0,T]\times\R^d\times\R,\R)$,  
	$u \in C([0,T]\times\R^d,\R)$, 
let 
	$\norm{\cdot}\colon \R^d \to [0,\infty)$ 
be a norm on $\R^d$,
let 
	$L\colon (0,\infty)\to [0,\infty)$ 
be a function, 
let 
	$(\Omega,\mathcal{F},\P)$ 
be a probability space, 
let
	$\mathcal{R}^{\theta}\colon \Omega \to [0,1]$, $\theta\in\Theta$, 
be independent $\mathcal{U}_{[0,1]}$-distributed random variables, 
let 
	$W^{\theta}\colon [0,T]\times\Omega\to\R^d$, $\theta\in\Theta$,
be independent standard Brownian motions, 
assume that 
	$(\mathcal{R}^{\theta})_{\theta\in\Theta}$
and
	$(W^{\theta})_{\theta\in\Theta}$ 
are independent, 
let 
	$R^{\theta}\colon [0,T]\times\Omega \to [0,T]$, $\theta\in\Theta$, 
satisfy for every 
	$\theta\in\Theta$, 
	$t\in [0,T]$
that 
	$R^{\theta}_t = t \mathcal{R}^{\theta}$,
for every 
	$\theta\in \Theta$,	
	$s\in [0,T]$, 
	$t\in [s,T]$,
	$x\in \R^d$
let 
	$X^{\theta}_{s,t,x}\colon\Omega\to\R^d$
satisfy 
	$
	X^{\theta}_{s,t,x} 
	= 
	x + \sqrt{2}
	(W^{\theta}_t - W^{\theta}_s)
	$,
assume for every 
	$r\in (0,\infty)$,
	$t\in (0,T]$, 
	$x\in \R^d$, 
	$v\in \R$, 
	$w,\mathfrak{w}\in [-r,r]$ 
that 
	$vf(t,x,v) \leq c (1+v^2)$,
	$|f(t,x,w) - f(t,x,\mathfrak{w})|\leq L(r)|w-\mathfrak{w}|$,
	$\int_0^t \EXP{|f(s,X^{0}_{s,t,x},0)|}\,ds < \infty$,
	$\fr(t,x,v) = f(t,x,\min\{r,\max\{-r,v\}\})$,  
	$e^{cT}(1+ |u(0,x)|^2)^{\nicefrac12} \leq \rho$, 
	$u|_{(0,T]\times\R^d}\in C^{1,2}((0,T]\times\R^d,\R)$,
	$\inf_{a\in \R} [ 
	\sup_{s\in [0,T]}\sup_{y\in\R^d} (e^{a\norm{y}^2}|u(s,y)|) 
	] 
	< \infty
	$, 
and 
\begin{equation}
\label{convergenceLocalLipschitz_forward_formulation:assumptions_on_u}
(\tfrac{\partial}{\partial t}u)(t,x) 
= 
(\Delta_x u)(t,x) 
+ 
f(t,x,u(t,x)),  
\end{equation}
and let 
	$U^{\theta}_{n,M,r}\colon [0,T]\times\R^d\times\Omega\to\R$, 
	$\theta\in\Theta$,  
	$n\in\N_0$, 
	$M\in\N$,
	$r\in (0,\infty)$, 
satisfy for every 
	$\theta\in\Theta$,	
	$n,M\in\N$, 
	$r\in (0,\infty)$,
	$t\in [0,T]$, 
	$x\in \R^d$ 
that
	$U^{\theta}_{0,M,r}(t,x) = 0$
and
	\begin{equation} \label{convergenceLocalLipschitz_forward_formulation:transformed_mlp} 
	\begin{split}
	& U^{\theta}_{n,M,r}(t,x) 
	= 
	\frac{1}{M^n}\left[ 
	\sum_{m=1}^{M^n} 
	\left( 
	u(0,X^{(\theta,0,-m)}_{0,t,x}) 
	+ 
	t \, f\big(R^{(\theta,0,m)}_{t}, 
	X^{(\theta,0,m)}_{R^{(\theta,0,m)}_{t},t,x},0\big)
	\right)
	\right] 
	\\ 
	& \qquad 
	+ 
	\sum_{k=1}^{n-1} \frac{t}{M^{n-k}} 
	\Bigg[ 
	\sum_{m=1}^{M^{n-k}} 
	\bigg(
	\fr\Big( 
	R^{(\theta,k,m)}_{t}, 
	X^{(\theta,k,m)}_{R^{(\theta,k,m)}_{t},t,x},
	U^{(\theta,k,m)}_{k,M,r}\big(
	R^{(\theta,k,m)}_{t}, 
	X^{(\theta,k,m)}_{R^{(\theta,k,m)}_{t},t,x}
	\big)
	\Big)
	\\ 
	& \qquad\qquad\qquad\qquad\quad
	-
	\fr\Big(R^{(\theta,k,m)}_{t}, 
	X^{(\theta,k,m)}_{R^{(\theta,k,m)}_{t},t,x}, 
	U^{(\theta,-k,m)}_{k-1,M,r}\big(R^{(\theta,k,m)}_{t}, 
	X^{(\theta,k,m)}_{R^{(\theta,k,m)}_{t},t,x}\big)
	\Big)
	\bigg)
	\Bigg]\!.
	\end{split}
	\end{equation}
Then it holds for every 
	$n\in\N_0$, 
	$M\in\N$, 
	$r\in [\rho,\infty)$, 
	$x\in\R^d$
that 
	\begin{equation} \label{convergenceLocalLipschitz_forward_formulation:claim}
	\begin{split}
	&
	\left(\Exp{
		|
		U^{0}_{n,M,r}(T,x)
		-
		u(T,x)
		|^2}\right)^{\!\nicefrac12}
	\leq 
	e^{L(r)T} 
	\left[
	\frac{e^{\nicefrac{M}{2}}(1+2L(r)T)^n
	}{M^{\nicefrac{n}{2}}}
	\right]
	\\
	& \qquad \qquad 
	\cdot 
	\left[ 
	\left(
	\Exp{|u(0,X^{0}_{0,T,x})|^2} 
	\right)^{\!\nicefrac12} + \sqrt{T}\left| 
	\int_0^T \Exp{|f(s,X^{0}_{s,T,x},0)|^2}\!\,ds \right|^{\nicefrac12}
	\right]\!.
	\end{split} 
	\end{equation} 
\end{prop}

\begin{proof}[Proof of \cref{thm:convergenceLocalLipschitz_forward_formulation}]
Throughout this proof let 
	$v \colon [0,T]\times\R^d \to \R$
be the function which satisfies for every 
	$t\in [0,T]$, 
	$x\in \R^d$ 
that 
	$v(t,x) = u(T-t,x\sqrt{2})$, 
let 
	$F \colon [0,T]\times\R^d\times\R \to \R$ 
be the function which satisfies for every 
	$t \in [0,T]$, 
	$x \in \R^d$, 
	$w \in \R$ 
that 
	$F(t,x,w) = f(T-t,x\sqrt{2},w)$,
let 
	$\mathbf{F}_{r}\colon [0,T]\times\R^d\times\R \to \R$,  $r\in (0,\infty)$, 
be the functions which satisfy for every 
	$r \in (0,\infty)$,
	$t \in [0,T]$, 
	$x \in \R^d$, 
	$w \in \R$ 
that 
	$\mathbf{F}_{r}(t,x,w) = F(t,x,\min\{r,\max\{-r,w\}\})$, 
let 
	$\mathcal{S}^{\theta}\colon\Omega\to [0,1]$, $\theta\in\Theta$, 
satisfy for every 
	$\theta\in\Theta$ 
that 
	$\mathcal{S}^{\theta} = 1-\mathcal{R}^{\theta}$, 
let 
	$S^{\theta}\colon [0,T]\times\Omega \to [0,T]$, $\theta\in\Theta$, 
satisfy for every 
	$\theta\in\Theta$,
	$t\in [0,T]$
that 
	$S^{\theta}_t = t + (T-t)\mathcal{S}^{\theta}$, 
for every 
	$\theta\in\Theta$, 
	$t\in [0,T]$, 
	$s\in [t,T]$, 
	$x\in \R^d$ 
let 
	$Y^{\theta}_{t,s,x}\colon \Omega \to \R^d$ 
satisfy that 
	$Y^{\theta}_{t,s,x} = \frac{1}{\sqrt{2}}X^{\theta}_{T-s,T-t,x\sqrt{2}}
	= x + W^{\theta}_{T-t} - W^{\theta}_{T-s}
	= x + (W^{\theta}_T-W^{\theta}_{T-s}) - (W^{\theta}_{T}-W^{\theta}_{T-t})$, 
and let 
	$V^{\theta}_{n,M,r}\colon [0,T]\times\R^d\times\Omega \to \R$, 
	$\theta\in\Theta$, $n\in\N_0$, $M\in\N$, $r\in (0,\infty)$ 
satisfy for every 
	$n,M\in\N$,
	$\theta\in\Theta$, 
	$r\in (0,\infty)$, 
	$t\in [0,T]$, 
	$x\in \R^d$ 
that
	$V^{\theta}_{n,M,r}(t,x) = U^{\theta}_{n,M,r}(T-t,x\sqrt{2})$. 
Note that \eqref{convergenceLocalLipschitz_forward_formulation:assumptions_on_u} hence ensures for every 
	$t\in [0,T)$, 
	$x\in \R^d$ 
that 
	$v\in C([0,T]\times\R^d,\R)$, 
	$v|_{[0,T)\times\R^d}\in C^{1,2}([0,T)\times\R^d,\R)$, 
	$\inf_{a\in\R}[\sup_{(s,y)\in [0,T]\times\R^d} 
	(e^{a\norm{y}^2}|v(s,y)|)] < \infty$, 
and 
\begin{equation} \label{convergenceLocalLipschitz_forward_formulation:backward_solution}
	\begin{split}
	& 
	(\tfrac{\partial }{\partial t}v)(t,x) 
	+ 
	\tfrac12 
	(\Delta_x v)(t,x) 
	+ 
	F(t, x, v(t,x)) 
	\\
	& 
	= 
	-(\tfrac{\partial u}{\partial t})(T-t,x\sqrt{2}) 
	+ 
	(\Delta_x u)(T-t,x\sqrt{2}) 
	+ 
	f(T-t,x\sqrt{2},u(T-t,x\sqrt{2}))
	= 0.
	\end{split}
	\end{equation} 
In addition, note that the hypothesis that for every 
	$t\in [0,T]$, 
	$x\in \R^d$, 
	$w\in \R$ 
it holds that 
	$wf(t,x,w)\leq c(1+w^2)$ 
guarantees that for every 
	$t\in [0,T]$, 
	$x\in \R^d$, 
	$w\in \R$ 
it holds that 
	\begin{equation} \label{convergenceLocalLipschitz_forward_formulation:check_coercivity}
	wF(t,x,w) = w f(T-t,x\sqrt{2},w) \leq c (1 + w^2). 
	\end{equation} 
Moreover, observe that it holds for every 
	$t\in [0,T]$,
	$\theta\in\Theta$
that  
	\begin{equation} \label{R_and_S}
	S^{\theta}_t = 
	t + (T-t)\mathcal{S}^{\theta} 
	= 
	t + (T-t)(1-\mathcal{R}^{\theta}) 
	= 
	T - (T-t) \mathcal{R}^{\theta} 
	= 
	T - {R}^{\theta}_{T-t}. 
	\end{equation} 
Next observe that the assumption that for every 
	$r\in (0,\infty)$, 
	$t\in [0,T]$, 
	$x\in \R^d$, 
	$w,\mathfrak{w}\in [-r,r]$ 
it holds that 
$|f(t,x,w)-f(t,x,\mathfrak{w})| \leq L(r)|w-\mathfrak{w}|$ 
implies that for every 
	$r\in (0,\infty)$, 
	$t\in [0,T]$, 
	$x\in \R^d$, 
	$w,\mathfrak{w}\in [-r,r]$ 
it holds that 
	\begin{equation} \label{convergenceLocalLipschitz_forward_formulation:check_lipschitz_condition}
	|F(t,x,w)-F(t,x,\mathfrak{w})| 
	= 
	|f(T-t,x\sqrt{2},w)-f(T-t,x\sqrt{2},\mathfrak{w})| 
	\leq 
	L(r) |w-\mathfrak{w}|. 
	\end{equation} 
In addition, note that 
	\begin{equation} 
	e^{cT}\left[1+\sup_{x\in\R^d} |v(T,x)|^2\right]^{\nicefrac12} 
	= 
	e^{cT}\left[1+\sup_{x\in\R^d} 
	|u(0,x\sqrt{2})|^2\right]^{\nicefrac12}
	\leq \rho.
	\end{equation} 
Furthermore, note that for every 
	$t\in [0,T]$, 
	$x\in \R^d$
it holds that 
	\begin{equation} \label{convergenceLocalLipschitz_forward_formulation:condition_check1}
	\begin{split}
	& 
	\int_t^T 
	\Exp{|F(s,Y^{0}_{t,s,x},0)|}\!\,ds 
	= 
	\int_0^{T-t} 
	\Exp{|F(T-s,Y^{0}_{t,T-s,x},0)|}\!\,ds 
	\\
	& =
	\int_0^{T-t} 
	\Exp{|f(s,\sqrt{2}Y^{0}_{t,T-s,x},0)|}\!\,ds
	=  
	\int_0^{T-t} 
	\Exp{|f(s,X^{0}_{s,T-t,x\sqrt{2}},0)|}\!\,ds
	< 
	\infty 
	\end{split}
	\end{equation}
and 
	\begin{equation} \label{convergenceLocalLipschitz_forward_formulation:condition_check2}
	\begin{split}
	& \int_0^T \Exp{|F(s,Y^{0}_{0,s,\nicefrac{x}{\sqrt{2}}},0)|^2}\!\,ds 
	= 
	\int_0^T \Exp{|f(T-s,\sqrt{2}\,Y^{0}_{0,s,\nicefrac{x}{\sqrt{2}}},0)|^2}\!\,ds 
	\\
	& = 
	\int_0^T 
	\Exp{|f(T-s,X^{0}_{T-s,T,x},0)|^2}\!\,ds
	= 
	\int_0^T 
	\Exp{|f(s,X^{0}_{s,T,x},0)|^2}\!\,ds.  
	\end{split}
	\end{equation} 
Moreover, observe that \eqref{convergenceLocalLipschitz_forward_formulation:transformed_mlp} guarantees for every 
	$n,M\in\N$, 
	$\theta\in\Theta$, 
	$r\in (0,\infty)$, 
	$t\in [0,T]$, 
	$x\in \R^d$ 
that 
	\begin{equation}
	\begin{split}
	& 
	U^{\theta}_{n,M,r}(T-t,x\sqrt{2}) 
	\\ 
	& = 
	\frac{1}{M^n} 
	\left[ 
	\sum_{m=1}^{M^n} 
	\left(  u\big(0,X^{(\theta,0,-m)}_{0,T-t,x\sqrt{2}}\big) 
	+ 
	(T-t) \, f\big( R^{(\theta,0,m)}_{T-t},  X^{(\theta,0,m)}_{R^{(\theta,0,m)}_{T-t},T-t,x\sqrt{2}}, 0 \big) 
	\right) 
	\right] 
	\\ 
	& + 
	\sum_{k=1}^{n-1} 
	\frac{T-t}{M^{n-k}} 
	\bigg[  
	\sum_{m=1}^{M^{n-k}} 
	\bigg( 
	\fr\Big( 
	R^{(\theta,k,m)}_{T-t}, X^{(\theta,k,m)}_{R^{(\theta,k,m)}_{T-t},T-t,x\sqrt{2}}, U^{(\theta,k,m)}_{k,M,r}\big(
		R^{(\theta,k,m)}_{T-t}, X^{(\theta,k,m)}_{R^{(\theta,k,m)}_{T-t},T-t,x\sqrt{2}}\big)
	\Big) 
	\\ 
	& 
	\qquad \qquad \qquad \qquad 
	- 
	\fr\Big( 
	R^{(\theta,k,m)}_{T-t},
	X^{(\theta,k,m)}_{R^{(\theta,k,m)}_{T-t},T-t,x\sqrt{2}}, U^{(\theta,-k,m)}_{k-1,M,r}\big(R^{(\theta,k,m)}_{T-t},  X^{(\theta,k,m)}_{R^{(\theta,k,m)}_{T-t},T-t,x\sqrt{2}}\big)
	\Big)
	\bigg) 
	\Bigg]\!.   
	\end{split}
	\end{equation}
The fact that for every 
	$\theta\in\Theta$,
	$t\in [0,T]$, 
	$s\in [t,T]$, 
	$x\in \R^d$ 
it holds that 
	$X^{\theta}_{t,s,x\sqrt{2}} 
	= 
	\sqrt{2}Y^{\theta}_{T-s,T-t,x}$
and \eqref{R_and_S} therefore imply that for every 
	$n,M\in\N$, 
	$\theta\in\Theta$, 
	$r\in (0,\infty)$, 
	$t\in [0,T]$, 
	$x\in \R^d$ 
it holds that 
	\begin{equation}
	\begin{split}
	& 
	U^{\theta}_{n,M,r}(T-t,x\sqrt{2}) 
	\\ 
	& = 
	\frac{1}{M^n} 
	\left[ 
	\sum_{m=1}^{M^n} 
	\left(  u\big(0,\sqrt{2}\,Y^{(\theta,0,-m)}_{t,T,x}\big) 
	+ 
	(T-t) \, f\big( T-S^{(\theta,0,m)}_{t},  \sqrt{2}\,Y^{(\theta,0,m)}_{t,S^{(\theta,0,m)}_{t},x}, 0 \big) 
	\right) 
	\right] 
	\\ 
	& + 
	\sum_{k=1}^{n-1} 
	\tfrac{T-t}{M^{n-k}} 
	\bigg[  
	\sum_{m=1}^{M^{n-k}} 
	\bigg( 
	\fr\Big( 
	T-S^{(\theta,k,m)}_{t}, \sqrt{2}\,Y^{(\theta,k,m)}_{t,S^{(\theta,k,m)}_{t},x}, U^{(\theta,k,m)}_{k,M,r}\big(T-S^{(\theta,k,m)}_{t}, \sqrt{2}\,Y^{(\theta,k,m)}_{t,S^{(\theta,k,m)}_{t},x}\big)
	\Big) 
	\\ 
	& 
	\qquad \qquad \qquad
	- 
	\fr\Big( 
	T-S^{(\theta,k,m)}_{t},
	\sqrt{2}\,Y^{(\theta,k,m)}_{t,S^{(\theta,k,m)}_{t},x}, U^{(\theta,-k,m)}_{k-1,M,r}\big(T-S^{(\theta,k,m)}_{t},  \sqrt{2}\,Y^{(\theta,k,m)}_{t,S^{(\theta,k,m)}_{t},x}\big)
	\Big)
	\bigg) 
	\Bigg].   
\end{split}
\end{equation}
Combining this with \eqref{convergenceLocalLipschitz_forward_formulation:transformed_mlp} and the fact that for every 
	$M\in\N$, 
	$\theta\in\Theta$, 
	$n\in\N_0$, 
	$r\in (0,\infty)$, 
	$t\in [0,T]$, 
	$x\in \R^d$ 
it holds that 
	$V^{\theta}_{n,M,r}(t,x) = U^{\theta}_{n,M,r}(T-t,x\sqrt{2})$
yields hat for every 
	$\theta\in\Theta$, 
	$n,M\in\N$, 
	$r\in (0,\infty)$, 
	$t\in [0,T]$, 
	$x\in \R^d$ 
it holds that 
	$V^{\theta}_{0,M,r}(t,x) = 0$ 
and 
	\begin{equation} 
	\begin{split}
	& 
	V^{\theta}_{n,M,r}(t,x) 
	\\
	& = 
	\frac{1}{M^n} 
	\left[ 
	\sum_{m=1}^{M^n} 
	\left(  u\big(0,\sqrt{2}\,Y^{(\theta,0,-m)}_{t,T,x}\big) 
	+ 
	(T-t) \, f\big( T-S^{(\theta,0,m)}_{t},  \sqrt{2}\,Y^{(\theta,0,m)}_{t,S^{(\theta,0,m)}_{t},x}, 0 \big) 
	\right) 
	\right] 
	\\ 
	& + 
	\sum_{k=1}^{n-1} 
	\tfrac{T-t}{M^{n-k}} 
	\bigg[  
	\sum_{m=1}^{M^{n-k}} 
	\bigg( 
	\fr\Big( 
	T-S^{(\theta,k,m)}_{t},
	\sqrt{2}\,Y^{(\theta,k,m)}_{t,S^{(\theta,k,m)}_{t},x}, V^{(\theta,k,m)}_{k,M,r}\big(S^{(\theta,k,m)}_{t}, Y^{(\theta,k,m)}_{t,S^{(\theta,k,m)}_{t},x}\big)
	\!\Big) 
	\\ 
	& 
	\qquad \qquad \qquad 
	- 
	\fr\Big( 
	T-S^{(\theta,k,m)}_{t}, \sqrt{2}\,Y^{(\theta,k,m)}_{t,S^{(\theta,k,m)}_{t},x}, V^{(\theta,-k,m)}_{k-1,M,r}\big(S^{(\theta,k,m)}_{t}, Y^{(\theta,k,m)}_{t,S^{(\theta,k,m)}_{t},x}\big)
	\Big)\!
	\bigg) \!
	\Bigg].   
	\end{split}
	\end{equation} 
This and the fact that for every 
	$r \in (0,\infty)$, 
	$t \in [0,T]$, 
	$x \in \R^d$, 
	$w \in \R$ 
it holds that  
	$u(0,x\sqrt{2})=v(T,x)$
and
	$\mathbf{F}_{r}(t,x,w) 
	= 
	F(t,x,\min\{r,\max\{-r,w\}\})
	= 
	f(T-t,x\sqrt{2},\min\{r,\max\{-r,w\}\}) 
	= 
	\fr(T-t,x\sqrt{2},w)$ 
demonstrate that for every 
	$\theta\in\Theta$,  
	$n,M\in\N$,
	$r\in (0,\infty)$, 
	$t\in [0,T]$, 
	$x\in \R^d$ 
it holds that 
	$V^{\theta}_{0,M,r}(t,x) = 0$ 
and 
	\begin{equation} \label{main_theorem_forward_formulation:mlp_transformed_back}
	\begin{split}
	%& 
	V^{\theta}_{n,M,r}(t,x) 
%	\\
	& = 
	\frac{1}{M^n} 
	\left[ 
	\sum_{m=1}^{M^n} 
	\left( 
	v(T,Y^{(\theta,0,-m)}_{t,T,x}) 
	+ 
	(T-t) \,F(S^{(\theta,0,m)}_{t}, Y^{(\theta,0,m)}_{t,S^{(\theta,0,m)}_{t},x}, 0)
	\right) 
	\right] 
	\\
	& + 
	\sum_{k=1}^{n-1} 
	\frac{(T-t)}{M^{n-k}} \Bigg[ 
	\sum_{m=1}^{M^{n-k}} \bigg(
	\mathbf{F}_{r}\!\left(S^{(\theta,k,m)}_{t}, Y^{(\theta,k,m)}_{t,S^{(\theta,k,m)}_t,x}, 
	V^{(\theta,k,m)}_{k,M,r}\big(S^{(\theta,k,m)}_{t}, Y^{(\theta,k,m)}_{t,S^{(\theta,k,m)}_t,x}\big)
	\right)
	\\
	& \qquad \qquad \qquad \quad 
	- 
	\mathbf{F}_{r}\!\left(S^{(\theta,k,m)}_{t}, Y^{(\theta,k,m)}_{t,S^{(\theta,k,m)}_t,x}, 
	V^{(\theta,-k,m)}_{k-1,M,r}\big(S^{(\theta,k,m)}_{t}, Y^{(\theta,k,m)}_{t,S^{(\theta,k,m)}_t,x}\big)\!
	\right)\!
	\bigg)
	\Bigg].
	\end{split}
	\end{equation} 
This, \eqref{convergenceLocalLipschitz_forward_formulation:backward_solution}--\eqref{convergenceLocalLipschitz_forward_formulation:condition_check2}, and \cref{thm:convergenceLocalLipschitz} (with 
	$d=d$, 
	$T=T$, 
	$\Theta=\Theta$, 
	$f=F$, 
	$g=(\R^d\ni x \mapsto v(T,x)=u(0,x\sqrt{2})\in\R)$,
	$\fr=\mathbf{F}_{r}$, 
	$(\Omega,\cF,\P)=(\Omega,\cF,\P)$, 
	$\mathcal{R}^{\theta}=\mathcal{S}^{\theta}$, 
	$W^{\theta}=([0,T]\times\Omega\ni (t,\omega) \mapsto W^{\theta}_T(\omega) - W^{\theta}_{T-t}(\omega) \in \R^d)$, 
	$X^{\theta}_{t,s,x}=Y^{\theta}_{t,s,x}$, 
	$R^{\theta}=S^{\theta}$,
	$U^{\theta}_{n,M,r}=V^{\theta}_{n,M,r}$,
	$\rho=\rho$, 
	$c=c$, 
	$\norm{\cdot}=\norm{\cdot}$, 
	$L=L$, 
	$u=v$
for 
	$\theta\in\Theta$, 
	$r\in (0,\infty)$, 
	$t\in [0,T]$, 
	$s\in [t,T]$
	$x\in \R^d$
in the notation of \cref{thm:convergenceLocalLipschitz}) demonstrate that for every 
	$n\in\N_0$, 
	$M\in\N$, 
	$r\in [\rho,\infty)$, 
	$x\in\R^d$
it holds that 
	\begin{equation} 
	\begin{split}
	&
	\left(
	\Exp{|  V^{0}_{n,M,r} (0, \nicefrac{x}{\sqrt{2}}) - v(0, \nicefrac{x}{\sqrt{2}})  |^2 }
	\right)^{\!\nicefrac{1}{2}}
	\leq
	\frac{ e^{\nicefrac{M}{2}} ( 1 + 2 L(r) T )^n }{M^{\nicefrac{n}{2}}}
	\\
	& \quad
	\cdot
	e^{L(r)T}
	\left[ 
	\left( \Exp{|v(T,Y^{0}_{0,T,\nicefrac{x}{\sqrt{2}}})|^2} \right)^{\!\nicefrac12}
	+ 
	\sqrt{T}
	\left|
	\int_0^T \Exp{|F(s, Y^{0}_{0,s,\nicefrac{x}{\sqrt{2}}},0)|^2}\!\,ds\right|^{\nicefrac12}
	\right]\!.  
	\end{split} 
	\end{equation}	
Combining this with 	\eqref{convergenceLocalLipschitz_forward_formulation:condition_check2} and the fact that for every 
	$\theta\in\Theta$, 
	$n\in\N_0$, 
	$M\in\N$,
	$r\in (0,\infty)$,
	$t\in [0,T]$, 
	$x\in\R^d$ 
it holds that 
	$u(t,x) = v(T-t,\nicefrac{x}{\sqrt{2}})$, 
	$U^{\theta}_{n,M,r}(t,x) = V^{\theta}_{n,M,r}(T-t,\nicefrac{x}{\sqrt{2}})$, 
and 
	$\sqrt{2}\,Y^{0}_{0,T,\nicefrac{x}{\sqrt{2}}} = X^{0}_{0,T,x}$
%	$F_d(t,x,0) = f_d(T-t,x\sqrt{2},0)$ 
establishes \eqref{convergenceLocalLipschitz_forward_formulation:claim}. The proof of \cref{thm:convergenceLocalLipschitz_forward_formulation} is thus completed.
\end{proof}

\section{Computational cost analysis for truncated MLP approximations}
\label{sec:complexity}

Our next goal is to estimate the overall complexity of the MLP approximation scheme. This is achieved in \cref{thm:main_theorem_forward_formulation} below. %For this we compare the cost of computing multilevel Picard approximations for reaction-diffusion PDEs (see \cref{lemma:comp_cost} below) with the accuracy obtained by such approximations (see \cref{thm:convergenceLocalLipschitz_forward_formulation} in \cref{sec:algorithm} above). 
We first quote an elementary result (see \cite[Lemma 3.6]{Overcoming}) which provides a bound for the computational cost.  \cref{comp_error_vs_comp_cost}--\cref{sum_of_costs} are technical statements needed for the proof of \cref{thm:main_theorem_forward_formulation}. 

\begin{lemma}[Computational cost] \label{lemma:comp_cost}
Let 
	$d\in\N$, 
	$(\mathfrak{C}_{n,M})_{n\in\N_0,M\in\N}\subseteq\N_0$ 
satisfy for every
	$n,M\in\N$ 
that 
	$\mathfrak{C}_{0,M} = 0$ 
and 
	\begin{equation}
	\mathfrak{C}_{n,M} 
	\leq 
	(2d+1) M^n 
	+ 
	\sum_{l=1}^{n-1} M^{n-l}\left(d+1+\mathfrak{C}_{l,M}+\mathfrak{C}_{l-1,M}\right). 
	\end{equation}
Then it holds for every 
	$n,M\in\N$ 
that 
	$\mathfrak{C}_{n,M} \leq d(5M)^n$. 
\end{lemma}

\begin{proof}[Proof of \cref{lemma:comp_cost}]
This is an immediate consequence of \cite[Lemma 3.6]{Overcoming} (with 
	$d=d$, 
	$RV_{n,M}=\mathfrak{C}_{n,M}$ 
for $n\in\N_0$, $M\in\N$ in the notation of \cite[Lemma 3.6]{Overcoming}). 
The proof of \cref{lemma:comp_cost} is thus completed. 
\end{proof} 

%The next result, \cref{comp_error_vs_comp_cost} below, is another elementary technical result which will be used in the proof of \cref{thm:main_theorem_forward_formulation} below.  

\begin{lemma} \label{comp_error_vs_comp_cost}
Let 
	$\alpha,\beta,c,\kappa,\rho \in (0,\infty)$, 
	$K\in\N_0$,
	$(\gamma_n)_{n\in\N}\subseteq [0,\infty)$, 
	$(\epsilon_{n,r})_{n\in\N,r\in [\rho,\infty)} \subseteq [0,\infty)$, 
let 
	$L\colon (0,\infty) \to [0,\infty)$ 
be a function, 
assume for every 
	$n\in\N$, 
	$r\in [\rho,\infty)$
that
	$\gamma_n \leq (\alpha n)^n$ 
and 
	$\epsilon_{n,r} \leq c e^{L(r)} \kappa^n(1+\beta L(r))^n n^{-\nicefrac{n}{2}}$, 
and let 
	$\varrho\colon\N\to (0,\infty)$
satisfy that 
	\begin{equation} \label{comp_error_vs_comp_cost:growth_ass}
		\limsup_{ n\to\infty} \left[ \frac{L(\varrho_n)}{\ln(n)} + \frac{1}{\varrho_n} \right] = 0. 
	\end{equation} 
Then there exist 
	$ \mathfrak{N}\colon (0,1] \to \N$ 
and 
	$ \mathfrak{c} \colon (0,\infty) \to [0,\infty) $
such that for every 
	$\delta \in (0,\infty)$, 
	$\varepsilon \in (0,1]$
it holds that 
	$ 
	\sup_{n\in [1,\mathfrak{N}_{\varepsilon}+K]\cap\N} \gamma_{n} 
	\leq 
	\mathfrak{c}_{\delta} \varepsilon^{-(2+2\delta)} $ 
	and 
	$
	\sup_{n\in [\mathfrak{N}_{\varepsilon},\infty)\cap\N} \epsilon_{n,\varrho_{n}} \leq \varepsilon $. 
\end{lemma}

\begin{proof}[Proof of \cref{comp_error_vs_comp_cost}]
Throughout this proof let 
	$\mathfrak{a}_{\delta} \in [0,\infty]$, $\delta\in (0,\infty)$,
and 
	$\mathfrak{b} \in [0,\infty)$
satisfy for every 
	$\delta\in (0,\infty)$ 
that 
	\begin{equation} \label{comp_error_vs_comp_cost:definition_of_a_delta}
	\mathfrak{a}_{\delta} 
	=
	c^{2+2\delta} 
	\sup_{n\in\N} 
	\left[ 
	\frac{[\max\{\alpha,1\} (n+1)]^{(n+1)}}{n^{n(1+\delta)}}
	e^{L(\varrho_n)(2+2\delta)}[\kappa(1+\beta L(\varrho_n))]^{n(2+2\delta)}
	\right]
	\end{equation}
and 
\begin{equation} \label{comp_error_vs_comp_cost:definition_of_b_delta}
	\mathfrak{b}
	= 
	[\max\{\alpha,1\}(K+1)]^{(K+1)}. 
	\end{equation} 
First, observe that the fact that for every 
	$t\in (0,\infty)$ 
it holds that 
	$\ln(t)\leq t-1$ 
and 
	\eqref{comp_error_vs_comp_cost:growth_ass} 
ensure that 	
	\begin{equation} 
	\begin{split}
	& 
	\limsup_{ n \to \infty } 
	\Big[ 
	\!\ln\!\left( ce^{L(\varrho_n)} \kappa^n( 1 + \beta L(\varrho_n))^n n^{-\nicefrac{n}{2}} \right)\!
	\Big]
	\\
	& = 
	\limsup_{ n \to \infty }
	\left[ \ln(c) + L(\varrho_n) 
	+ n \ln( \kappa )
	+ n \ln( 1+\beta L(\varrho_n) )
	- \frac{n}{2}\ln(n) \right] 
	\\
	& 
	\leq  
	\limsup_{ n \to \infty } 
	\left[ \ln(c) + L(\varrho_n) 
	+ n \ln( \kappa )
	+ n \beta L(\varrho_n)  
	- \frac{n}{2}\ln(n) \right] 
	\\
	& = 
	\limsup_{n\to\infty} 
	\left[ n \ln(n) 
	\left(
	\frac{\ln(c)}{n \ln(n)} 
	+ 
	\frac{L(\varrho_n)}{n \ln(n)} 
	+ 
	\frac{\ln(\kappa)}{\ln(n)}
	+ 
	\frac{\beta L(\varrho_n)}{\ln(n)} 
	- 
	\frac12
	\right)
	\right] 
	%	\\ & 
	= -\infty. 
	\end{split}
	\end{equation} 
This and the fact that 
	$ \lim_{s\to -\infty} e^{s} = 0 $ 
imply that 
	\begin{equation} 
	\begin{split}
	0 
	& 
	\leq 
	\limsup_{n\to\infty} \left[ ce^{L(\varrho_n)} \kappa^n(1+\beta L(\varrho_n))^n n^{-\nicefrac{n}{2}} \right] 
	\\
	& 
	= \limsup_{n\to\infty} 
	\left[ 
	\exp\!\left( 
	\ln\!\left( ce^{L(\varrho_n)} \kappa^n ( 1 + \beta L(\varrho_n) )^n n^{-\nicefrac{n}{2}} \right)
	\right)
	\right] 
	=  0. 
	\end{split}
	\end{equation}
Hence, we obtain that there exist
	$N_{\varepsilon} \in \N$, $ \varepsilon \in (0,\infty) $, 
which satisfy for every 
	$\varepsilon \in (0,\infty)$ 
that 
	\begin{equation} \label{comp_error_vs_comp_cost:definition_of_N_epsilon}
	N_{\varepsilon} 
	= 
	\min\!\left\{ n\in\N\colon \sup_{m\in [n,\infty)\cap\N} \left[ ce^{L(\varrho_m)} \kappa^m (1+\beta L(\varrho_m))^m m^{-\nicefrac{m}{2}} \right] \leq \varepsilon \right\}\!.
	\end{equation} 	
Moreover, the assumption that 
	$\liminf_{n\to\infty} \varrho_n = \infty$ 
implies that there exists 
	$\mathfrak{n} \in \N$ 
which satisfies that 
	$\inf_{n\in [\mathfrak{n},\infty)\cap\N} \varrho_n \geq \rho$. 
Next let 
	$\eta \in (0,\infty)$ 
satisfy that 
	$\eta < ce^{L(\varrho_{\mathfrak{n}})}\kappa^{\mathfrak{n}}(1+\beta L(\varrho_{\mathfrak{n}}))^{\mathfrak{n}} \mathfrak{n}^{-\nicefrac{\mathfrak{n}}{2}}$. 
This implies for every 
	$\varepsilon \in (0,\eta]$ 
that 
	$N_{\varepsilon} > \mathfrak{n}$. 
Hence, we obtain that for every 
	$\varepsilon \in (0,\eta]$ 
it holds that 
	$\inf_{n\in [N_{\varepsilon},\infty)\cap\N} \varrho_n 
	\geq 
	\inf_{n\in [\mathfrak{n},\infty)\cap\N} \varrho_n 
	\geq \rho$. 
This, the assumption that for every 
	$n\in\N$, 
	$r\in [\rho,\infty)$ 
it holds that 
	$\epsilon_{n,r} \leq c e^{L(r)}\kappa^n (1+\beta L(r))^n n^{-\nicefrac{n}{2}}$, 
and \eqref{comp_error_vs_comp_cost:definition_of_N_epsilon} ensure that for every 
	$\varepsilon \in (0,\eta]$ 
it holds that 
	\begin{equation} \label{comp_error_vs_comp_cost:error_estimate_below_eta}
	\sup_{n\in [N_{\varepsilon},\infty) \cap \N} 
	\epsilon_{n,\varrho_{n}}
	\leq 
	\sup_{n\in [N_{\varepsilon},\infty) \cap \N} 
	[ c e^{L(\varrho_n)} \kappa^n (1+\beta L(\varrho_n) )^n 
	n^{-\nicefrac{n}{2}} ]
	\leq \varepsilon.
	\end{equation}  
Next let 
	$ E = \{\varepsilon \in (0,\infty)\colon N_{\varepsilon} > 1\}$. 
Observe that \eqref{comp_error_vs_comp_cost:definition_of_N_epsilon} yields for every 
	$ \varepsilon \in E $ 
that 
	\begin{equation} 
	(N_{\varepsilon}-1)^{\nicefrac{(N_{\varepsilon}-1)}{2}} < 
	\frac{c}{\varepsilon}e^{L(\varrho_{N_{\varepsilon}-1})}[\kappa(1+\beta L(\varrho_{N_{\varepsilon}-1}))]^{(N_{\varepsilon}-1)}. 
	\end{equation} 	
This and the assumption that for every 
	$ n \in \N $ 
it holds that 
	$ \gamma_n \leq (\alpha n)^n $ 
imply that for every 
	$ \varepsilon \in E $, 
	$ \delta \in (0,\infty) $
it holds that 
	\begin{equation} \label{comp_error_vs_comp_cost:estimate_01}
	\begin{split}
	&
	\sup_{n\in [1,N_{\varepsilon}+K]\cap\N}\gamma_{n}  
	\leq
	\sup_{n\in [1,N_{\varepsilon}+K]\cap\N} (\alpha n)^n
	\leq 
	\sup_{n\in [1,N_{\varepsilon}+K]\cap\N} (\max\{\alpha,1\} n)^n
	\\
	& =  
	[\max\{\alpha,1\} (N_{\varepsilon}+K)]^{N_{\varepsilon}+K} 
	= 
	\frac{[\max\{\alpha,1\} (N_{\varepsilon}+K)]^{N_{\varepsilon}+K}}{(N_{\varepsilon}-1)^{(N_{\varepsilon}-1)(1+\delta)}}	(N_{\varepsilon}-1)^{(N_{\varepsilon}-1)(1+\delta)}
	\\
	& \leq 
	\frac{[\max\{\alpha,1\} (N_{\varepsilon}+K)]^{N_{\varepsilon}+K}}{(N_{\varepsilon}-1)^{(N_{\varepsilon}-1)(1+\delta)}}
	\frac{c^{2+2\delta}}{\varepsilon^{2+2\delta}} 
	e^{L(\varrho_{N_{\varepsilon}-1})(2+2\delta)} 
	[\kappa(1+\beta L(\varrho_{N_{\varepsilon}-1}))]^{(N_{\varepsilon}-1)(2+2\delta)}
	\\
	& \leq 
	c^{2+2\delta} 
	\varepsilon^{-(2+2\delta)} 
	\sup_{n\in\N} 
	\left[ 
	\frac{[\max\{\alpha,1\} (n+K+1)]^{(n+K+1)}}{n^{n(1+\delta)}}
	e^{L(\varrho_n)(2+2\delta)}[\kappa(1+\beta L(\varrho_n))]^{n(2+2\delta)}
	\right]
	\\
	& = \mathfrak{a}_{\delta}\varepsilon^{-(2+2\delta)}. 
	\end{split}
	\end{equation} 	
Next observe that the fact that for every 
	$t\in (0,\infty)$ 
it holds that 
	$\ln(t) \leq t-1$ 
and \eqref{comp_error_vs_comp_cost:growth_ass} ensure once again that for every 
	$ \delta \in (0,\infty) $ 
it holds that 
	\begin{equation} 
	\begin{split}
	& \limsup_{n\to\infty} 
	\left[ 
	\ln\!\left( 
	\frac{[\max\{\alpha,1\} (n+K+1)]^{(n+K+1)}}{n^{n(1+\delta)}}
	e^{L(\varrho_n)(2+2\delta)}[\kappa(1+\beta L(\varrho_n))]^{n(2+2\delta)}\right)
	\right]
	\\
	& = 
	\limsup_{n\to\infty} 
	\bigg[ 
	(n+K+1) \ln(\max\{\alpha,1\}) 
	+ 
	(n+K+1) \ln(n+K+1) 
	- 
	n ( 1 + \delta) \ln(n) 
	\\
	& \qquad \qquad  + 
	L(\varrho_n) (2 + 2\delta) 
	+ 
	n(2+2\delta) \ln(\kappa)
	+ 
	n(2+2\delta) \ln(1 + \beta L(\varrho_n))
	\bigg] 
	\\
	& \leq 
	\limsup_{n\to\infty} 
	\bigg[ 
	n\ln(n)
	\Big( 
	\frac{(n+K+1)\ln(\max\{\alpha,1\})}{n\ln(n)} 
	+ 
	\frac{n+K+1}{n} \frac{\ln(n+K+1)}{\ln(n)}
	- 
	(1+\delta)  
	\\
	& \qquad \qquad 
	+ 
	\frac{L(\varrho_n)}{n\ln(n)}(2+2\delta) 
	+ 
	(2+2\delta) \frac{\ln(\kappa)}{\ln(n)}
	+ 
	(2+2\delta) \beta \frac{L(\varrho_n)}{\ln(n)} 
	\Big)
	\bigg]
	= 
	-\infty. 
	\end{split}
	\end{equation} 
This, \eqref{comp_error_vs_comp_cost:growth_ass}, and \eqref{comp_error_vs_comp_cost:definition_of_a_delta} imply for every 
	$ \delta \in (0,\infty) $ 
that 
	\begin{equation} \label{comp_error_vs_comp_cost:finite_sup}
	a_{\delta} 
	= 
	c^{2+2\delta} 
	\sup_{n\in\N} 
	\left[ 
	\frac{[\max\{\alpha,1\} (n+K+1)]^{(n+K+1)}}{n^{n(1+\delta)}}
	e^{L(\varrho_n)(2+2\delta)}[\kappa(1+\beta L(\varrho_n))]^{n(2+2\delta)}
	\right]
	< \infty. 
	\end{equation} 
Next observe that the assumption that for every 
	$n\in\N$ 
it holds that 
	$\gamma_n \leq (\alpha n)^n$ 
	and \eqref{comp_error_vs_comp_cost:definition_of_b_delta} 
ensure that for every 
	$ \varepsilon \in (0,\eta] \setminus E $, 
	$ \delta \in (0,\infty)$
it holds that 
	\begin{equation} 
	\sup_{n\in [1,N_{\varepsilon}+K]\cap\N} 
	\gamma_{n} 
	= 
	\sup_{n\in [1,K+1]\cap\N} \gamma_n 
	\leq 
	[\max\{\alpha,1\}(K+1)]^{(K+1)}
	\left[ 
		\frac{\eta}{\varepsilon}
	\right]^{(2+2\delta)} 
	= 
	\mathfrak{b}\eta^{(2+2\delta)}
	\varepsilon^{-(2+2\delta)}
	.
	\end{equation} 	
Combining this with
	\eqref{comp_error_vs_comp_cost:definition_of_a_delta}, 
	\eqref{comp_error_vs_comp_cost:definition_of_b_delta},
	\eqref{comp_error_vs_comp_cost:error_estimate_below_eta}, 
and 
	\eqref{comp_error_vs_comp_cost:estimate_01} 
we obtain that for every 
	$\delta\in (0,\infty)$, 
	$\varepsilon \in (0,\eta]$ 
it holds that 
	$
	\sup_{n\in [N_{\varepsilon},\infty)\cap\N} 
	\epsilon_{n,\varrho_{n}} 
	\leq \varepsilon
	$
and 
	\begin{equation} \label{comp_error_vs_comp_cost:estimate_02}
	\sup_{n\in [1,N_{\varepsilon}+K]\cap\N} \gamma_{n} 
	\leq 
	\varepsilon^{-(2+2\delta)} 
	\max\!\left\{	
		\mathfrak{a}_{\delta},
		\mathfrak{b}\eta^{(2+2\delta)}
	\right\}\!.
	\end{equation} 
Next let 
	$\mathfrak{N}_{\varepsilon}\in \N_0$, $\varepsilon\in (0,1]$, 
satisfy for every 
	$\varepsilon \in (0,1]$ 
that 
	\begin{equation} \label{comp_error_vs_comp_cost:definition_of_mathfrak_N}
	\mathfrak{N}_{\varepsilon} 
	= 
	\begin{cases} 
	N_{\varepsilon} & \colon 0 < \varepsilon \leq \eta \\
	N_{\eta} & \colon \eta < \varepsilon \leq 1.  
	\end{cases} 
	\end{equation}
This and \eqref{comp_error_vs_comp_cost:estimate_02} ensure that for every 
	$ \delta \in (0,\infty) $, 
	$ \varepsilon \in (\eta,1] $ 
it holds that 
	$
	\sup_{n\in [\mathfrak{N}_{\varepsilon},\infty) \cap\N} \epsilon_{n,\rho_{n}} 
	= 
	\sup_{n\in [{N}_{\eta},\infty) \cap \N} 
	\epsilon_{n,\rho_{n}}
	\leq \eta 
	\leq \varepsilon 
	$
and 
	\begin{equation} \label{comp_error_vs_comp_cost:estimate_03}
	\begin{split}
	\sup_{n\in [1,\mathfrak{N}_{\varepsilon}+K]\cap\N} \gamma_{n} 
	& = 
	\sup_{n\in [1,N_{\eta}+K]\cap\N} \gamma_{n} 
	\leq 
	\max\left\{
		\mathfrak{a}_{\delta}, 
		\mathfrak{b}\eta^{2+2\delta}
	\right\} \eta^{-(2+2\delta)} 
	=
	\max\left\{ 	
	\mathfrak{a}_{\delta}{\eta}^{-(2+2\delta)}, 
	\mathfrak{b}
	\right\} 
	\\
	& \leq 
	\max\left\{ 	
	\mathfrak{a}_{\delta}{\eta}^{-(2+2\delta)}, 
	\mathfrak{b}
	\right\} 
	\varepsilon^{-(2+2\delta)}
	. 
	\end{split}
	\end{equation} 
Combining this with 
	\eqref{comp_error_vs_comp_cost:estimate_02} and \eqref{comp_error_vs_comp_cost:definition_of_mathfrak_N} 
establishes that for every 
	$\delta \in (0,\infty)$,
	$\varepsilon \in (0,1]$
it holds that 
	\begin{equation} 
	\sup_{n\in [1,\mathfrak{N}_{\varepsilon}+K]\cap\N} \gamma_{n} 
	\leq 
	\left(
	\max\!\left\{ 1, \eta^{2+2\delta} \right\}
	\max\!\left\{ 
		\mathfrak{a}_{\delta}\eta^{-(2+2\delta)} , 
		\mathfrak{b} \right\}
	\right) 
	\varepsilon^{-(2+2\delta)} 
	\,\,\text{and}\,\,
	\sup_{n\in [\mathfrak{N}_{\varepsilon},\infty)\cap\N} 
	\epsilon_{n,\varrho_n} 
	\leq \varepsilon.  
	\end{equation} 
The proof of \cref{comp_error_vs_comp_cost} is thus completed. 
\end{proof}

\begin{lemma} \label{cumulative_cost_is_at_most_twice_final_cost}
Let 
	$\alpha\in [1,\infty)$. 
Then it holds for every 
	$n\in\N$ 
that 
	$\sum_{m=1}^n (\alpha m)^m \leq 2 (\alpha n)^n$. 
\end{lemma} 

\begin{proof}[Proof of \cref{cumulative_cost_is_at_most_twice_final_cost}] 
First, note that the claim is clear in the case $n=1$. Next observe that for all 
	$n\in\N\cap [2,\infty)$ 
it holds that 
	$\alpha n \geq 2$. 
This implies that for all 
	$n\in\N\cap [2,\infty)$ 
it holds that 
	\begin{equation}
	\begin{split}
	\sum_{m=1}^{n} \frac{(\alpha m)^m}{(\alpha n)^n} 
	& \leq  
	\sum_{m=1}^{n} \frac{(\alpha n)^m}{(\alpha n)^n} 
	= 
	\sum_{m=1}^{n} \frac{1}{(\alpha n)^{n-m}} %\left(\frac{m}{n}\right)^{m} 
	= 
	\sum_{k=0}^{n-1} \frac{1}{(\alpha n)^k}
	\leq 
	\sum_{k=0}^{n-1} \left(\frac{1}{2}\right)^{k} 
%	\\
%	& \leq 
%	\frac{1}{1-\frac{1}{\alpha n}} 
%	= 
%	\frac{\alpha n}{\alpha n - 1} 
	\leq 2 . 
	\end{split}
	\end{equation}
The proof of \cref{cumulative_cost_is_at_most_twice_final_cost} is thus completed. 
\end{proof}

%The next result, \cref{sum_of_costs}, is an immediate consequence of combining
%	\cref{comp_error_vs_comp_cost} 
%and  
%	\cref{cumulative_cost_is_at_most_twice_final_cost} 
%above.  

\begin{lemma} \label{sum_of_costs}
Let 
	$\alpha,\beta,c,\kappa,\rho \in (0,\infty)$, 
	$K\in\N_0$,
	$(\gamma_n)_{n\in\N}\subseteq [0,\infty)$,
	$(\epsilon_{n,r})_{n\in\N,r\in [\rho,\infty)}\subseteq [0,\infty)$, 
let 
	$L\colon (0,\infty) \to [0,\infty)$ 
be a function, 
assume for every 
	$n\in\N$, 
	$r\in [\rho,\infty)$
that		
	$\gamma_n \leq (\alpha n)^n$ 
and 
	$\epsilon_{n,r} \leq c e^{L(r)} \kappa^n(1+\beta L(r))^n n^{-\nicefrac{n}{2}}$, 
and let 
	$\varrho\colon\N\to (0,\infty)$
satisfy 
	$\limsup_{ n\to\infty} ( \frac{L(\varrho_n)}{\ln(n)} + \frac{1}{\varrho_n} ) = 0$. 
Then there exist 
	$ \mathfrak{N}\colon (0,1] \to \N$ %)_{\varepsilon} \in \N$, $\varepsilon \in (0,1]$, 
and 
	$ \mathfrak{c}\colon (0,\infty) \to [0,\infty)$ %_{\delta} \in [0,\infty) $, $\delta \in (0,\infty)$, 
such that for every 
	$\delta \in (0,\infty)$, 
	$\varepsilon \in (0,1]$
it holds that 
	$ 
	\sum_{n=1}^{\mathfrak{N}_{\varepsilon}+K} \gamma_{n} 
	\leq 
	\mathfrak{c}_{\delta} \varepsilon^{-(2+2\delta)} $ 
and 
	$
	\sup_{n\in [\mathfrak{N}_{\varepsilon},\infty)\cap\N} \epsilon_{n,\varrho_{n}} \leq \varepsilon $. 

\end{lemma}

\begin{proof}[Proof of \cref{sum_of_costs}]
First, observe that for every 
	$n\in\N$ 
it holds that 
	$\gamma_n \leq (\max\{\alpha,1\}n)^n$. 
\cref{comp_error_vs_comp_cost}
(with 
	$\alpha=\max\{\alpha,1\}$, 
	$\beta=\beta$, 
	$c=c$, 
	$\kappa=\kappa$, 
	$\rho=\rho$, 
	$K=K$, 
	$L=L$,
	$\varrho_n=\varrho_n$, 
	$\gamma_n=(\max\{\alpha,1\}n)^n$, 
	$\epsilon_{n,r}=\epsilon_{n,r}$
for 
	$r\in [\rho,\infty)$, 
	$n\in\N$
in the notation of \cref{comp_error_vs_comp_cost}) 
therefore guarantees that there exist 
	$\mathfrak{N}_{\varepsilon}\in\N$, $\varepsilon\in (0,1]$, 
and 
	$\mathfrak{c}_{\delta}\in [0,\infty)$, $\delta\in (0,\infty)$, 
such that for every 
	$\delta\in (0,\infty)$, 
	$\varepsilon \in (0,1]$ 
it holds that 
	$\sup_{n\in [1,\mathfrak{N}_{\varepsilon}+K]\cap\N} (\max\{\alpha,1\}n)^n \leq \mathfrak{c}_{\delta} \varepsilon^{-(2+2\delta)} $ 
and 
	$\sup_{n\in [\mathfrak{N}_{\varepsilon},\infty)\cap\N} \epsilon_{n,\varrho_n}\leq \varepsilon$. 
The fact that for every 
	$n\in\N$ 
it holds that 
	$\gamma_n \leq (\max\{\alpha,1\}n)^n$, 
the fact that for every 
	$N\in\N$ 
it holds that 
	$\sup_{n\in [1,N]\cap\N} (\max\{\alpha,1\}n)^n = (\max\{\alpha,1\}N)^N$, 
and \cref{cumulative_cost_is_at_most_twice_final_cost} hence imply that for every 
	$\varepsilon\in (0,1]$ 
it holds that 
	$\sup_{n\in [\mathfrak{N}_{\varepsilon},\infty)\cap\N} \epsilon_{n,\varrho_n} \leq \varepsilon$
and 
	\begin{equation}
	\sum_{n=1}^{\mathfrak{N}_{\varepsilon}+K} 
	\gamma_n 
	\leq
	\sum_{n=1}^{\mathfrak{N}_{\varepsilon}+K} 
	(\max\{\alpha,1\}n)^n 
	\leq 
	2 (\max\{\alpha,1\}(\mathfrak{N}_{\varepsilon}+K)^{(\mathfrak{N}_{\varepsilon}+K)} 
	\leq 2\mathfrak{c}_{\delta}\varepsilon^{-(2+2\delta)}. 
	\end{equation} 
The proof of \cref{sum_of_costs} is thus completed. 
\end{proof}

\begin{theorem} 
	\label{thm:main_theorem_forward_formulation}
Let
	$\rho,T\in (0,\infty)$,
	$c,\gamma,p \in [0,\infty)$,
	$K\in\N_0$,		
	$\Theta = \cup_{n\in\N} \Z^n$,
	$(f_d)_{d\in\N},
	(\fdr)_{d\in\N, r\in (0,\infty)} \subseteq C([0,T]\times\R^d\times\R,\R)$,
	% \subseteq C([0,T]\times\R^d\times\R,\R)$, 
	%$
let 
	$L\colon (0,\infty)\to [0,\infty)$ be a function, 
%for every 
%	$d \in \N$ 
%let 
%	$\norm{\cdot}_d\colon \R^d \to [0,\infty)$ be a norm on $\R^d$,
let 
$(\Omega,\mathcal{F},\P)$ 
be a probability space, 
let
$\mathcal{R}^{\theta}\colon \Omega \to [0,1]$, $\theta\in\Theta$, 
be independent $\mathcal{U}_{[0,1]}$-distributed random variables, 
let 
$W^{d,\theta}\colon [0,T]\times\Omega\to\R^d$, $d\in\N$, $\theta\in\Theta$,
be independent standard Brownian motions, 
assume that 
	$(\mathcal{R}^{\theta})_{\theta\in\Theta}$
and
	$(W^{d,\theta})_{(d,\theta)\in\N\times\Theta}$ 
are independent, 
let 
	$R^{\theta}\colon [0,T]\times\Omega \to [0,T]$, $\theta\in\Theta$, 
satisfy for every 
$\theta\in\Theta$, 
$t\in [0,T]$
that 
$R^{\theta}_t = t \mathcal{R}^{\theta}$,
for every 
$d\in \N$,
$\theta\in \Theta$,	
$s\in [0,T]$, 
$t\in [s,T]$,
$x\in \R^d$
let 
$X^{d,\theta}_{s,t,x}\colon\Omega\to\R^d$
satisfy 
$
X^{d,\theta}_{s,t,x} 
= 
x + \sqrt{2}
(W^{d,\theta}_t - W^{d,\theta}_s)
$,  
assume for every 
	$d \in \N$,
	$r\in (0,\infty)$, 
	$t\in (0,T]$, 
	$x\in \R^d$, 
	$u,v\in [-r,r]$, 
	$w\in\R$ 
that  
	$wf_d(t,x,w) \leq c (1+w^2)$, 
	$\fdr(t,x,w) = f_d(t,x,\min\{r,\max\{-r,w\}\})$,  
	$\EXP{\int_0^t |f_d(s,X^{d,0}_{s,t,x},0)|\,ds} < \infty$,
and 
	$|f_d(t,x,u) - f_d(t,x,v)|\leq L(r)|u-v|$,
let 
	$u_d \in C([0,T]\times\R^d,\R)$, $d\in\N$,  
satisfy for every 
	$d\in\N$,
	$t\in (0,T]$, 
	$x\in \R^d$ 
that 
	$e^{cT}(1+ |u_d(0,x)|^2)^{\nicefrac12} \leq \rho$, 
	$
	\inf_{a\in \R} 
	[ 
	\sup_{s\in [0,T]}\sup_{y=(y_1,\ldots,y_d)\in\R^d} (e^{a(|y_1|^2+\ldots+|y_d|^2)}|u_d(s,y)|) 
	] 
	< \infty
	$,
	$u_d|_{(0,T]\times\R^d}\in C^{1,2}((0,T]\times\R^d,\R)$, 
and 
	\begin{equation}
	\label{prop:main_theorem_assumptions_on_u}
	(\tfrac{\partial}{\partial t}u_d)(t,x) 
	= 
	(\Delta_x u_d)(t,x) 
	+ 
	f_d(t,x,u_d(t,x)),
	\end{equation}  
let 
$U^{d,\theta}_{n,M,r}\colon [0,T]\times\R^d\times\Omega\to\R$, 
$d,M\in\N$,
$\theta\in\Theta$,  
$n\in\N_0$, 
$r\in (0,\infty)$, 
satisfy for every 
$d,n,M\in\N$, 
$\theta\in\Theta$,
$r\in (0,\infty)$,
$t\in [0,T]$, 
$x\in \R^d$ 
that
$U^{d,\theta}_{0,M,r}(t,x) = 0$
and
\begin{equation}
\begin{split}
\label{main_theorem_forward_formulation:transformed_mlp}
& 
U^{d,\theta}_{n,M,r}(t,x) 
= 
\frac{1}{M^n}\left[ 
\sum_{m=1}^{M^n} 
\left( 
u_d(0,X^{d,(\theta,0,-m)}_{0,t,x}) 
+ 
t \, f_d\big(R^{(\theta,0,m)}_{t}, 
X^{d,(\theta,0,m)}_{R^{(\theta,0,m)}_{t},t,x},0\big)
\right)
\right] 
\\ 
& \qquad 
+ 
\sum_{k=1}^{n-1} \frac{t}{M^{n-k}} 
\Bigg[ 
\sum_{m=1}^{M^{n-k}} 
\bigg(
\fdr\Big( 
R^{(\theta,k,m)}_{t}, 
X^{d,(\theta,k,m)}_{R^{(\theta,k,m)}_{t},t,x},
U^{d,(\theta,k,m)}_{k,M,r}\big(
R^{(\theta,k,m)}_{t}, 
X^{d,(\theta,k,m)}_{R^{(\theta,k,m)}_{t},t,x}
\big)
\Big)
\\ 
& \qquad\qquad\qquad\qquad
-
\fdr\Big(R^{(\theta,k,m)}_{t}, 
X^{d,(\theta,k,m)}_{R^{(\theta,k,m)}_{t},t,x}, 
U^{d,(\theta,-k,m)}_{k-1,M,r}\big(R^{(\theta,k,m)}_{t}, 
X^{d,(\theta,k,m)}_{R^{(\theta,k,m)}_{t},t,x}\big)
\Big)
\bigg)
\Bigg]\!, 
\end{split}
\end{equation}  
let 
	$\varrho\colon\N\to (0,\infty)$ 
satisfy	
	$\limsup_{n\to\infty} ( \frac{L(\varrho_n)}{\ln(n)} + \frac{1}{\varrho_n}) = 0$, 
and let 
	$\mathfrak{C}_{d,n,M}\in \N_0$, $d,M\in\N$, $n\in\N_0$,  
satisfy for every 
	$d,n,M\in\N$ 
that 
	$\mathfrak{C}_{d,0,M} = 0$ 
and 
	\begin{equation} \label{prop:main_thm_computational_cost}
	\mathfrak{C}_{d,n,M} 
	\leq 
	(2d+1)M^n 
	+ 
	\sum_{l=1}^{n-1} M^{n-l} \,(d+1+\mathfrak{C}_{d,l,M} 
	+ 
	\mathfrak{C}_{d,l-1,M}). 
	\end{equation} 
Then there exist 
	$\mathfrak{N} \colon (0,1] \to \N$
and 
	$\mathfrak{c}\colon (0,\infty) \to [0,\infty)$
such that for every 
	$d\in \N$, 
	$\delta \in (0,\infty)$, 
	$\varepsilon \in (0,1]$, 
	$x\in\R^d$ 
with 
	$(\int_0^T \EXP{|f_d(s,X^{d,0}_{s,T,x},0)|^2}\,ds)^{\nicefrac12}\leq\gamma d^p$ 
it holds that 
	$
	\sum_{n = 1}^{\mathfrak{N}_{(\varepsilon\slash d^p)}+K} 
	\mathfrak{C}_{d,n,n}
	\leq \mathfrak{c}_{\delta} d^{1+p(2+\delta)} \varepsilon^{-(2+\delta)} 
	$
and 
	\begin{equation} \label{main_theorem_forward_formulation:claim} 
%	\begin{split}
%	&
%	\left[ 
%	\smallsum\limits_{n = 1}^{\mathfrak{N}_{\nicefrac{\varepsilon}{d^p}}+K} 
%	\mathfrak{C}_{d,n,n}\right] 
%	\leq \mathfrak{c}_{\delta} d^{1+p(2+2\delta)} \varepsilon^{-(2+2\delta)} 
%	\qquad\text{and} 
%	\\& 
	\left[ 
	\sup_{n\in [\mathfrak{N}_{(\varepsilon\slash d^p)},\infty)\cap\N} 
	\left(
	\Exp{|U^{d,0}_{n,n,\varrho_n}(T,x) - u_d(T,x)|^2} 
	\right)^{\!\nicefrac12}
	\right]
	\leq \varepsilon.
%	\end{split}
	\end{equation}
\end{theorem} 

\begin{proof}[Proof of \cref{thm:main_theorem_forward_formulation}]
Throughout this proof let 
	$\mathfrak{X}_d \subseteq \R^d$, $d\in\N$, 
satisfy for every 
	$d\in\N$ 
that 
	\begin{equation} 
	\mathfrak{X}_d 
	= 
	\left\{ 
	x \in \R^d \colon 
	\left(
	\int_0^T \Exp{\left|f_d(s,X^{d,0}_{s,T,x},0)\right|^2}\!\,ds\right)^{\!\nicefrac12} \leq \gamma d^p
	\right\}\!, 
	\end{equation} 
let 
	$\epsilon_{n,r} \in [0,\infty]$, $n\in\N$, $r\in (0,\infty)$, 
satisfy for every 
	$n\in\N$, 
	$r\in (0,\infty)$ 
that 
	\begin{equation} 
	\epsilon_{n,r} 
	= 
	\sup\!\left( 
	\left\{ 
	\frac{1}{d^p}
	\left( 
	\Exp{\left| U^{d,0}_{n,n,r}(T,x) - u_d(T,x) \right|^2}
	\right)^{\!\nicefrac12} 
	\colon x\in\mathfrak{X}_d, d\in\N
	\right\}
	\cup 
	\{0\}
	\right) 
%	\sup_{d\in\N} 
%	\sup_{x\in\mathfrak{X}_d}
%	\left[ 
%	\frac{1}{d^p}
%	\left( 
%	\Exp{\left| U^{d,0}_{n,n,r}(0,x) - u_d(0,x) \right|^2}
%	\right)^{\!\nicefrac12}
%	\right]
\!,
\end{equation} 
and let 
	$\gamma_n\in [0,\infty]$, $n\in\N$, 
satisfy for every 
	$n\in\N$ 
that 
\begin{equation} 
\gamma_{n} 
= 
\sup_{d\in\N}\left(\frac{\mathfrak{C}_{d,n,n}}{d}\right)\!.
\end{equation}
Note that \cref{lemma:comp_cost} demonstrates that for every 
$d,n,M\in \N$ 
it holds that 
$\mathfrak{C}_{d,n,M} \leq d (5M)^n$. 
This implies for every 
$n\in\N$ 
that 
\begin{equation} \label{convergeLocalLipschitz:gamma_estimate}
\gamma_n 
= 
\sup_{d\in\N} \left(
\frac{\mathfrak{C}_{d,n,n}}{d} \right) 
\leq 
\sup_{d\in\N} \left(
\frac{d(5n)^n}{d} \right) 
= (5n)^n < \infty.
\end{equation} 
Next observe that \cref{thm:convergenceLocalLipschitz_forward_formulation} (with 
	$d=d$, 
	$T=T$, 
	$\Theta=\Theta$, 
	$f=f_d$, 
	$\fr=\fdr$, 
	$(\Omega,\cF,\P)=(\Omega,\cF,\P)$, 
	$\mathcal{R}^{\theta}=\mathcal{R}^{\theta}$, 
	$W^{\theta}=W^{d,\theta}$, 
	$R^{\theta}=R^{\theta}$, 
	$X^{\theta}_{t,s,x}=X^{d,\theta}_{t,s,x}$, 
	$U^{\theta}_{n,M,r}(t,x)=U^{d,\theta}_{n,M,r}(t,x)$, 
	$\rho=\rho$, 
	$c=c$, 
	$\norm{\cdot}=(\R^d\ni y=(y_1,\ldots,y_d)\mapsto |y_1|^2+\ldots+|y_d|^2\in\R)$, 
	$L=L$, 
	$u=u_d$
for 
	$d,M\in\N$, 
	$\theta\in\Theta$, 	
	$n\in\N_0$,	
	$r\in (0,\infty)$, 
	$t\in [0,T]$, 
	$s\in [t,T]$, 
	$x\in \R^d$
in the notation of \cref{thm:convergenceLocalLipschitz_forward_formulation}) 
ensures that it holds for every 
	$d,M \in \N$,
	$n \in \N_0$, 
	$r\in [\rho,\infty)$, 
	$x \in \R^d$
that 
	\begin{equation} %\label{convergenceLocalLipschitz:claim1} 
	\begin{split}
	&
	\left(
	\Exp{|  U^{d,0}_{n,M,r} (T, x) - u_d(T, x)  |^2 }
	\right)^{\!\nicefrac{1}{2}}
	\leq
	\frac{ e^{\nicefrac{M}{2}} ( 1 + 2 L(r) T )^n }{M^{\nicefrac{n}{2}}}
	\\
	& \qquad \qquad
	\cdot
	e^{L(r)T}
	\left[ 
	\left( \Exp{|u_d(0,X^{d,0}_{0,T,x})|^2} \right)^{\!\nicefrac12}
	+ 
	\sqrt{T}
	\left|
	\int_0^T \Exp{|f_d(s,X^{d,0}_{s,T,x},0)|^2}\!\,ds\right|^{\nicefrac12}
	\right]\!. 
	\end{split} 
	\end{equation}
This implies that for every 
$n\in\N$, 
$r\in [\rho,\infty)$ 
it holds that 
	\begin{equation} 
	\begin{split}
	\epsilon_{n,r} 
	& = 
	\sup\!\left( 
	\left\{ 
	\frac{1}{d^p}
	\left( 
	\Exp{\left| U^{d,0}_{n,n,r}(T,x) - u_d(T,x) \right|^2}
	\right)^{\!\nicefrac12} 
	\colon x\in\mathfrak{X}_d, d\in\N
	\right\}
	\cup 
	\{0\}
	\right)
	\\
	& \leq 
	\sup_{d\in\N} 
	\left(
	\frac{e^{\nicefrac{n}{2}}(1+2L(r)T)^n}{n^{\nicefrac{n}{2}}}
	\frac{e^{L(r)T}}{d^p}
	\left[ 
	\sup_{m\in\N}\sup_{x\in\R^d} |u_m(0,x)| 
	+ 
	\gamma \sqrt{T} d^p
	\right]
	\right) 
	\\
	& \leq 
	\left[ 
	\left( \sup_{d\in\N}\sup_{x\in\R^d} 
	|u_d(0,x)| \right) 
+ 
\gamma \sqrt{T}
\right]
e^{L(r)T}
\frac{e^{\nicefrac{n}{2}}(1+2L(r)T)^n}{n^{\nicefrac{n}{2}}} < \infty.
\end{split}
\end{equation} 
This, \eqref{convergeLocalLipschitz:gamma_estimate},
and \cref{sum_of_costs} (with 
	$\alpha = 5$, 
	$\beta = 2$,
	$ c = 1 + \sup_{d\in\N}\sup_{x\in\R^d} 
	|u_d(0,x)| 
	+ 
	\gamma\sqrt{T}
	$, 
	$\kappa = \sqrt{e}$, 
	$\rho = \rho$, 
	$K = K$, 
	$L(s) = L(s)T $, 
	$\gamma_n = \gamma_n $, 
	$\epsilon_{n,r} = \epsilon_{n,r} $, 
	$\varrho_n=\varrho_n$ 
for 
$ n \in \N $, 
$ r \in [\rho,\infty) $,
$ s \in (0,\infty) $
in the notation of \cref{comp_error_vs_comp_cost}) 
guarantee that there exist
$\mathfrak{N} \colon (0,1] \to \N$
and 
$\mathfrak{c} \colon (0,\infty) \to [0,\infty)$
which satisfy that for every 
	$\delta \in (0,\infty)$, 
	$\varepsilon \in (0,1]$
it holds that 
	\begin{equation} 
	\left[\smallsum\limits_{n=1}^{\mathfrak{N}_{\varepsilon}+K} 
	\gamma_n 
	\right]
	\leq 
	\mathfrak{c}_{\delta} \varepsilon^{-(2+2\delta)}
	\qquad
	\text{and}\qquad 
	\sup_{n\in [\mathfrak{N}_{\varepsilon},\infty)\cap\N} 
	\epsilon_{n,\varrho_n} \leq \varepsilon 
	\end{equation}
This implies that for every 
	$d\in\N$, 
	$\delta \in (0,\infty)$, 
	$\varepsilon\in (0,1]$, 
	$x\in\mathfrak{X}_d$
it holds that 
\begin{equation} \label{thm_main_theorem_backward_formulation:comp_cost}
\left[ 
\smallsum\limits_{n = 1}^{\mathfrak{N}_{(\varepsilon\slash d^p)}+K}
\mathfrak{C}_{d,n,n} \right]
\leq 
\mathfrak{c}_{\delta} d 
\left[\frac{\varepsilon}{d^p}\right]^{-(2+2\delta)} 
= 
\mathfrak{c}_{\delta} d^{1+p(2+2\delta)} 
\varepsilon^{-(2+2\delta)} 
\end{equation}
and
\begin{equation} 	
\begin{split}
& 
\sup_{n\in [\mathfrak{N}_{(\varepsilon\slash d^p)},\infty)\cap\N}
\left(
\Exp{|U^{d,0}_{n,n,\varrho_{n}}(T,x) - u_d(T,x)|^2} 
\right)^{\!\nicefrac12}
\leq 
d^p	
\left[ 
\sup_{n\in [\mathfrak{N}_{(\varepsilon\slash d^p)},\infty)\cap\N} 
\epsilon_{n,\varrho_n} 
\right]
\leq 
d^p \frac{\varepsilon}{d^p} 
= \varepsilon
. 
\end{split}
\end{equation} 
This establishes \eqref{main_theorem_forward_formulation:claim}. 
The proof of \cref{thm:main_theorem_forward_formulation} is thus completed.
\end{proof}

\section{MLP approximations for Allen--Cahn type partial differential equations}
\label{sec:applications}

In this section we consider sample applications of \cref{thm:main_theorem_forward_formulation}. This provides us with examples of Allen--Cahn PDEs for which the curse of dimensionality can be broken in numerical approximations. %In the next result, \cref{cor_lip} below, we consider nonlinearities without explicit dependence on time and space (see \eqref{cor_lip:PDE}). In \cref{cor:polynomially_growing} we specialize on locally Lipschitz continuous with at most polynomially growing Lipschitz bounds and we slightly adjust the MLP approximation scheme (see \eqref{polynomially_growing:mlp_scheme}). Finally, in \cref{polynomkorollar} below, we turn to the prototypical Allen--Cahn equation.  

\begin{cor} \label{cor_lip}
Let
	$T \in (0,\infty)$,
	$c,\rho \in [0,\infty)$,
	$K \in \N_0$,
	$\Theta = \cup_{n\in\N} \Z^n$, 
	$f\in C(\R,\R)$, 
	$(\fr)_{r\in (0,\infty)} \subseteq C(\R,\R)$, 
let 
	$L\colon (0,\infty)\to [0,\infty)$ 
be a function, 
assume for every 
	$r\in (0,\infty)$, 
	$u,v\in [-r,r]$, 
	$w\in\R$ 
that
	$wf(w)\leq c(1+w^2)$, 
	$\fr(w) = f(\min\{r,\max\{-r,w\}\})$, 
and 
	$|f(u) - f(v)|\leq L(r)|u-v|$,   
%for every 
%	$d\in\N$ 
%let 
%	$\norm{\cdot}_d \colon \R^d \to [0,\infty)$ 
%be a norm on $\R^d$,
let 
	$\varrho\colon \N \to (0,\infty)$ 
satisfy
	$\limsup_{n\to\infty} ( \frac{L(\varrho_n)}{\ln(n)} + \frac{1}{\varrho_n}) = 0$, 
let 
	$u_d\in C([0,T]\times\R^d,\R)$, $d\in\N$, 
satisfy for every 
	$d\in \N$, 
	$t\in [0,T]$, 
	$x\in \R^d$ 
that 
	$e^{cT}(1 + |u_d(0,x)|^2)^{\nicefrac12} 
	\leq \rho$,  
	$u_d|_{(0,T]\times\R^d}\in C^{1,2}((0,T]\times\R^d,\R)$, 
	$
	\inf_{a\in \R} 
	[ 
	\sup_{s\in [0,T]}\sup_{y=(y_1,\ldots,y_d)\in\R^d} (e^{a(|y_1|^2+\ldots+|y_d|^2)}|u_d(s,y)|) 
	] 
	< \infty
	$,
and 
	\begin{equation} \label{cor_lip:PDE}
	(\tfrac{\partial}{\partial t}u_d)(t,x) 
	= 
	(\Delta_x u_d)(t,x) 
	+ 
	f(u_d(t,x)),
	\end{equation} 
let 
	$(\Omega,\mathcal{F},\P)$ 
be a probability space, 
let
	$\mathcal{R}^{\theta}\colon \Omega \to [0,1]$, $\theta\in\Theta$, 
be independent $\mathcal{U}_{[0,1]}$-distributed random variables, 
let 
	$W^{d,\theta}\colon [0,T]\times\Omega\to\R^d$,
	$d\in\N$, $\theta\in\Theta$,
be independent standard Brownian motions, 
assume that 
	$(\mathcal{R}^{\theta})_{\theta\in\Theta}$ 	
and 
	$(W^{d,\theta})_{(d,\theta)\in\N\times\Theta}$ 
are independent, 
let 
	$R^{\theta}_t\colon \Omega \to [0,t]$, $\theta\in\Theta$, $t\in [0,T]$, 
satisfy for every 
	$t\in [0,T]$ 
that 
	$R^{\theta}_t = t \mathcal{R}^{\theta}$,
for every 
	$d\in\N$, 
	$t\in [0,T]$, 
	$s\in [t,T]$,
	$x\in \R^d$, 
	$\theta\in \Theta$ 
let 
	$
	X^{d,\theta}_{t,s,x}\colon\Omega\to\R^d
	$
satisfy 
	$
	X^{d,\theta}_{t,s,x} 
	= 
	x + \sqrt{2}
	(W^{d,\theta}_s - W^{d,\theta}_t)
	$, 
let 
	$
	U^{d,\theta}_{n,M,r}\colon [0,T]\times\R^d\times\Omega\to\R$, 
	$\theta\in\Theta$, 
	$d,M\in\N$,
	$n\in\N_0$,
	$r\in (0,\infty)$, 
satisfy for every 
	$d,n,M\in\N$,
	$\theta\in\Theta$,
	$r\in (0,\infty)$,
	$t\in [0,T]$, 
	$x\in \R^d$ 
that
	$U^{d,\theta}_{0,M,r}(t,x) = 0$
and 
	\begin{equation} 
	\begin{split}
	& 
	U^{d,\theta}_{n,M,r}(t,x) 
	=
	\sum_{k=1}^{n-1} \frac{t}{M^{n-k}} 
	\Bigg[ 
	\sum_{m=1}^{M^{n-k}} 
	\Big(
	\fr\big(
	U^{d,(\theta,k,m)}_{k,M,r}(
	R^{(\theta,k,m)}_{t}, 
	X^{d,(\theta,k,m)}_{R^{(\theta,k,m)}_{t},t,x}
	)
	\big)
	\\
	& 
	-
	\fr\big(U^{d,(\theta,-k,m)}_{k-1,M,r}(R^{(\theta,k,m)}_{t}, 
	X^{d,(\theta,k,m)}_{R^{(\theta,k,m)}_{t},t,x})
	\big)
	\Big)
	\Bigg]
	+ 	\frac{1}{M^n}\left[ 
	\sum_{m=1}^{M^n} 
	\left( 
	u_d(0,X^{d,(\theta,0,-m)}_{0,t,x}) 
	+ 
	t \, f(0)
	\right)
	\right]
	\!, 
	\end{split} 
	\end{equation} 
and let 
	$\mathfrak{C}_{d,n,M}\in\N_0$, $d,M\in\N$, $n\in\N_0$, 
satisfy for every 
	$d,n,M\in\N$ 
that 
	$\mathfrak{C}_{d,0,M} = 0$ 
and 
	\begin{equation} \label{computational_cost}
	\mathfrak{C}_{d,n,M} 
	\leq 
	(2d+1)M^n 
	+ 
	\sum_{l=1}^{n-1} M^{n-l} \,(d+1+\mathfrak{C}_{d,l,M} 
	+ 
	\mathfrak{C}_{d,l-1,M}).   
	\end{equation} 
Then 
	\begin{enumerate}[(i)] 
		\item\label{cor_lip:item1} 
	it holds for every 
		$d,M\in\N$, 
		$n\in\N_0$, 
		$r\in [\rho,\infty)$
	that 
		\begin{equation} 
		\begin{split}
		&
		\sup_{x\in\R^d}
		\left(\Exp{
			|
			U^{d,0}_{n,M,r}(T,x)
			-
			u_d(T,x)
			|^2}\right)^{\!\nicefrac12}
		\\
		& %\qquad \qquad \qquad \qquad
		\leq 
		e^{L(r)T}
		\left[\sup_{x\in\R^d} |u_d(0,x)| + T |f(0)| \right]
		\left[
		\frac{
			e^{\nicefrac{M}{2}}(1+2L(r)T)^n
		}{M^{\nicefrac{n}{2}}}
		\right] 
		\end{split} 
		\end{equation} 
	and 
		\item \label{cor_lip:item2} 
	there exist 
		$\mathfrak{N}\colon (0,1] \to \N$ 
	and 
		$\mathfrak{c}\colon (0,\infty) \to [0,\infty)$
	such that for every 
		$d\in\N$, 
		$\delta\in (0,\infty)$,
		$\varepsilon \in (0,1]$ 
	it holds that 
		$
		\sum_{n = 1}^{\mathfrak{N}_{\varepsilon}+K}
		\mathfrak{C}_{d,n,n} 
		\leq d \mathfrak{c}_{\delta} \varepsilon^{-(2+\delta)} 
		$
	and 
	\begin{equation}
%		\begin{split}
%		& 
%		\left[ 
%		\smallsum\limits_{n = 1}^{\mathfrak{N}_{\varepsilon}+K}
%		\mathfrak{C}_{d,n,n} 
%		\right]		
%		\leq d \mathfrak{c}_{\delta} \varepsilon^{-(2+2\delta)} 
%		\qquad\text{and}
%		\\
%		& 
		\sup_{n\in [\mathfrak{N}_{\varepsilon},\infty)\cap\N}
		\left[
		\sup_{x \in\R^d}
		\left(
		\Exp{|U^{d,0}_{n,n,\varrho_n}(T,x) - u_d(T,x)|^2} 
		\right)^{\!\nicefrac12}
		\right]
		\leq \varepsilon.
%		\end{split}
		\end{equation}
	\end{enumerate} 
\end{cor}

\begin{proof}[Proof of \cref{cor_lip}]
Throughout this proof let 
	$F_d\colon [0,T]\times\R^d\times\R \to \R$, $d\in\N$, 
be the functions which satisfy for every 	
	$d\in\N$, 
	$t\in [0,T]$, 
	$x\in\R^d$, 
	$w\in\R$ 
that 
	\begin{equation}\label{cor_lip:F_d}
	F_d(t,x,w) = f(w). 
	\end{equation}
Observe that the fact that for every 
	$d\in\N$, 
	$x\in \R^d$ 
it holds that 
	$|u_d(0,x)| \leq \rho < \infty$ 
and \eqref{cor_lip:F_d} ensure that for every 
	$d\in\N$, 
	$t\in [0,T]$, 
	$x\in\R^d$, 
	$w\in\R$
it holds that
	$\fr(w)=F_d(t,x,\min\{r,\max\{-r,w\}\})$, 
	$ 
	\EXP{\int_0^t |F_d(s,X^{d,0}_{s,t,x},0)|\,ds} 
	= t |f(0)| < \infty
	$, 
	$(\EXP{\int_0^T |F_d(s,X^{d,0}_{s,T,x},0)|^2\,ds})^{\nicefrac12}
	=
	\sqrt{T}|f(0)|$, 
and 	
	\begin{equation} 
	\begin{split}
	& 
	\left( \Exp{|u_d(0,X^{d,0}_{0,T,x})|^2} \right)^{\!\nicefrac12} 
	+ 
	\sqrt{T} 
	\left| 
	\int_0^T
	\Exp{|F_d(s,X^{d,0}_{s,T,x},0)|^2}\!\,ds
	\right|^{\nicefrac12}
	\leq 
	\sup_{\xi\in\R^d} 
	|u_d(0,\xi)| + T|f(0)|. 
	\end{split}
	\end{equation} 
This and \cref{thm:main_theorem_forward_formulation} (with 
	$\rho=\rho$, 
	$T=T$, 	
	$c=c$, 
	$\gamma=\sqrt{T}|f(0)|$,
	$p=0$, 
	$K=K$,
	$\Theta=\Theta$, 
	$L=L$, 
%	$\norm{\cdot}_d=\norm{\cdot}_d$, 
	$f_d=F_d$, %([0,T]\times\R^d\times\R \ni (t,x,y) \mapsto f(y) \in \R)$, 
	$\fdr=([0,T]\times\R^d\times\R\ni (t,x,y) \mapsto \fr(y)\in\R)$, 
	$u_d=u_d$, 
	$\varrho=\varrho$, 
	$(\Omega,\cF,\P)=(\Omega,\cF,\P)$, 
	$\mathcal{R}^{\theta}=\mathcal{R}^{\theta}$, 
	$W^{d,\theta}=W^{d,\theta}$, 
	$R^{\theta}=R^{\theta}$, 
	$X^{d,\theta}_{t,s,x}=X^{d,\theta}_{t,s,x}$, 
	$U^{d,\theta}_{n,M,r}(t,x)=U^{d,\theta}_{n,M,r}(t,x)$, 
	$\mathfrak{C}_{d,n,M}=\mathfrak{C}_{d,n,M}$
for 
	$d,M\in\N$,
	$\theta\in\Theta$, 
	$n\in\N_0$,
	$r\in (0,\infty)$, 
	$t\in [0,T]$, 
	$s\in [t,T]$, 
	$x\in \R^d$ 	
in the notation of \cref{thm:main_theorem_forward_formulation}) ensure that 
\begin{enumerate}[(I)]
	\item 
	for every 
		$d,M\in\N$, 
		$n\in\N_0$, 
		$r\in [\rho,\infty)$,
		$x\in\R^d$
	it holds that 
		\begin{equation} 
		\begin{split}
		& 
		\left( \Exp{|U^{d,0}_{n,M,r}(T,x) - u_d(T,x)|^2} \right)^{\!\nicefrac12} 
		\\
		& 
		\leq 
		e^{L(r)T}
		\left[ 
		\sup_{\xi\in\R^d} |u_d(0,\xi)| 
		+ 
		T |f(0)| 
		\right] 
		\left[
		\frac{e^{\nicefrac{M}{2}}(1+2L(r)T)^n}{M^{\nicefrac{n}{2}}} 
		\right]
		\end{split}
		\end{equation} 
		and 
	\item there exist 
		$\mathfrak{N}\colon (0,1]\to\N$ 
	and 
		$\mathfrak{c}\colon (0,\infty)\to [0,\infty)$
	such that for every 
		$d\in\N$, 
		$\delta\in (0,\infty)$, 
		$\varepsilon \in (0,1]$, 
		$x\in\R^d$ 
    it holds that 
    \begin{equation}
    \begin{split}
%	& 
	\left[\smallsum\limits_{n = 1}^{\mathfrak{N}_{\varepsilon}+K}
    	\mathfrak{C}_{d,n,n} 
    \right]
    	\leq 
    	\mathfrak{c}_{\delta} d \varepsilon^{-(2+2\delta)}
    	\quad\text{and}\quad
%	&
	\sup_{n\in [\mathfrak{N}_{\varepsilon},\infty)\cap\N} 
		\left( 
		\Exp{|U^{d,0}_{n,n,\varrho_n}(T,x) - u_d(T,x)|^2}
		\right)^{\!\nicefrac12} 
		\leq \varepsilon. 
		\end{split}
	\end{equation} 
	\end{enumerate}
This establishes Items~\eqref{cor_lip:item1} and \eqref{cor_lip:item2}. The proof of \cref{cor_lip} is thus completed. 
\end{proof}

%In the next result, \cref{cor:polynomially_growing} below, we consider at most polynomially growing nonlinearities without explicit dependence on space and time. This allows to choose cutoff parameters $\varrho_n$, $n\in\N$, (cf., for example, \cref{cor_lip} or \cref{cor:polynomially_growing} for a specification) growing like $\ln(\ln(n))$ for $n\to\infty$.  

\begin{cor} \label{cor:polynomially_growing}
Let
	$T \in (0,\infty)$,
	$c \in [0,\infty)$, 
	$K \in \N_0$, 
	$\Theta = \cup_{n\in\N} \Z^n$, 
	$f\in C(\R,\R)$, 
	$(\mathbf{f}_n)_{n\in\N}\subseteq C(\R,\R)$, 
let 
	$\varrho \colon \N \to (0,\infty)$ 
satisfy
	$\limsup_{n\to\infty} (\frac{\varrho_n}{\ln(\ln(n))}) < \infty =\liminf_{n\to\infty} \varrho_n$, 	
assume for every 
	$n\in\N$,
	$u,v\in\R$ 
that 
%let
%$\mathbf{f}_n\colon\R\to\R$, $n\in\N$, 
%satisfy for every 
%$n\in\N$, 
%$v\in\R$
%that 
	$|f(u)-f(v)| \leq c(1+|u|^c+|v|^c)|u-v|$,  
	$vf(v)\leq c(1+v^2)$,  
and 
	$\mathbf{f}_n(v) = f(\min\{\varrho_n,\max\{-\varrho_n,v\}\})$, 
let 
	$u_d\in C([0,T]\times\R^d,\R)$, $d\in\N$,
satisfy for every 
	$d\in\N$,
	$t\in (0,T]$,
	$x\in\R^d$ 
that 
	$
	\inf_{a\in \R} 
	[ 
	\sup_{s\in [0,T]}\sup_{y=(y_1,\ldots,y_d)\in\R^d} (e^{a(|y_1|^2+\ldots+|y_d|^2)}|u_d(s,y)|) 
	] 
	< \infty
	$,
	$|u_d(0,x)| \leq c$,  
	$u_d|_{(0,T]\times\R^d}\in C^{1,2}((0,T]\times\R^d,\R)$, 
and 
	\begin{equation}
	(\tfrac{\partial}{\partial t}u_d)(t,x) 
	= 
	(\Delta_x u_d)(t,x) 
	+ 
	f(u_d(t,x)),
	\end{equation} 
let 
	$(\Omega,\mathcal{F},\P)$ 
be a probability space, 
let
	$\mathcal{R}^{\theta}\colon \Omega \to [0,1]$, $\theta\in\Theta$, 
be independent $\mathcal{U}_{[0,1]}$-distributed random variables, 
let 
	$W^{d,\theta}\colon [0,T]\times\Omega\to\R^d$, 
	$d\in\N$,
	$\theta\in\Theta$,
be independent standard Brownian motions, 
assume that 
	$(\mathcal{R}^{\theta})_{\theta\in\Theta}$ 
and
	$(W^{d,\theta})_{(d,\theta)\in\N\times\Theta}$ 
are independent, 
let 
	$R^{\theta}\colon [0,T] \times \Omega \to [0,T]$, $\theta\in\Theta$,  
satisfy for every 
	$\theta\in\Theta$,
	$t\in [0,T]$
that 
	$R^{\theta}_t = t \mathcal{R}^{\theta}$,
for every 
	$d\in \N$,
	$t\in [0,T]$, 
	$s\in [t,T]$,
	$x\in \R^d$, 
	$\theta\in \Theta$ 
let 
	$
	X^{d,\theta}_{t,s,x}\colon\Omega\to\R^d
	$
satisfy 
	$
	X^{d,\theta}_{t,s,x} 
	= 
	x + \sqrt{2}(W^{d,\theta}_s - W^{d,\theta}_t)
	$, 
let 
	$
	U^{d,\theta}_{n,M}\colon [0,T]\times\R^d\times\Omega\to\R$, 
	$d,M\in\N$, 
	$\theta\in\Theta$, 
	$n\in\N_0$, 
satisfy for every 
	$d,n,M\in\N$,
	$\theta\in\Theta$,
	$t\in [0,T]$, 
	$x\in \R^d$ 
that
	$U^{d,\theta}_{0,M}(t,x) = 0$ 
and 
	\begin{equation} \label{polynomially_growing:mlp_scheme}
	\begin{split}
	& 
	U^{d,\theta}_{n,M}(t,x) 
	= 
	\sum_{k=1}^{n-1} \frac{t}{M^{n-k}} 
	\Bigg[ 
	\sum_{m=1}^{M^{n-k}} 
	\bigg(
	\mathbf{f}_{M}\big( 
	U^{d,(\theta,k,m)}_{k,M}(R^{(\theta,k,m)}_{t}, 
	X^{d,(\theta,k,m)}_{R^{(\theta,k,m)}_{t},t,x}
	)
	\big)
	\\ 
	& %\qquad\qquad\qquad\qquad
	-
	\mathbf{f}_{M}\Big( 
	U^{d,(\theta,-k,m)}_{k-1,M}\big( R^{(\theta,k,m)}_{t}, 
	X^{d,(\theta,k,m)}_{R^{(\theta,k,m)}_{t},t,x}\big)
	\Big)
	\bigg)
	\Bigg]
	+ 
	\frac{1}{M^n}\!\left[ 
	\sum_{m=1}^{M^n} 
	\left( 
	u_d(0,X^{d,(\theta,0,-m)}_{0,t,x}) 
	+ 
	t \, f(0)
	\right)\!
	\right]\!, 
	\end{split}
	\end{equation}
and let 
	$\mathfrak{C}_{d,n,M}\in\N_0$, 
	$d,M\in\N$, 
	$n\in\N_0$, 
satisfy for every 
	$d,n,M\in\N$ 
that 
	$\mathfrak{C}_{d,0,M} = 0$ 
and 
	\begin{equation} \label{cor_polynomially_growing:comp_cost}
	\mathfrak{C}_{d,n,M} 
	\leq 
	(2d+1)M^n 
	+ 
	\sum_{l=1}^{n-1} M^{n-l} \,(d+1+\mathfrak{C}_{d,l,M} 
	+ 
	\mathfrak{C}_{d,l-1,M}).   
	\end{equation}
Then there exist
	$\mathfrak{N} \colon (0,1] \to \N$ 
and 
	$\mathfrak{c}\colon (0,\infty) \to [0,\infty)$
such that for every 
	$d\in\N$, 
	$\delta \in (0,\infty)$,
	$\varepsilon\in (0,1]$ 
it holds that 
 	\begin{equation}
	\begin{split}
	&
	\left[ 
	\smallsum\limits_{n =  1}^{\mathfrak{N}_{\varepsilon}+K} 
	\mathfrak{C}_{d,n,n}%n_{\varepsilon}, n_{\varepsilon}}
	\right] 
	\leq c d \varepsilon^{-(2+\delta)}
	\quad\text{and}\quad
	\sup_{n\in [\mathfrak{N}_{\varepsilon},\infty)\cap\N} 
	\left[ 
	\sup_{x\in\R^d}
	\left(  
	\Exp{
		\big|
		U^{d,0}_{n,n}(T,x)
		-
		u_d(T,x) 
		\big|^2}  
	\right)^{\!\nicefrac12} 
	\right]
	\leq \varepsilon. 
	\end{split}
	\end{equation}
\end{cor}

\begin{proof}[Proof of \cref{cor:polynomially_growing}] 
Throughout this proof let 
	$L\colon (0,\infty) \to [0,\infty)$ 
satisfy for every 
	$r\in (0,\infty)$ 
that 
	$L(r) = c ( 1 + 2r^c )$, 
let 
	$F_r\colon\R\to\R$, $r\in (0,\infty)$, 
be the functions which satisfy for every 
	$r\in (0,\infty)$, 
	$v\in \R$ 
that 
	$
	F_r(v)=f(\min\{r,\max\{-r,v\}\})
	$, 
and let 
	$
	V^{d,\theta}_{n,M,r}\colon [0,T]\times\R^d\times\Omega\to\R$, 
		$d,M\in\N$,
		$\theta\in\Theta$, 
		$n\in\N_0$, 
satisfy for every 
	$d,n,M\in\N$,
	$\theta\in\Theta$,
	$r\in (0,\infty)$, 
	$t\in [0,T]$, 
	$x\in \R^d$ 
that
	$V^{d,\theta}_{n,M,r}(0,x)=0$ 
and
	\begin{equation} \label{cor_polynomially_growing:U}
	\begin{split}
	& 
	V^{d,\theta}_{n,M,r}(t,x)
	= 
	\sum_{k=1}^{n-1} \frac{t}{M^{n-k}} 
	\Bigg[ 
	\sum_{m=1}^{M^{n-k}} 
	\bigg(
	F_r\big( 
	V^{d,(\theta,k,m)}_{k,M,r}(R^{(\theta,k,m)}_{t}, 
	X^{d,(\theta,k,m)}_{R^{(\theta,k,m)}_{t},t,x}
	)
	\big)
	\\ 
	& %\qquad\qquad\qquad\qquad
	-
	F_r\Big( 
	V^{d,(\theta,-k,m)}_{k-1,M,r}\big( R^{(\theta,k,m)}_{t}, 
	X^{d,(\theta,k,m)}_{R^{(\theta,k,m)}_{t},t,x}\big)
	\Big)
	\bigg)
	\Bigg]
	+ 
	\frac{1}{M^n}\!\left[ 
	\sum_{m=1}^{M^n} 
	\left( 
	u_d(0,X^{d,(\theta,0,-m)}_{0,t,x}) 
	+ 
	t \, f(0)
	\right)\!
	\right]\!. 
	\end{split}
	\end{equation} 
Next observe that the hypothesis that 
	$\limsup_{n\to\infty} (\frac{\varrho_n}{\ln(\ln(n))}) < \infty$
implies that there exists 
	$\gamma\in (0,\infty)$ 
which satisfies that for every 
	$n \in [3,\infty)\cap\N$ 
it holds that 
	$
	\varrho_n \leq \gamma \ln(\ln(n))
	$.  
This yields that 
	\begin{equation} \label{cor_polynomially_growing_growth_of_varrho_n}
	\begin{split}
	& \limsup_{n\to\infty} 
	\left[ 
	\frac{L(\varrho_n)}{\ln(n)} 
	\right] 
	\leq  
	\limsup_{n\to\infty} 
	\left[ 
	\frac{1 + 2 (\gamma \ln(\ln(n)))^c}{\ln(n)} 
	\right] 
	\\
	& 
	\leq 
	\limsup_{n\to\infty} 
	\left[\frac{1}{\ln(n)} \right] 
	+ 
	\limsup_{n\to\infty} 
	\left[\frac{2(\gamma\ln(\ln(n)))^c}{\ln(n)} \right] = 
	2 \gamma^c
	\limsup_{n\to\infty} 
	\left[ 
	\frac{(\ln(\ln(n)))^c}{\ln(n)}
	\right] 
	%	= 
	%	\beta \gamma^{\alpha} 
	%	\lim_{x\to\infty} 
	%	\left[ 
	%	\frac{(\ln(x))^{\alpha}}{x}
	%	\right] 
	= 
	0.
	\end{split}
	\end{equation} 
Next let 
	$A \subseteq \N_0$ 
be the set given by 
	\begin{equation} \label{cor_polynomially_growing:A}
	A = \left\{n \in \N \colon 
	\begin{array}{c}
	\text{For all}~
	d,M\in\N,
	\theta\in\Theta,
	k\in\N_0\cap [0,n-1],   
	t\in [0,T],
	x\in\R^d\\
	\text{it holds that}~ 
	U^{d,\theta}_{k,M}(t,x) 
	= 
	V^{d,\theta}_{k,M,\varrho_M}(t,x) 
	\end{array} \right\}.
	\end{equation}
Note that the fact that for every 
	$d,M\in \N$, 
	$\theta\in\Theta$, 
	$t\in [0,T]$, 
	$x\in \R^d$, 
	$r\in (0,\infty)$ 
it holds that 
	$V^{d,\theta}_{0,M,r}(t,x) 
	= 
	0
	= 
	U^{d,\theta}_{0,M}(t,x)$ 
ensures that 
	$1 \in A$. 
Moreover, note that \eqref{polynomially_growing:mlp_scheme}, \eqref{cor_polynomially_growing:U}, and \eqref{cor_polynomially_growing:A} ensure that for every 
	$d,M\in \N$, 
	$\theta\in\Theta$,
	$n \in A$, 
	$t\in [0,T]$, 
	$x\in \R^d$	
it holds that 
	\begin{align}
	\nonumber 
	&
	U^{d,\theta}_{n,M}(t,x) 
	= 
	\sum_{k=1}^{n-1} \frac{t}{M^{n-k}} 
	\Bigg[ 
	\sum_{m=1}^{M^{n-k}} 
	\bigg(
	\mathbf{f}_{M}\big( 
	U^{d,(\theta,k,m)}_{k,M}(R^{(\theta,k,m)}_{t}, 
	X^{d,(\theta,k,m)}_{R^{(\theta,k,m)}_{t},t,x}
	)
	\big)
	\\ \nonumber
	& %\qquad\qquad\qquad\qquad
	-
	\mathbf{f}_{M}\Big( 
	U^{d,(\theta,-k,m)}_{k-1,M}\big( R^{(\theta,k,m)}_{t}, 
	X^{d,(\theta,k,m)}_{R^{(\theta,k,m)}_{t},t,x}\big)
	\Big)
	\bigg)
	\Bigg]
	+ 
	\frac{1}{M^n}\!\left[ 
	\sum_{m=1}^{M^n} 
	\left( 
	u_d(0,X^{d,(\theta,0,-m)}_{0,t,x}) 
	+ 
	t \, f(0)
	\right)
	\right] 
	\\ 
	& = \sum_{k=1}^{n-1} \frac{t}{M^{n-k}} 
	\Bigg[ 
	\sum_{m=1}^{M^{n-k}} 
	\bigg(
	F_{\varrho_M}\big( 
	V^{d,(\theta,k,m)}_{k,M,\varrho_M}(R^{(\theta,k,m)}_{t},
	X^{d,(\theta,k,m)}_{R^{(\theta,k,m)}_{t},t,x}
	)
	\big)
	\\ \nonumber
	& %\qquad\qquad\qquad\qquad
	-
	F_{\varrho_M}\Big( 
	V^{d,(\theta,-k,m)}_{k-1,M,\varrho_M}\big( R^{(\theta,k,m)}_{t}, 
	X^{d,(\theta,k,m)}_{R^{(\theta,k,m)}_{t},t,x}\big)
	\Big)
	\bigg)
	\Bigg]
	+ 
	\frac{1}{M^n}\!\left[ 
	\sum_{m=1}^{M^n} 
	\left( 
	u_d(0,X^{d,(\theta,0,-m)}_{0,t,x}) 
	+ 
	t \, f(0)
	\right)
	\right] \\ \nonumber
	& = V^{d,\theta}_{n,M,\varrho_M}(t,x). 
	\end{align}
Hence, we obtain that for every 
	$n\in A$ 
it holds that 
	$n+1\in A$. 
Combining this with the fact that 
	$1\in A$ 
and induction ensures that 
	$A=\N$. 
This yields that for every 
	$d,M\in\N$, 
	$\theta\in\Theta$,
	$n\in\N_0$, 
	$t\in [0,T]$, 
	$x\in \R^d$ 
it holds that 
	\begin{equation} 
	V^{d,\theta}_{n,M,\varrho_M}(t,x) 
	= 
	U^{d,\theta}_{n,M}(t,x).
	\end{equation} 
The fact that for all 
	$r\in (0,\infty)$, 
	$w,\mathfrak{w}\in [-r,r]$ 
it holds that $|f(w)-f(\mathfrak{w})|\leq L(r)|w-\mathfrak{w}|$, 
\eqref{cor_polynomially_growing_growth_of_varrho_n}, 
and 
\cref{cor_lip} (with 
	$\rho=\exp(T\sup_{v\in\R} (\frac{vf(v)}{1+v^2}))[1+\sup_{d\in\N}\sup_{x\in\R^d} |u_d(0,x)|^2]^{\nicefrac12}$, 
	$c=\sup_{v\in\R} (\frac{vf(v)}{1+v^2})$, 
	$T=T$,
	$K=K$,
	$\Theta=\Theta$, 
	$f=f$, 
	$\fr=\fr$, 
	$L=L$, 
	$\norm{\cdot}_d=\norm{\cdot}_d$, 
	$\varrho=\varrho$, 
	$u_d=u_d$, 
	$(\Omega,\cF,\P)=(\Omega,\cF,\P)$, 
	$\mathcal{R}^{\Theta}=\mathcal{R}^{\theta}$, 
	$W^{d,\theta}=W^{d,\theta}$, 
	$X^{d,\theta}_{t,s,x}=X^{d,\theta}_{t,s,x}$,
	$U^{d,\theta}_{n,M,r}(t,x)=V^{d,\theta}_{n,M,r}(t,x)$, 
	$\mathfrak{C}_{d,n,M}=\mathfrak{C}_{d,n,M}$ 
for 
	$d,M\in\N$, 
	$\theta\in\Theta$, 
	$n\in\N_0$, 
	$r\in (0,\infty)$, 
	$t\in [0,T]$, 
	$x\in \R^d$	
in the notation of \cref{cor_lip}) therefore guarantee that there exist 
	$\mathfrak{N}\colon (0,1] \to \N$ 
and 
	$\mathfrak{c}\colon (0,\infty) \to [0,\infty)$ 
such that for every 
	$d\in\N$,
	$\delta\in (0,\infty)$, 
	$\varepsilon \in (0,1]$ 
it holds that 
	$
	\sum_{n = 1}^{\mathfrak{N}_{\varepsilon}+K} \mathfrak{C}_{d,n,n} 
	\leq 
	\mathfrak{c}_{\delta} d \varepsilon^{-(2+2\delta)} 
	$
and 
	\begin{equation}
	\begin{split}
	& 
	\sup_{n\in [\mathfrak{N}_{\varepsilon},\infty) \cap\N} 
	\left[ 
	\sup_{x\in\R^d}
	\left( 
	\Exp{|U^{d,0}_{n,n}(T,x) - u_d(T,x)|^2}
	\right)^{\!\nicefrac12}
	\right] 
	\\
	& 
	= \sup_{n\in [\mathfrak{N}_{\varepsilon},\infty) \cap\N} 
	\left[ 
	\sup_{x\in\R^d}
	\left( 
	\Exp{|V^{d,0}_{n,n,\varrho_n}(T,x) - u_d(T,x)|^2}
	\right)^{\!\nicefrac12}
	\right]  
	\leq \varepsilon .
	\end{split}
	\end{equation} 
The proof of \cref{cor:polynomially_growing} is thus completed.
\end{proof}

%Note that for every 	
%	$d,M\in\N$,
%	$n\in\N_0$, 
%	$x\in\R^d$ 
%the number
%	$\mathfrak{C}_{d,n,M}\in\N_0$ 
%in \eqref{cor_polynomially_growing:comp_cost} provides an upper bound for the number of realizations of scalar (standard normal or uniformly distributed) random variables needed to compute one realization of 
%	$U^{d,0}_{n,M}(T,x)\colon\Omega\to\R$ 
%in \eqref{polynomially_growing:mlp_scheme}. Last but not least we consider in \cref{polynomkorollar} below the prototype of an Allen--Cahn PDE. 

\begin{cor}\label{polynomkorollar}
Let 
	$c,T\in (0,\infty)$, 
	$K\in\N_0$,
	$\Theta = \cup_{n\in\N} \Z^n$, 
for every 
	$d\in\N$ 
let 
	$\norm{\cdot}_d\colon\R^d\to[0,\infty)$ 
be a norm on $\R^d$, 
let 
	$u_d\in C([0,T]\times\R^d,\R)$, $d\in\N$,
satisfy for every 
	$d\in\N$,
	$t\in [0,T]$,
	$x\in\R^d$ 
that 
	$|u_d(0,x)| \leq c$,  
	$\inf_{a\in\R}[\sup_{s\in [0,T]}\sup_{y=(y_1,\ldots,y_d)\in\R^d}(e^{a(|y_1|^2+\ldots+|y_d|^2)} |u_d(s,y)|)] < \infty$,
	$u_d|_{(0,T]\times\R^d}\in C^{1,2}((0,T]\times\R^d,\R)$, 
and 
	\begin{equation}
	%	\label{thm:main_theorem_assumptions_on_u}
	(\tfrac{\partial}{\partial t}u_d)(t,x) 
	= 
	(\Delta_x u_d)(t,x) 
	+ 
	u_d(t,x)-(u_d(t,x))^{3},
	\end{equation} 
let 
	$\varrho \colon \N \to (0,\infty)$ 
be a function which satisfies that 
	$\limsup_{n\to\infty} (\frac{\varrho_n}{\ln(\ln(n))}) < \infty =\liminf_{n\to\infty} \varrho_n$, 	
let 
	$f_n\colon \R\to\R$, $n\in\N$, 
be the functions which satisfy for every 
	$n\in\N$, 
	$v\in\R$
that 
	$f_n(v) = (\min\{\varrho_n,\max\{-\varrho_n,v\}\})-(\min\{\varrho_n,\max\{-\varrho_n,v\}\})^3$, 
let 
	$(\Omega,\mathcal{F},\P)$ 
be a probability space, 
let
	$\mathcal{R}^{\theta}\colon \Omega \to [0,1]$, $\theta\in\Theta$, 
be independent $\mathcal{U}_{[0,1]}$-distributed random variables, 
let 
	$W^{d,\theta}\colon [0,T]\times\Omega\to\R^d$, 
	$d\in\N$,
	$\theta\in\Theta$,
be independent standard Brownian motions, 
assume that 
	$(\mathcal{R}^{\theta})_{\theta\in\Theta}$ 
and 
	$(W^{d,\theta})_{(d,\theta)\in\N\times\Theta}$ 
are independent, 
let 
	$R^{\theta} \colon \Omega \times [0,T] \to [0,T]$, $\theta\in\Theta$, 
satisfy for every 
	$\theta\in\Theta$,
	$t\in [0,T]$ 
that 
	$R^{\theta}_t = t \mathcal{R}^{\theta}$,
for every 
	$d\in \N$,
	$s\in [0,T]$, 
	$t\in [s,T]$,
	$x\in \R^d$, 
	$\theta\in \Theta$ 
let 
	$
	X^{d,\theta}_{s,t}\colon\Omega\to\R^d
	$
satisfy 
	$
	X^{d,\theta}_{s,t,x} 
	= 
	x + \sqrt{2}\,(W^{d,\theta}_t - W^{d,\theta}_s)
	$, 
let 
	$
	U^{d,\theta}_{n,M}
	\colon [0,T]\times\R^d\times\Omega\to\R$, 
	$d,M\in\N$,
	$\theta\in\Theta$, 
	$n\in\N_0$, 
satisfy for every 
	$d,n,M\in\N$,
	$\theta\in\Theta$,
	$r\in (0,\infty)$,
	$t\in [0,T]$, 
	$x\in \R^d$ 
that
	$U^{d,\theta}_{0,M}(t,x) = 0$ 
and 
	\begin{equation} 
	\begin{split} 
	U^{d,\theta}_{n,M}(t,x) 
	&=  
	\sum_{k=1}^{n-1} \frac{t}{M^{n-k}} 
	\Bigg[ 
	\sum_{m=1}^{M^{n-k}} 
	\bigg(
	f_M\Big(U^{d,(\theta,k,m)}_{k,M,r}\big( 		
	R^{(\theta,k,m)}_{t}, 
	X^{d,(\theta,k,m)}_{R^{(\theta,k,m)}_{t},t,x}
	\big)\Big) 
	\\
	& 
	- f_M\Big( 
	U^{d,(\theta,-k,m)}_{k-1,M}\big( 
	R^{(\theta,k,m)}_t, X^{d,(\theta,k,m)}_{R^{(\theta,k,m)},t,x}
	\big) 
	\Big) 
	\bigg)
	\Bigg]+
	\frac{1}{M^n}\left[ 
	\sum_{m=1}^{M^n} 
	u_d(0,X^{d,(\theta,0,-m)}_{0,t,x}) 
	\right] 
	\!, 
	\end{split} 
	\end{equation} 
and let 
	$\mathfrak{C}_{d,n,M}\in\N_0$, $d,M\in\N$, $n\in\N_0$, 
satisfy for every 
	$d,n,M\in\N$ 
that 
	$\mathfrak{C}_{d,0,M} = 0$ 
and 
	\begin{equation} 
	\mathfrak{C}_{d,n,M} 
	\leq 
	(2d+1)M^n 
	+ 
	\sum_{l=1}^{n-1} M^{n-l} \,(d+1+\mathfrak{C}_{d,l,M} 
	+ 
	\mathfrak{C}_{d,l-1,M}). 
	\end{equation} 
Then there exist 
	$\mathfrak{N}\colon (0,1] \to \N$ 
and 
	$\mathfrak{c}\colon (0,\infty) \to [0,\infty)$
such that for every 
	$d\in\N$, 
	$\delta\in (0,\infty)$,
	$\varepsilon \in (0,1)$ 
it holds that 
\begin{equation}
	\begin{split}
	&
	\left[
	\smallsum\limits_{n = 1}^{\mathfrak{N}_{\varepsilon}+K} 
	\mathfrak{C}_{d,n,n}
	\right]
	\leq d \mathfrak{c}_{\delta} \varepsilon^{-(2+\delta)} 
	\quad\text{and}\quad
	\sup_{n\in [\mathfrak{N}_{\varepsilon},\infty)\cap\N}
	\left[
	\sup_{x \in\R^d}
	\left(
	\Exp{|U^{d,0}_{n,n}(T,x) - u_d(T,x)|^2} 
	\right)^{\!\nicefrac12}
	\right]
	\leq \varepsilon.
	\end{split}
	\end{equation}
\end{cor}

\begin{proof}[Proof of \cref{polynomkorollar}] 
First, observe that for every 
	$r \in (0,\infty)$, 
	$v,w \in [-r,r]$ 
it holds that 
	\begin{equation} \label{polynomkorolar:lipschitz}
	\begin{split}
	|(v-v^{3})-(w-w^{3})| 
	& = 
	|(v-w)(1-w^2-wv-v^2)| 
	\leq 
	( 1 + |v|^2 + |w|^2 + |wv|)
	|v-w| 
	\\
	& \leq 
	2 ( 1 + |v|^2 + |w|^2) |v-w|. 
	\end{split}
	\end{equation} 
Moreover, note that for every 
	$v\in\R$ 
it holds that 
	$v(v-v^3) = v^2 - v^4 \leq 1+v^2$. 
This, the hypothesis that for every 
	$d\in\N$, 
	$x\in\R^d$ 
it holds that 
	$|u_d(0,x)|\leq c$, 
the fact that for every 
	$d\in\N$ 
it holds that 
	$\norm{\cdot}_d$ 
and the Euclidean norm on $\R^d$ are equivalent, 
\eqref{polynomkorolar:lipschitz}, and \cref{cor:polynomially_growing} (with 
	$T=T$, 
	$c=\max\{c,2\}$, 
	$K=K$, 
	$\Theta=\Theta$, 
	$f=(\R\ni u\mapsto u-u^{3}\in\R)$, 
	$\mathbf{f}_M=f_M$, 
	$\varrho=\varrho$, 
	$u_d=u_d$, 
	$(\Omega,\cF,\P)=(\Omega,\cF,\P)$, 
	$\mathcal{R}^{\theta}=\mathcal{R}^{\theta}$, 
	$W^{d,\theta}=W^{d,\theta}$, 
	$R^{\theta}=R^{\theta}$, 
	$X^{d,\theta}_{s,t,x}=X^{d,\theta}_{s,t,x}$, 
	$U^{d,\theta}_{n,M}(t,x)=U^{d,\theta}_{n,M}(t,x)$, 
	$\mathfrak{C}_{d,n,M}=\mathfrak{C}_{d,n,M}$ 
for 
	$d,M\in\N$, 
	$\theta\in\Theta$, 
	$n\in\N_0$, 
	$r\in (0,\infty)$, 
	$t\in [0,T]$, 
	$s\in [0,t]$, 
	$x\in \R^d$ 
in the notation of \cref{cor_lip})
ensure that there exist 
	$\mathfrak{N} \colon (0,1] \to \N$ 
and 
	$\mathfrak{c} \colon (0,\infty) \to [0,\infty)$ 
such that for every 
	$d\in\N$, 
	$\delta\in (0,\infty)$, 
	$\varepsilon \in (0,1]$ 
it holds that 
	\begin{equation}
	\left[ 
	\smallsum\limits_{n=1}^{\mathfrak{N}_{\varepsilon}+K}
	\mathfrak{C}_{d,n,n} 
	\right]
	\leq \mathfrak{c}_{\delta} d \varepsilon^{-(2+2\delta)}
	\quad\text{and}\quad
	\sup_{n\in [\mathfrak{N}_{\varepsilon},\infty)\cap\N} 
	\left[ 
	\sup_{x \in\R^d} 
	\left(  
	\Exp{|V^{d,0}_{n,n}(T,x)-v_d(T,x)|^2}
	\right)^{\!\nicefrac12}
	\right] \leq \varepsilon.
	\end{equation} 
The proof of \cref{polynomkorollar} is thus completed.
\end{proof}

\section*{Acknowledgements}

This project has been partially supported by the Deutsche Forschungsgemeinschaft (DFG) via research grant HU 1889/6-1. 
F.\,H.\,gratefully acknowledges a Research Travel Grant by the Karlsruhe House of Young Scientists (KHYS) supporting his stay at ETH Zurich.

\bibliographystyle{acm}
\bibliography{References}

\end{document}